\documentclass[a4paper,12pt,reqno,final]{article}
\usepackage[margin=25.4mm]{geometry}
\usepackage[utf8]{inputenc}
\usepackage{mathrsfs,eucal}

\usepackage{epsfig, amssymb, amsmath,amsfonts,bef_alex,amssymb}
\renewcommand{\ti}{{\times}}

\numberwithin{equation}{section}
\usepackage{hyperref}

\usepackage[dvipsnames]{xcolor}
\usepackage{eucal,paralist,enumerate}
\usepackage[textsize=footnotesize]{todonotes}

\usepackage[notref,notcite,color]{showkeys}
\renewcommand{\showkeyslabelformat}[1]%
   {{\!\!{\color{blue}\tiny\sffamily#1}\!\!}}

\newcommand{\INFCONV}{\overset{\text{inf}}{\circ}}
\newcommand{\eff}{\mathrm{eff}}
\makeatletter
\renewcommand*\env@cases[1][1.2]{%
  \let\@ifnextchar\new@ifnextchar
  \left\lbrace
  \def\arraystretch{#1}%
  \array{@{}c@{\quad}l@{}}%
}
\makeatother

\newtheorem{hypothesis}[theorem]{Hypothesis}

\newcommand{\cR}{\mathcal{R}}
\newcommand{\cE}{\mathcal{E}}

\newcommand{\dom}{\mathrm{dom}}

\newcommand{\up}{\uparrow}
\newcommand{\weakto}{\rightharpoonup}
\newcommand{\mfE}{\mathfrak{E}}
\newcommand{\AC}{\mathrm{AC}}

\newcommand{\gen}[2]{\gE(#1,#2)}
\newcommand{\en}[2]{\calE(#1,#2)}
\newcommand{\domE}{\mathrm{D}}
\newcommand{\pet}[2]{\partial_t\calE(#1,#2)}
\newcommand{\frsub}[2]{\partial \calE(#1,#2)}
\newcommand{\frsubopt}[3]{\partial^{\Spx {#1}} \calE(#2,#3)}

\newcommand{\frnameopt}[1]{\partial^{\Spx {#1}} \calE}
\newcommand{\pairing}[4]{ \sideset{_{ #1 }}{_{ #2 }}  {\mathop{\langle #3 , #4
\rangle}}}
\newcommand{\foraa}{\text{for a.a.}}
\newcommand{\tdis}[1]{\widetilde{\calR}_{#1}}
\newcommand{\subl}[1]{S_{#1}}

\newcommand{\gX}{\mathscr{X}}
\newcommand{\gE}{\mathscr{E}}
\newcommand{\gR}{\mathscr{R}}
\newcommand{\gD}{\mathscr{D}}

\newcommand{\pws}[2]{{#1}_{\kern-1pt#2}}
\newcommand{\pwM}[2]{\widetilde{#1}_{\kern-1pt#2}}
\newcommand{\piecewiseLinear}[2]{{\widehat{#1}_{\kern-1pt#2}}}
\newcommand{\pwl}{\piecewiseLinear}
\newcommand{\dis}[3]{{#1}_{#2}^{#3}}
\newcommand{\ttau}{|\bftau|}
\newcommand{\upiecewiseConstant}[2]{\underline{#1}_{\kern-1pt#2}}
\newcommand{\upwc}{\upiecewiseConstant}
\newcommand{\piecewiseConstant}[2]{\overline{#1}_{\kern-1pt#2}}
\newcommand{\pwc}{\piecewiseConstant}
\def\Xint#1{\mathchoice
{\XXint\displaystyle\textstyle{#1}}%
{\XXint\textstyle\scriptstyle{#1}}%
{\XXint\scriptstyle\scriptscriptstyle{#1}}%
{\XXint\scriptscriptstyle\scriptscriptstyle{#1}}%
\!\int}
\def\XXint#1#2#3{{\setbox0=\hbox{$#1{#2#3}{\int}$ }
\vcenter{\hbox{$#2#3$ }}\kern-.6\wd0}}

\def\dashint{\Xint-}

\newcommand{\sft}{\mathsf{t}}

\newcommand{\xforce}[2]{\TT^{(#2)}_{#1} \pws \xi{#1}}
\newcommand{\xforcel}{\bar{\bar{\xi}}} 

\newcommand{\xcur}[2]{\TT^{(#2)}_{#1} \pws U{#1}}

\newcommand{\mt}[2]{\mathsf{k}_{#1}(#2)}

\newcommand{\lint}[2]{\mathrm{I}_{\mathrm{left}}^{#1,#2}}
\newcommand{\rint}[2]{\mathrm{I}_{\mathrm{right}}^{#1,#2}}

\newcommand{\Argmin}{\mathop{\mathrm{Argmin}}}

\newcommand{\blockspace}{\mathbf{U}}
\newcommand{\Spx}[1]{\mathbf{X}_{#1}}
\newcommand{\Spxw}[1]{\mathbf{X}_{#1, \mathrm{w}}}
\newcommand{\Spy}{\mathbf{Y}}

\newcommand{\Spz}{\mathbf{Z}}

\newcommand{\Spgx}{\gX}
\newcommand{\Spgxw}{\Spgx_{\mathrm{w}}}

\newcommand{\Spxname}{\mathbf{X}}
\newcommand{\gnorm}[1]{\|#1\|}
\newcommand{\gnorms}[1]{\|#1\|_*}
\newcommand{\norm}[2]{\| #1\|_{#2}}
\newcommand{\norms}[2]{\| #1\|_{#2,*}}
\newcommand{\Rtrial}{\mathfrak{R}_*}
\newcommand{\TT}{\mathbb{T}}

\newcommand{\unipsi}{\Psi}
\newcommand{\recnote}[1]{{\footnotesize (#1)}}
\newcommand{\QYE}{\mathrm{QYE}}
\newcommand{\Vopt}[1]{V_{#1}}

\newcommand{\disblo}[1]{\calR_{\mathrm{#1}}}

\newcommand{\yvar}{y}
\newcommand{\zvar}{z}
\newcommand{\yvel}{v}
\newcommand{\zvel}{w}
\newcommand{\yfor}{\eta}
\newcommand{\zfor}{\zeta}
\newcommand{\thalf}{\dis t \bftau{k{-}1/2}}

\makeatletter
\renewcommand*\env@cases[1][1.2]{%
   \let\@ifnextchar\new@ifnextchar
   \left\lbrace
   \def\arraystretch{#1}%
   \array{@{\,}c@{\ }l@{}}%
}
\makeatother

\newcommand{\blosub}[1]{\partial_{\mathrm{#1}}\calE}
\newcommand{\Ls}[2]{\mathrm{Ls}_{{#1}}^{\mathrm{weak}}\big(#2\big)}
\newcommand{\STEP}[1]{\noindent\underline{\emph{Step #1}}}

\begin{document}
\title{On time-splitting methods for gradient flows\\
        with two dissipation mechanisms\thanks{The research of AM was partially
          supported by Deutsche Forschungsgemeinschaft via SPP\,2256,
          subproject Mi\,459/9-1 (project no.\,441470105). RR has been partially supported by GNAMPA and by the 
          MIUR - PRIN project 2017TEXA3H ``Gradient flows, Optimal Transport and Metric Measure Structures".}}
\author{%
 Alexander Mielke\thanks{WIAS Berlin and Humboldt-Universit\"at zu Berlin, Germany, alexander.mielke@wias-berlin.de}, 
 Riccarda Rossi\thanks{DIMI, Universit\`a degli studi di Brescia, Italy, riccarda.rossi@unibs.it},\ and
 Artur Stephan\thanks{WIAS Berlin, Germany, artur.stephan@wias-berlin.de}}

\date{July 21, 2023}

\maketitle

\begin{abstract}
  We consider generalized gradient systems in Banach spaces whose evolutions
   are generated by the interplay between  an  energy functional
  and a dissipation potential. We  focus on  the case in which the dual
  dissipation potential is given by a sum of two functionals  and show
  that solutions of the associated gradient-flow evolution equation with
  combined dissipation can be constructed by a split-step method, i.e. by
  solving alternately the gradient systems featuring  only one of the
  dissipation potentials and concatenating the  corresponding 
  trajectories. Thereby the construction of solutions is provided either by
  semiflows, on the time-continuous level, or by using Alternating Minimizing
  Movements in the time-discrete setting. In both cases the convergence
  analysis relies on the energy-dissipation principle for gradient systems.
\end{abstract}

{\small\tableofcontents}


\section{Introduction}
\label{se:Intro}

This paper revolves around the application of time-splitting methods to
dissipative evolutionary processes that are generated by a generalized gradient
system $(\gX, \gE, \gR)$, which is a triple such that
\begin{enumerate}
\item the ambient space $(\gX,\gnorm{\cdot})$ is a (separable) reflexive Banach
  space;
\item the energy is a lower semicontinuous, time-dependent functional
  $\gE: [0,T]\ti \gX \to (-\infty,\infty]$, bounded from below  and has the proper
  domain   $[0,T] \ti \gD$; 
\item and the dissipation mechanisms are encoded by a convex lower
  semicontinuous dissipation potential $\gR: \gX \to [0,\infty)$. 
\end{enumerate}
In what follows, we will confine the discussion to the case in which both $\gR$
and its convex conjugate $\gR^*:\gX^* \to [0,\infty)$,
$\xi \mapsto \sup_{v\in \gX} (\langle \xi, v \rangle {-} \gR(v))$ have
superlinear growth at infinity.
We will refer to the triple $(\gX, \gE,
\gR)$ as a \emph{generalized} gradient system, because for true gradient
systems the dissipation potential $\gR:\gX\to
[0,\infty)$ has to be quadratic, leading to a Hilbert-space structure, see
\cite{Miel23IAGS}. Whenever convenient, we will alternatively write $(\gX, \gE,
\gR^*)$.  Throughout the paper, the dissipation potential
$\gR$ will be assumed \emph{state-independent}.

  Typically, in this class there fall processes whose evolution results from a
  balance between the two competing mechanism of  decrease of the energy
  $\gE$ and dissipation of energy according to
  $\gR$.  Thus, they are governed by the subdifferential inclusion
\begin{equation}
\label{genDNE}
\partial \gR(u'(t)) + \partial \gE(t,u(t)) \ni 0 \quad 
\text{in } \gX^* \ \ \foraa\  t \in (0,T).
\end{equation}
Indeed, \eqref{genDNE} is a balance law between frictional forces in the
(convex analysis) subdifferential $\partial\gR: \gX \rightrightarrows
\gX^*$ given by
\[
\partial\gR(v)= \{ \omega \in \gX^*\, : \ \gR(\hat{v})-\gR(v) \geq
\pairing{}{}{\omega}{\hat{v}{-}v} \quad \text{for all } \hat v \in \gX\} 
\]
(where $\langle \cdot, \cdot
\rangle$ denotes the duality pairing between $\gX^*$ and
$\gX$), and potential restoring forces in the Fr\'echet subdifferential
$\partial \gE: [0,T]\times \gX \rightrightarrows \gX^*$ of
$\gE$ with respect to its second variable: at a given $(t,u)\in [0,T]\ti
\gD$ we define $ \partial\gE(t,u)$ as the set of all $\xi \in \gX^*$ satisfying
\begin{equation}
\label{gdef-subfiff}
 \gE(t,w)-\gE(t,u) \geq \pairing{}{}{\xi}{w{-}u} {} + \mathrm{o}(\|w{-}u\|) 
  \quad \text{as } \gnorm{w{-}u}\to 0\,.
\end{equation}
Therefore, \eqref{genDNE} can be recast as
\begin{subequations}
\label{primal-dual}
\begin{equation}
\label{primal}
 -\xi(t) \in \partial \gR(u'(t)) \ \text{ in } \gX^* \quad\text{and} \quad   \xi(t) \in  \partial \gE(t,u(t))  \qquad  \foraa\  t \in (0,T). 
\end{equation}
From the thermodynamical point of view the
\emph{primal} dissipation potential $\gR$ has a prominent role in defining the
\emph{kinetic relation} $\eta \in \pl\gR(v)$ between the rate $v\in \gX$ and
friction force $\eta\in \gX^*$. By the Fenchel equivalence in convex analysis
the kinetic relation can be inverted as $ v \in \pl\gR^*(\eta)$, such that  it is 
also meaningful to reformulate \eqref{gdef-subfiff}  in the rate form 
\begin{equation}
\label{dual}
u'(t)  \in \partial \gR^*({-}\xi(t)) \  \text{ in } \gX \quad  \text{and} \quad
\xi(t) \in  \partial \gE(t,u(t))  \qquad  \foraa\  t \in (0,T),
\end{equation} 
\end{subequations}
which casts the \emph{dual} dissipation potential under the spotlight.  The
first existence results for evolution equations \eqref{primal},\eqref{dual}, in
the Hilbert space setting and for a quadratic dissipation, date back to the
late '60s \cite{Komura67, Crandall-Pazy69}; in particular, we refer to the
monograph \cite{Brez73OMMS}. Existence in general \emph{doubly nonlinear} case
has been first systematically tackled in  the seminal papers
\cite{ColliVisintin90, Colli92}. In the last three decades, the existence
theory has been extended to encompass \emph{nonsmooth} and \emph{nonconvex}
driving energies \cite{RossiSavare06, MRS2013}, based on the \emph{variational
  theory} for the analysis of gradient flows in metric spaces
\cite{Ambrosio95,AGS08, RMS08}.

Both dissipation potentials $\gR$ and $\gR^*$ feature in the
\emph{energy-dissipation balance}
 \begin{equation}
 \label{EDB-intro}
 \gen t{u(t)} + \int_s^t \left\{  \gR (u'(r)){+} \gR^*({-}\xi(r)) \right\}   \dd r  = \gen s{u(s)} +\int_s^t \partial_t \gen r{u(r)} \dd r  
\end{equation}
for all $0\leq s \leq t\leq
T$, which is equivalent to the primal and dual formulations under the validity
of a suitable chain-rule property for the gradient system
$(\gX,\gE,\gR)$. As a matter of fact, \eqref{EDB-intro} lies at the core
of our variational approach to gradient systems. Indeed, it turns out that a
pair $(u,\xi)$ fulfills \eqref{EDB-intro} (and thus \eqref{primal-dual}) if and only
if it satisfies the upper inequality
$\leq$, and this has paved the way for the usage of the toolbox from Calculus
of Variations in order to prove existence results for \eqref{primal-dual},
see  \cite{AGS08, RMS08, MRS2013}  and the references in the survey
\cite{Miel23IAGS}.

\subsection*{The time-splitting approach}
In this paper we focus on  the case in which the dual  potential $\gR^*$ is given
by the sum
\begin{equation}
\label{additive-struct}
\gR^* = \calR_1^* + \calR_2^* \text{ for two dissipation potentials } 
\calR_j: \Spx j \to [0,\infty),
\end{equation}
with $\Spx j$ (separable) reflexive spaces, `ordered' in such a way that
$\Spx 2 \subset \Spx 1$ continuously.   The energy functional
$\calE: [0,T]\times \Spx 2 \to (-\infty,\infty]$ is defined on the smaller
space and is extended to the larger space by $+\infty$.  Clearly, the
quintuple $(\Spx 1, \Spx 2, \calE, \calR_1, \calR_2)$ gives rise both to the
individual gradient systems $(\Spx 1, \calE, \calR_1)$ and
$(\Spx 2, \calE, \calR_2)$, and to the system
$(\Spx 1, \calE, \calR_1^*{+} \calR_2^*)$, whose dissipation mechanisms are
encoded by a `combination' of $\calR_1$ and $\calR_2$ and which is thus
governed by the subdifferential inclusion
\begin{equation}
  \label{dual-formulation-SPLIT-intro}
  u'(t)  \in \partial (\calR_1^*{+}\calR_2^*)({-}\xi(t)) \  \text{ in } \Spx 1 
  \quad  \text{and}\quad  \xi(t) \in  \partial \calE(t,u(t))  
  \qquad  \foraa\  t \in (0,T). 
\end{equation}
In order to construct solutions to \eqref{dual-formulation-SPLIT-intro}, it may
then be convenient  to use a  time-splitting approach, capable of
handling the different properties of the potentials $\calR_1$ and $\calR_2$.
Indeed, split-step methods with time step $\tau = T/N$ and $N\gg 1$  amount to
\begin{itemize} 
\item[(i)] solving on the semi-intervals of length $\tfrac12 \tau$,
  alternately, the single-dissipation gradient systems
  $(\Spx j,\calE,2\calR_j^*)$, $j\in \{1,2\}$ (i.e., with rescaled potentials
  $ \widetilde{\calR}_j (\cdot)= 2 \calR_j(\tfrac{1}2 \cdot) $);
\item[ (ii)] 
concatenating the solutions to obtain a trajectory $U_\tau:[0,T]\to \Spx 1$.
\end{itemize}
Then, and this is the main objective of the paper, the task in the convergence
analysis is to show that in the limit $N\to\infty$ the sequence of trajectories
$(U_\tau)_\tau$ converge to the solution $U$ of an \emph{effective
  gradient-flow equation}
\begin{equation}
\label{effective-dne}
  \pl\calR_\eff(U'(t)) + \pl\calE(t,U(t)) \ni 0 \qquad \foraa\  t \in (0,T). 
\end{equation}

For instance, in the case of reaction-diffusion systems, with this approach the
approximate solutions can be constructed by concatenating a solution obtained
in the diffusion step, by a method tailored to the linear parabolic structure,
and a solution arising from  a pure reaction step, which exploits the
distinct features of the nonlinear ODE.  Split-step methods for evolution
equations, where the right-hand side is given by a sum of two parts as in
\eqref{dual-formulation-SPLIT-intro}, have a long history starting with the
works of Lie and Trotter for linear evolution equations. Later
\cite{KatMas78TPFNS} provided a generalization where the evolution is given by
subdifferentials of convex functions in the spirit of Br\'ezis
\cite{Brez73OMMS}. In \cite{CleMaa11TPFG} convergence of split-step methods for
gradient flows in metric spaces as in \cite{AGS08} has been shown, in the case
the driving energy consists of two contributions with different properties.

In contrast, the functional setup considered in this paper is significantly
different from that usually addressed for time-splitting methods applied to
gradient systems,  because  we tackle the situation in which the
\emph{dissipation} consists of two parts.  Hence, we have  to account
for the two different geometries in the underlying space.

\subsection*{Our analysis}
In the time-splitting approach to the generalized gradient system
$(\Spx 1, \calE, \calR_1^*{+} \calR_2^*)$, on the one hand the
`concatenation step' may take place on the time-continuous level by
concatenating alternating solutions to the individual subdifferential
inclusions for $(\Spx j,\calE,2\calR_j^*)$, on intervals of length
$\frac12\tau$ to finally fill the whole interval $[0,T]$.  On the other hand,
one can work on the time-discrete level by using an \emph{Alternating
Minimizing Movement Scheme}.  We briefly illustrate  the latter 
procedure in the following lines and postpone a detailed analysis to Section\
\ref{s:AltMinMov}, where we prove our convergence result in Theorem
\ref{th:conv-split-step-MM}.  The time-splitting scheme on the time-continuous
level will be  carefully studied in Section\ \ref{se:TimeSplit},  see Theorem
\ref{th:conv-split-step}. 

Indeed, a commonly used procedure for constructing solutions to the
subdifferential inclusions for the single-dissipation gradient systems
$(\Spx j,\calE,\calR_j)$ is via \emph{Minimizing Movements}.  Adopting this
method in the context of the time-splitting approach means that for each
$j=1,2$ we solve the time-incremental minimization problems involving the
potentials $ \widetilde{\calR}_j (\cdot)= 2 \calR_j(\tfrac{1}2\,\cdot\,)$,
whose rescaling corresponds to the halved length $\tfrac12 \tau$ of the
discrete intervals. Approximate solutions are then defined by piecing together
these discrete solutions.

For illustrative reasons we present in the Introduction the time-splitting
approach for a uniform partition by intervals of equal length. The rigorous
analysis in the main text  allows for general partitions.
Hence, let $\tau=T/N$ define a uniform partition of the interval $[0,T]$ in
sub-intervals $((k{-}1)\tau, k\tau)$, $k\in\{1,\dots,N\}$, with midpoint
$(k{-}1/2)\tau$.  Starting from an initial datum $u_0 \in \mathrm{dom}(\calE)$,
we define the \emph{piecewise constant} time-discrete solutions
$\pwc U{\tau} :[0,T] \to \mathrm{dom}(\calE) \subset \Spx1$ via
$\pwc U{\tau}(0): = u_0=:U^2_0$ and, for $k=1,\ldots,N$, we set
\begin{align}
& 
\label{min-schemes-i-intro}
\pwc U{\tau}(t) : = U_k^1 \ \text{ for } t \in  ((k{-}1)\tau, (k{-}1/2)\tau], \quad  \pwc U{\tau}(t) : = U_k^2 \ \text{ for } t \in
((k{-}1/2)\tau, k\tau]  ,
\\ \nonumber
&\text{where } U_k^1  \in \Argmin_{U \in  \Spx 1} \left\{ \tfrac{\tau}2\, 
  \wt{\calR}_1 \big(  \tfrac2{\tau}
  ( U{-}U^2_{k-1} ) \big)  {+} \en {(k{-}1/2)\tau}U  \right\}, 
\\ \nonumber
&  \ \text{ \ and }U_k^2  \in \Argmin_{U \in  \Spx 2} \left\{ \tfrac{\tau}2 
  \,\wt{\calR}_2 \big(\tfrac2{\tau} ( U{-}  U^1_k ) \big) 
  {+} \en {k\tau}U \right\}\,,
\end{align}
cf.\ also \eqref{min-schemes-i} ahead. Further, we introduce the
piecewise linear function $\pwl U{\tau} :[0,T] \to \Spx1$
so obtained by affinely interpolating the values $\pwc U{\tau}(t)$ and
$\pwc U{\tau}(t{-}\tau/2)$ for $t$ in the  semi-intervals  $ [(k{-}1)\tau, (k{-}1/2)\tau] $ and
$  [(k{-}1/2)\tau, k\tau] $. We also consider the piecewise constant
interpolant $\pwc \xi \tau$ of the discrete forces
$\xi_k^1\in \partial^{\Spx 1} \calE((k{-}1/2)\tau, U_k^1)$ and
$\xi_k^2\in \partial^{\Spx 2} \calE(k\tau, U_k^2)$ (with
$\partial^{\Spx j}  \calE: \Spx j\rightrightarrows \Spx j^* $ the Fr\'echet
subdifferentials of $\calE(t,\cdot)$ with respect to the
$\pairing{}{\Spx j}{\cdot}{\cdot}$ pairings), that feature in the
Euler-Lagrange equations for the minimum problems \eqref{min-schemes-i-intro}.

Assume now that
$\calE: [0,T]\times \Spx 2\to (-\infty,\infty]$ is $\lambda$-convex in its
second variable, uniformly in $t\in [0,T]$, with respect to the coarser norm
$\|\cdot \|_{\Spx 1}, $ namely
 \[
 \begin{aligned}
 &
 \exists\, \lambda \in \R \ \forall\, t \in [0,T]  \ \forall\, u_0,u_1 \in \mathrm{dom}(\calE)  \ \forall\, \theta \in [0,1]\, : 
 \\
  & \qquad \calE(t,(1{-}\theta)u_0{+}\theta u_1) \leq (1{-}\theta) \calE(t,u_0)+ \theta \calE(t,u_1)
   - \frac\lambda 2 \theta(1{-}\theta) \|u_0{-}u_1\|_{\Spx 1}^2\,.
 \end{aligned}
 \]
Then, the functions $(\pwc U\tau, \pwl U\tau, \pwc \xi \tau)$  satisfy a discrete version of the energy-dissipation  upper estimate,  i.e.\
for every $0 \leq s \leq t \leq T$ we have
\begin{equation}
\label{discr-UEDE-intro}
\begin{aligned}
&
\en 
{\pwc {\mathsf{t}}\tau(t)}{\pwc U\tau(t)} + \mathcal{D}_{\tau}^{\mathrm{rate}} ([s,t]) +\mathcal{D}_{\tau}^{\mathrm{slope}} ([s,t]) 
\\
 & 
 \leq 
\en {\pwc {\mathsf{t}}\tau(s)}{\pwc U\tau(s)} +\int_s^t \partial_t 
\en {\pwc {\mathsf{t}}\tau(r)}{\pwc U\tau\left(r{-}\tfrac\tau2\right)} \dd r  + \mathrm{Rem}_{\tau}([s,t])\,,
\end{aligned}
\end{equation}
where $\pwc {\mathsf{t}}\tau: [0,T]\to [0,T]$ denotes the piecewise constant
interpolant of the notes $(k\tau)_{k=1}^N$ of the partition.  Here, the rate
contribution $ \mathcal{D}_{\tau}^{\mathrm{rate}}$ incorporates the primal
dissipation potential depending on the rate $\pwl U{\tau}'$ featuring in the
discrete energy-dissipation balances for the individual systems
$(\Spx j,\calE,2\calR_j^*)$, namely
\begin{subequations}
\begin{equation}
\label{rate-term-INTRO}
\mathcal{D}_{\tau}^{\mathrm{rate}} ([s,t]) := 
 2  \int_{s}^t
\left\{
\chi_{\tau} (r) \calR_1 \left(\tfrac12 \pwl U{\tau}'(r)\right) {+} (1{-}\chi_{\tau}(r) ) \calR_2 \left(\tfrac12 \pwl U{\tau}'(r)\right) \right\} \dd r,
\end{equation}
where $\chi_\tau:[0,T]\to\{0,1\}$ is the characteristic function of the union
of the  left semi-intervals  $((k{-}1)\tau, (k{-}1/2)\tau]$.
Accordingly, the slope contribution $ \mathcal{D}_{\tau}^{\mathrm{slope}}$
features the dual dissipation potentials evaluated at the force
$\pwc \xi{\tau} (t) \in \partial \calE(\pwc {\mathsf{t}}\tau(t), \pwc
U\tau(t))$, i.e.\
\begin{equation}
\label{slope-term-INTRO}
\mathcal{D}_{\tau}^{\mathrm{slope}} ([s,t]) : =   2  \int_{s}^t
\left\{ \chi_{\tau}(r) \calR_1^* ({-} \pwc \xi{\tau} (r))  {+} (1{-}\chi_{\tau}(r) ) \calR_2^* ({-}\pwc \xi{\tau} (r))  \right\}  \dd r \,.
\end{equation}
\end{subequations}
  The previously required  $\lambda$-convexity of the energy $\calE(t,\cdot)$  plays a key role in the estimate of  the remainder term  via
\[
\begin{aligned}
 \mathrm{Rem}_{\tau}([s,t])  & : = \frac1\tau \int_s^t \left(  
  \en {\pwc {\mathsf{t}}\tau(r)}{\pwc U\tau(r)}{-}
 \en {\pwc {\mathsf{t}}\tau(r)}{\pwc U\tau(r{-}\tfrac\tau2)} 
 {-} \pairing{}{\Spx 1}{\pwc\xi \tau(r)}{\pwc U\tau(r){-}
     \pwc U\tau(r{-}\tfrac\tau2)}\right)  \dd r
 \\
 &
 \ \leq \ \frac\lambda2 \int_s^t \|  \pwl U{\tau}'(r) \|_{\Spx 1} \, \| \pwc
 U\tau(r){-} \pwc U\tau(r{-}\tfrac\tau2) \|_{\Spx 1} \dd r\,. 
 \end{aligned}
\]
This estimate ensures that $ \lim_{\tau\to 0}\mathrm{Rem}_{\tau}([s,t])=0$.
However, we emphasize that $\lambda$-convexity of $\calE(t,\cdot)$ is not
necessary for our analysis and will not be used elsewhere in the paper.  It is
assumed here, only, in order to illustrate our results for the gradient system
$(\Spx 1, \calE, \calR_1^*{+}\calR_2^*)$ in simple and self-contained a way.
As we will see in Section\ \ref{s:AltMinMov}, this additional convexity
condition can be avoided by a more careful handling of the discrete estimates
via the \emph{variational interpolant} of the discrete solutions.

Taking the limit in \eqref{discr-UEDE-intro} leads to an upper
energy-dissipation inequality that, assuming the validity of a suitable chain
rule property for the energy $\calE$, is in fact equivalent to the
corresponding energy-dissipation balance and yields a solution to the
generalized gradient system $(\Spx 1, \calE, \calR_1^*{+}\calR_2^*)$. This is
summarized in the following result, anticipating Theorem
\ref{th:conv-split-step-MM} ahead. In the statement below  we do not
detail all  technical  assumptions on the quintuple
$(\Spx 1, \Spx 2, \calE, \calR_1, \calR_2)$, but  we highlight the
crucial, additional requirement that the Fr\'echet subdifferential is a
singleton, called \emph{singleton condition} subsequently.

The counterexample constructed in
Section\ \ref{ss:3.2} shows that convergence of the time-splitting scheme to a
solution of \eqref{dual-formulation-SPLIT-intro} may in fact be \emph{false},
if the singleton condition \eqref{singleton-intro} is not assumed, even in
the simple case $\Spx1=\Spx2=\R^2$. \medskip

\begin{nonotheorem}\label{th:Intro}\itshape
Under suitable conditions on $(\Spx 1, \Spx 2, \calE, \calR_1, \calR_2)$,
suppose also that
\begin{equation}
 \label{singleton-intro}
 \partial^{\Spx 1} \calE(t,u) = \partial^{\Spx 2} \calE(t,u)  \text{ is a
   singleton for all } (t,u) \in \mathrm{dom}(\calE)\,. 
\end{equation}
Then, for any null sequence $\tau\to0$ the curves $(\pwc U{\tau})$,
$(\pwl U{\tau})$, and $(\pwc \xi{\tau})$ suitably converge to a pair $(U,\xi)$
 solving the subdifferential inclusion \eqref{dual-formulation-SPLIT-intro}
and fulfilling the energy-dissipation balance
\begin{equation}
  \label{EDB-effect-INTRO}
  \en t{U(t)} + \int_s^t \left(  \calR_\eff (U'(r)){+} (\calR_1^*{+}
    \calR_2^*)({-}\xi(r)) \right)   \dd r  = \en s{U(s)} 
  + \int_s^t \pet r{U(r)}\dd r   
\end{equation}
for every $ 0\leq s \leq t \leq T$, where $\calR_\eff : \Spx 1\to [0,\infty)$
is the primal dissipation potential corresponding to $(\calR_1^*{+}\calR_2^*)$,
namely the $\inf$-convolution of $\calR_1$ and $\calR_2$
\[
 \calR_\eff(v) : = \inf_{v_1,\, v_2 \in \Spx 1, \ v = v_1+v_2 } \left(
   \calR_1(v_1)+\calR_2(v_2) \right)\,. 
\]
\end{nonotheorem}

Without entering into the details of the proof, we now motivate how
$\calR_\eff$ naturally arises in the passage to the limit in the rate term
from \eqref{rate-term-INTRO}. For simplicity, and with no loss
in generality, we illustrate this when $[s,t]=[0,T]$. Then,  for $\tau>0$
we have 
\[
\begin{aligned}
\mathcal{D}_{\tau}^{\mathrm{rate}} ([0,T])  &
= \sum_{k=1}^{N}  \left\{  \int_{(k{-}1)\tau}^{(k{-}1/2)\tau}  \! 
  2\calR_1(\tfrac12 \pwl U{\tau}'(r))   \dd r + 
 \int_{(k{-}1/2)\tau}^{k\tau}  \! 2\calR_2(\tfrac12 
   \pwl U{\tau}'(r))  \dd r   \right\} 
 \\
 &
 \overset{(1)}{=} \sum_{k=1}^{N} \tau  \left\{ \dashint_{(k{-}1)\tau}^{(k{-}1/2)\tau} 
\! \calR_1(\tfrac12 \pwl U{{\tau}}'(r))  \dd r +
 \dashint_{(k{-}1/2)\tau}^{k\tau}  \! \calR_2 
   (\tfrac12 \pwl U{\tau}'(r)) \dd r    \right\}
 \\
 & 
  \stackrel{(2)}{\geq} \sum_{k=1}^{N} \tau \left\{ \calR_1
    \!\left(\dashint_{(k{-}1)\tau}^{(k{-}1/2)\tau}  \!   
    \tfrac12 \pwl U{{\tau}}'(r))   \dd r \right)  +
 \calR_2 \!\left(\dashint_{(k{-}1/2)\tau}^{k\tau} \!  
  \tfrac12  \pwl U{{\tau}}'(r) \dd r   \right)   \right\}
\\
& 
 \stackrel{(3)}{\geq}   \sum_{k=1}^{N}  {\tau}  \: \calR_{\eff} \left( 
 \frac1{\tau}  \int_{(k{-}1)\tau}^{k\tau}  \bbU'_{{\tau}}(r)   \dd r\right)
   = \int_0^T {\calR}_{\eff} \big(\, \bbU'_\tau (r)) \dd r, 
 \end{aligned}
\]
 where $\bbU:[0,T]\to \Spx1$ denotes the piecewise affine interpolant with
$\bbU_\tau(k\tau)= \pws U \tau(k\tau)$ for $k\in \{0,\ldots, N\}$. In the above
calculation,  on the right-hand side of {\footnotesize (1)} the symbol
$\dashint$ denotes the integral average, for {\footnotesize (2)} we have used
 convexity of $\calR_j$ and Jensen's inequality, while {\footnotesize
  (3)} follows from the definition of $ {\calR}_{\eff}$.  Taking the limit
$\tau \to 0$ we can use $\bbU'_\tau \weakto U'$ in $\rmL^1(0,T;\Spx 1)$, and  
obtain the first liminf estimate: 
\[
\liminf_{\tau \to 0} \mathcal{D}_{\tau}^{\mathrm{rate}} ([0,T])  \geq \int_0^T
\calR_\eff\big(U'(t)\big) \dd t \,.
\] 

 Likewise, the role of the singleton condition can be understood by a
perusal of the argument for taking the limit in the slope term.  For this
we introduce a key tool  for our analysis: the
\emph{repetition operators} $\TT^{(1)}_{\tau} $ and $\TT^{(2)}_{\tau} $
that, applied to a given function $\zeta: [0,T]\to \Spx j^*$, are defined
via 
\[
  \begin{aligned}
   & (\TT^{(1)}_{\tau} \zeta)(t): = \left\{
    \begin{array}{cll}
      \zeta(t) & \text{if } t \in ((k{-}1)\tau, (k{-}1/2)\tau] 
      & \text{for  } k =1,\ldots,N,
      \\
      \zeta (t{-}\tfrac{\tau} 2)  & \text{if } t \in  ((k{-}1/2)\tau,
      k\tau] &  \text{for  } k =1,\ldots,N,
    \end{array}
  \right.
  \\
  & (\TT^{(2)}_{\tau} \zeta)(t): = \left\{
   \begin{array}{cll}
    \zeta(t{+}\tfrac{\tau} 2) & \text{if } t \in ((k{-}1)\tau,
    (k{-}1/2)\tau]  &  \text{for  } k =1,\ldots,N,
    \\
    \zeta (t)  & \text{if } t \in  ((k{-}1/2)\tau, k\tau]  
          & \text{for } k =1,\ldots,N.
   \end{array}
  \right.
 \end{aligned}
\]
Thus, $ \TT^{(1)}_{\tau}$ replicates, on the  right semi-intervals 
$ ((k{-}1/2)\tau, k\tau]$, the restriction of $\zeta$ to the preceding 
left semi-intervals  $ ((k{-}1/2)\tau, k\tau]$, while $ \TT^{(2)}_{\tau} $
does the converse. 
 A crucial property of the repetition operators $\TT^{(j)}$ is seen when
calculating the slope part of the dissipation integral, namely (see
\eqref{rephrasing-via-repet} for more details)  
\begin{align*}
 \mathcal{D}_{\tau}^{\mathrm{slope}} ([0,T])   &  =
\int_0^T \Big\{ \chi_\tau(r) \, 2\cR_1\big({-}\pwc\xi\tau(r)\big)
 + \big(1{-} \chi_\tau(r)\big) \,2\cR_2\big({-}\pwc\xi\tau(r) \big) 
\Big\} \dd r 
\\
&=
   \int_0^T \Big\{  \calR_1^* \big({-} \TT^{(1)}_{\tau} \pwc \xi{\tau}(r) \big)
 {+}  \calR_2^* \big({-}\TT^{(2)}_{\tau}  \pwc \xi{\tau}  (r) \big) 
 \Big\} \dd r \,.
\end{align*}
 In the first line the
integrand on the right-hand side is a sum of two products, where each factor
only weakly converges for $\tau\to 0$; thus convergence is not clear. 
However, in the second line, where the
repetition operators appear, we only have one weakly converging sequence in
each term. Moreover, the  above expression for
$ \mathcal{D}_{\tau}^{\mathrm{slope}}$ well motivates the crucial role of the
singleton condition \eqref{singleton-intro}. In fact,  the two sequences
$( \TT^{(j)}_{\tau} \pwc \xi{\tau})_\tau$  can be shown to converge to
limits $\xi_j $ that fulfill $\xi_j(t) \in \partial^{\Spx j} \calE(t,U(t))$ for
almost all $t\in (0,T)$, provided a suitable closedness property for the
subdifferentials $ \partial^{\Spx j} \calE$ is assumed. Then, condition
\eqref{singleton-intro} guarantees that $ \partial^{\Spx j} \calE(t,U(t))$ is a
singleton of $\Spx j^*$, so that
$ \partial^{\Spx 1} \calE(t,U(t)) = \{ \xi_1(t) \} = \{ \xi_2(t) \}
=\partial^{\Spx 2} \calE(t,U(t)) $ for a.a.\ $t\in (0,T)$. Then,
$ \xi := \xi_1=\xi_2 : (0,T) \to \Spx 1^*$ gives rise to the force term in
\eqref{EDB-effect-INTRO} since
\[
\liminf_{\tau \to 0}  \mathcal{D}_{\tau}^{\mathrm{slope}} ([0,T]) \geq  
 \int_0^T \!\! \big\{  \calR_1^*({-} \xi(r)) {+}  \calR_2^*({-}\xi(r)) 
\big\} \dd r  =  { \int_0^T \! \cR_\eff^*( {-} \xi(r) )  \dd r \,.  }
\]

\subsection*{Splitting schemes for block structures}
In the second part of the paper we tackle the application of the splitting
method to generalized gradient systems with a \emph{block structure}. In such
systems,
\[
\text{the state variable $u$ is a vector } u= (y,z)^\top 
\in \blockspace: = \Spy \ti \Spz,
\]
with $\Spy$ and $\Spz$ (separable) reflexive Banach spaces.  The evolution of
the system is governed by a driving energy functional
$ \calE: [0,T]\ti \blockspace \to (-\infty, \infty]$ and by two dissipation
potentials $\disblo y : \Spy \to [0,\infty)$ and
$ \disblo z: \Spz \to [0,\infty)$, each acting on the components of the
rate vector  $u'= (y',z')^\top$, namely
\[
\calR(u')=\calR(y',z')=\big(\disblo y{\otimes} \disblo z\big)(y',z') :=
\disblo y(y')+ \disblo z(z')\,. 
\]

 Models in solid mechanics described by the evolution of the displacement,
or the deformation, of the body, coupled with that of an internal variable
describing  inelastic processes such as, e.g., plasticity, heat transfer,
delamination, fracture, damage, typically fit into this framework.  Applying
the splitting method to this context boils down to letting first the variable
$y$ evolve  on a semi-interval  while keeping $z$ fixed, and then
letting $z$ evolve with $y$ fixed  on the next semi-interval.  On the
discrete level, this can be compared with staggered minimization schemes, 
which have been recently used for gradient flows and rate-independent systems,
cf.\ e.g.\ \cite{Roub10TRIP, RoubThoPana15, KneNeg17CAMS, ACFS17} among
others. 
\par
The analysis of this kind of systems can be framed 
 in the context of our splitting approach by  introducing the dissipation potentials 
$\calR_j: \blockspace \to [0,\infty]$ 
\begin{equation}
\label{eq:I.dissip-potent-Rj}
\begin{aligned}
\calR_1(u') =  \calR_1(\yvel,\zvel)=\disblo y(\yvel) + \mathcal{I}_{\{0\}}(\zvel),  \qquad  \calR_2(u') =  \calR_2(\yvel,\zvel)=\calI_{\{0\}}(\yvel) + \disblo z(\zvel),
\end{aligned}
\end{equation}
%
where $\calI_{\{0\}}$ is the indicator function of the singleton $\{0\}$, with
$\calI_{\{0\}}(0)=0$ and $\infty$ otherwise.  We then concatenate the
(time-discrete or time-continuous) solutions to the subdifferential inclusions
\[
\partial \calR_j(u'(t)) + \partial^{\blockspace} \calE(t,u(t))  \ni 0 \quad 
\text{in } \blockspace^* \qquad \foraa\  t \in (0,T)
\]
governing the single-dissipation systems $(\blockspace,\calE, \calR_j)$.
Because of \eqref{eq:I.dissip-potent-Rj}, in each step we either freeze the
variable $z$ (for $j{=}1$), or the variable $y$ (for $j{=}2$).  In particular, on
the time-discrete level the corresponding Alternating Minimizing Movement
scheme consists of two consecutive minimum problems in which either the
variable $z$ stays fixed and minimization only involves the variable $y$, or
$y$ is fixed and we minimize only with respect to $z$.  Nonetheless, let us
stress that the energy functional $\calE(t,\cdot)$ is defined on the
product space $\blockspace$, and likewise the Fr\'echet subdifferential
$\partial^{\blockspace} \calE$ involves the
$\pairing{}{\blockspace}{\cdot}{\cdot}$ duality. 

 The analysis in Sections \ref{s:true2} to \ref{s:AltMinMov} relies on the
ordering assumption that $\Spx2\subset \Spx1$ continuously and that $\calR_j$
and $\calR_j^*$ are superlinear on $\Spx j$. This assumption is no longer
applicable in the case with block structure, hence Section \ref{s:block} shows
how the analysis can be adapted; in particular the singleton condition
\eqref{singleton-intro} can be replaced by the weaker \emph{cross-product
  condition}  
\[
 \partial^{\blockspace} \calE(t,y,z) =\blosub y (t,y,z) \ti 
 \blosub z (t,y,z)  \qquad \text{for all } 
  (t,y,z) \in [0,T]\times \blockspace\,.
\]
Under this and other conditions on the  generalized gradient system
$(\blockspace, \calE, \disblo y{\otimes}\disblo z)$  we prove our two convergence
results, Theorem \ref{thm:sec8} for the concatenation of time-continuous
solutions, and Theorem\ \ref{thm:sec8-MM} for  the Alternating Minimizing
Movement scheme.  
\medskip

\noindent
\paragraph{\bf Plan of the paper.} Section \ref{s:true2} lays down the
foundations for our analysis: after recalling some known facts about the
energy-dissipation balance and its role in characterizing solutions for an
abstract gradient system $(\gX,\gE, \gR)$ in Section\ \ref{s:prelims}, in
Section\ \ref{s:2} we  specify our working assumptions on the quintuple
$(\Spx 1,\Spx 2, \calE, \calR_1, \calR_2)$ and expound some of their
consequences.

We devise the time-splitting scheme on the time-continuous level in Section\
\ref{se:TimeSplit}, which also contains the statement of our first main convergence
result in Theorem \ref{th:conv-split-step},  proved throughout Section
\ref{s:4}. Section \ref{s:examples} is instead devoted to examples illustrating
the theory: a counterexample to convergence of the time-splitting scheme
without the singleton condition, and a doubly-nonlinear PDE example.  In
Section \ref{s:AltMinMov} we show that the analysis carried out in Sections
\ref{se:TimeSplit} and \ref{s:4} can be easily adapted to the Alternating Minimizing
Movement scheme introduced in \eqref{min-schemes-i-intro} and state our second
convergence result in Theorem \ref{th:conv-split-step-MM}, which is the
discrete counterpart to Theorem \ref{th:conv-split-step}.  In Section
\ref{s:block} we turn to systems with block structure and provide our last two
convergence results in Theorems \ref{thm:sec8} and \ref{thm:sec8-MM}. Finally,
Appendices \ref{s:QYE} and \ref{s:appB} contain some auxiliary results.

\section{Setup and assumptions}
\label{s:true2}

In the ensuing Section\ \ref{s:prelims} we are going to collect some key facts
about general gradient systems $(\gX, \gE, \gR)$ that will be used throughout
the paper.  Then, in Section\ \ref{s:2} we will move to the context of two
dissipation mechanisms and settle our working assumptions on the gradient
systems $(\Spx j, \calE, \calR_j)$, $j \in \{1,2\}$.

\subsection{Preliminaries on gradient systems}
\label{s:prelims}

By \emph{generalized gradient system} we mean a triple $(\gX, \gE, \gR)$
satisfying the conditions \textbf{(1), (2), (3)} explained in the
Introduction. More precisely, throughout this section, we will assume that
$\gE$ and $\gR$ satisfy the following basic conditions:
\begin{itemize}
\item[$\mathbf{<E>}$ {\bf (time differentiability and power control):}] {\sl
    The energy $\gE: [0,T]\ti \gX \to (-\infty,\infty]$ is a lower
    semicontinuous functional, bounded from below by a positive constant, and
    has a proper domain  $ [0,T]\ti \gD$ with $\gD\subset \gX$. 
    Moreover, for every $u\in \gD$ the function $t\mapsto \gE(t,u)$ is
    differentiable and}
\begin{equation}
\label{gpower-control}
\exists\, C_{\#}>0 \ \ \forall\, (t,u) \in [0,T]\ti \gD: \quad 
  |\partial_t \gE(t,u)| \leq C_{\#} \gE(t,u)\; .
\end{equation}
\item[$\mathbf{<R>}$ {\bf (superlinear dissipation potential):}] {\sl The
     potential $\gR: \gX \to [0,\infty)$ is lower semicontinuous and
    convex, $\gR(0)=0$ and, together with its convex conjugate
    $\gR^*:(\gX^*,\|\cdot\|_*) \to [0,\infty)$, the functional $\gR$  is superlinear:}
\begin{equation}
\label{gsuperlinear-growth}
\lim_{\gnorm{v}\up \infty} \frac{\gR(v)}{\gnorm{v}} = \lim_{\gnorms{\xi}\up
  \infty} \frac{\gR^*(\xi)}{\gnorms{\xi}} =\infty\,. 
\end{equation}
\end{itemize}
In particular, we emphasize that we will confine our study to the case of
\emph{state-independent} dissipation potentials, although we believe that all
of our results could be extended to potentials $\gR=\gR(u,v)$ with a suitably
tamed dependence on the state variable $u$. Instead, our analysis cannot
encompass potentials taking the value $\infty$ (for instance, including
indicator terms that force unidirectional evolution), since the treatment of
the corresponding gradient systems necessitates additional estimates that are
outside the scope of the present paper.  
\par
The central result of this section is Proposition \ref{p:EDP}, which is
indeed at the core of our approach to gradient systems. It provides a
characterization of the gradient-system evolution given by the subdifferential
inclusion \eqref{genDNE}, in terms of the so-called 
Energy-Dissipation Principle, combined with the chain rule for the energy
functional $\gE$. The following characterization of \eqref{genDNE} via an upper
energy estimate  has been circulating for some time: it is in fact
underlying analysis of gradient flows and \emph{generalized} gradient flows in
Banach and metric spaces, cf.\ e.g.\ \cite{Ambrosio95,AGS08, RMS08, MRS2013, Miel23IAGS}.
Nonetheless, in order to make the paper self-contained
we detail the proof of the following result here as well.

\begin{proposition}[Energy-dissipation principle]
\label{p:EDP}
Let $u \in \AC([0,T];\gX )$ and $\xi \in \rmL^1 ([0,T];\gX^* )$ with
$\xi(t) \in \partial \gE(t,u(t))$ for almost all $t\in (0,T)$.  Suppose that
the map $t\mapsto \gE(t,u(t))$ is absolutely continuous on $[0,T]$ and that for
$(u,\xi)$ the chain rule
\begin{equation}
  \label{eq:48strong}
  \frac \dd{\dd t}  \gE(t,u(t)) - \partial_t  \gE(t,u(t))
  =   \pairing{}{}{\xi(t)}{u'(t)}  \quad 
  \text{for a.a.\ }t\in (0,T)
\end{equation}
holds. Then, the following conditions are equivalent:
\begin{enumerate}
\item[\upshape (C1)] The pair $(u,\xi)$ complies with the upper energy-dissipation
  estimate, i.e.
  \begin{align}
    \label{EDB-UE}
    \hspace*{-1em} \gen T{u(T)} + \int_0^T \!\! \big\{  \gR(u'(r)){+} \gR^*({-}\xi(r)) \big\} \dd r  
    \leq \gen 0{u(0)} +\int_0^T \!\! \partial_t \gen r{u(r)} \dd r \,.
  \end{align}
\item[\upshape (C2)] The pair $(u,\xi)$ solves 
  \begin{equation}
    \label{solving-GFE}
    - \xi(t) \in  \partial \gR (u'(t)) 
    \qquad \text{for a.a.\ $t\in (0,T)$,}
  \end{equation} 
  and fulfills  the energy-dissipation balance
  \begin{equation}
    \label{gEDB-effect}
    \gen t{u(t)} + \int_s^t \!\! \big\{  \gR (u'(r)){+} \gR^*({-}\xi(r)) \big\} 
    \dd r  = \gen s{u(s)} +\int_s^t \!\! \partial_t \gen r{u(r)} \dd r  
  \end{equation}
  for all $s,t \in [0,T]$ with  $s\leq t$. \vspace{-1em}
\end{enumerate}
\end{proposition}
\begin{proof}  Obviously, (C2) implies (C1). 

To show that condition (C1) implies (C2),  we apply the chain rule
\eqref{eq:48strong} and deduce from the energy-dissipation upper estimate
\eqref{EDB-UE} that
\[
\begin{aligned}
  \int_0^T \!\!\left\{ \gR (u'(r)){+} \gR^*({-}\xi(r)) \right\} \dd r & \leq
  \gen 0{u(0)} - \gen T{u(T)} +\int_0^T \!\! \partial_t \gen r{u(r)} \dd r
  \\
  & = \int_0^T \!\! \pairing{}{}{ {-}\xi(r)}{ u'(r) } \dd r\,.
\end{aligned}
\]
Now, since $\gR (v) + \gR^*(\zeta) \geq \pairing{}{\Spx {}}{\zeta} v $ for all
$(v,\zeta) \in \gX \ti \gX^*$, from the above inequality we immediately infer
that
\begin{equation}
\label{gFenchel-Moreau-equality}
 \gR (u'(t)) +  \gR^*({-}\xi(t))  = \pairing{}{}{ {-}\xi(t)}{ u'(t)}   
   \qquad \foraa\  t \in (0,T)
\end{equation}
which, by a well-known convex analysis result, is equivalent to the inclusion
\eqref{solving-GFE}.  The energy-dissipation balance \eqref{gEDB-effect}
follows by integrating \eqref{gFenchel-Moreau-equality} on an arbitrary time
interval $[s,t]\subset [0,T]$ and again applying the chain rule
\eqref{eq:48strong}.  This concludes the proof.
\end{proof}

\subsubsection*{The chain rule and the \emph{Quantitative Young Estimate}}

The cornerstone in the proof of Proposition \ref{p:EDP} is indeed the chain
rule formula \eqref{eq:48strong}.  Let us now gain further insight on its
validity: In view of estimate \eqref{gpower-control}, for any curve
$u\in \AC([0,T];\gX)$ along which $t\mapsto \gen t{u(t)}$ is absolutely
continuous we even have that $t\mapsto \partial_t \gen t{u(t)}$ is in
$\rmL^\infty([0,T])$. Therefore, underlying \eqref{eq:48strong} is the fact
that the function $t\mapsto \pairing{}{}{\xi(t)}{u'(t)} \in
\rmL^1([0,T])$. That is why, the chain rule condition is normally required
along curves $(u,\xi)$ for which the second estimate in \eqref{first-part-CR}
below is valid. Namely, the chain rule hypothesis for $\gE$ is typically
formulated in this way:
\begin{itemize}
\item[$\mathbf{<CR>}$ \textbf{(chain rule):}] {\sl the pair
    $(\gE,\partial \gE)$ satisfies the following property: for any
    $(u,\xi) \in \AC([0,T];\gX)\ti \rmL^1([0,T];\gX^*)$ such that
    $\xi(t) \in \partial \gE(t,u(t))$ for almost all $t\in (0,T)$ and}
\begin{equation}
\label{first-part-CR}
\sup_{t\in [0,T]} |\gen t{u(t)}|<\infty \text{ and } \int_0^T 
\gnorms{\xi(t)} \, \gnorm{u'(t)} \dd t <\infty \, ,
\end{equation}
{\sl  the function $t \mapsto \gen t{u(t)}$ is absolutely continuous and
   the chain rule \eqref{eq:48strong} holds.}
\end{itemize}

The \emph{Quantitative Young Estimate} $\QYE$ which we introduce below is the
condition on the dissipation potential $\gR$ that enables us to apply the chain
rule $\mathbf{<CR>}$ to pairs $(u,\xi)$ satisfying the energy-dissipation upper
estimate \eqref{EDB-UE}. It strengthens the \emph{Young estimate}
  \[
 \gR(v) +\gR^*(\xi) \geq\pairing{}{X}{\xi}v \qquad \text{for all } (v,\xi) \in \Spxname\ti \Spxname^*\,,
  \]
  in that one requires the
  \begin{itemize}
\item[$\mathbf{<QYE>}$ \textbf{(Quantitative Young Estimate):}] 
 \begin{equation}
  \label{eq:QuYouEst}
  \exists\,c, C>0\ \forall \, (v,\xi)\in \gX \ti \gX^*: \quad
\gR(v)+\gR^*(\xi) \geq c\|v\|\,\| \xi \|_* -C.
\end{equation}
\end{itemize}

Indeed, if $(\gX,\gE,\gR)$ satisfies conditions $\mathbf{<E>}$, $\mathbf{<R>}$,
$\mathbf{<CR>}$, and $\mathbf{<QYE>}$,  then, thanks  to
\eqref{eq:QuYouEst}, for any pair
$(u,\xi) \in \AC([0,T];\gX)\ti \rmL^1([0,T];\gX^*)$ satisfying \eqref{EDB-UE}
we have that  the estimates in \eqref{first-part-CR} hold. Hence, the
chain rule formula \eqref{eq:48strong} is valid.

Let us now gain further insight into condition $\mathbf{<QYE>}$. It is known
that estimate \eqref{eq:QuYouEst} is not true in general, see
\cite[Ex.\,3.4]{MR21} for a counterexample.  In turn, the following result
provides a useful sufficient condition for \eqref{eq:QuYouEst}.

\begin{lemma}
\label{l:2.2} 
Assume that there exists a continuous, convex and superlinear function
$\psi:[0,\infty[\to[0,\infty[$ and constants $c_1\geq1$ and $c_2,c_3>0$ such
that
\begin{equation}
\label{e:ALEX-suff-cond}
\forall\, v \in \gX\,: \quad  \psi(\frac{1}{c_3}\|v\|)- c_2 \leq \gR(v) \leq
c_1\, \psi(c_3\|v\|) + c_2. 
\end{equation}
Then, the quantitative Young estimate \eqref{eq:QuYouEst} holds with $c=1/(c_1c_3^2)$ and $C=2c_2$.
\end{lemma}
In particular, the quantitative Young estimate \eqref{eq:QuYouEst} holds if $\gR$ is given by a functional of the norm, i.e. $\gR(v) = \psi(\|\mathbb{B}v\|)$ with a bounded, invertible operator $\mathbb{B}:\gX\to \gX$.
\\[0.3em]
\begin{proof}
  Thanks to \eqref{e:ALEX-suff-cond},  we have 
  $\gR^*(\xi) \geq c_1\psi^*\big(\|\xi\|_*/(c_1c_3) \big) - c_2$. Combining this
  with the Young's inequality $\psi(r)+\psi^*(s)\geq rs$ and using that
  $c_1\geq 1$ we infer
\begin{align*}
\gR(v) + \gR^*(\xi) &\geq \psi(\|v\|/c_3)-c_2 + c_1 \psi^*(\|\xi\|_*/ c_1c_3) - c_2
\\
&\geq \psi(\|v\|/c_3)+ \psi^*(\|\xi\|_*/ c_1c_3) -2c_2  \geq \frac1{c_1c_3^2}\, \|v\| \,
\|\xi\|_* - 2c_2.
\end{align*}
Thus, the result is established. 
\end{proof}

\subsection{Assumptions  on the generalized gradient systems}
\label{s:2}
As we will see in settling our requirements on the gradient systems
$(\Spx j, \calE, \calR_j)$, $j \in \{1,2\}$, the conditions expounded in
Section\ \ref{s:prelims} will have to be adjusted to the interplay between the
topologies of the spaces $(\Spx j, \norm{\cdot}j)$.

\subsubsection*{Ordering of Banach spaces}
We consider two (separable) and reflexive Banach spaces $(\Spx 1, \norm{\cdot} 1) $ and $(\Spx 2,  \norm{\cdot} 2)$, such that 
\begin{subequations}
\label{eq:ordering}
\begin{equation}
  \label{eq:oderX2X1}
  \Spx 2 \subset \Spx 1 \ \text{ and } \ \Spx 1^* \subset \Spx 2^*
  \text{ densely and continuously.}
\end{equation}
More precisely,  we  assume that 
\begin{equation}
\label{norm-control}
\exists\, C_\rmN \geq 1 \ \forall\, v \in \Spx 2 : \quad \norm{v}1 \leq  C_\rmN
\norm v2\,. 
\end{equation}
\end{subequations}

\subsubsection*{Driving energy functional}
Clearly, the time-dependent energy functional $\calE$ needs to have a domain
contained in the smaller space $[0,T]\ti\Spx 2$. It is on this domain that we
require the first set of basic conditions, namely conditions $\mathbf{<E>}$
previously introduced: boundedness from below (by a constant that, up to a
shift, can be assumed positive), time differentiability, and control of the
power functional by means of the energy functional itself.

\begin{hypothesis}[Time differentiability]
\label{hyp:en-1}
The energy $\calE: [0,T]\ti \Spx 1 \to (-\infty,\infty]$ has the proper domain 
\begin{equation}
\label{bded-below}
\mathrm{dom}(\calE) = [0,T]\ti \domE_0  \text{ with } \domE_0 \subset \Spx 2 \
\text{ and }  \ \exists\, C_0>0 \ \forall\, (t,u) \in  
[0,T]\ti \domE_0 : \ \en tu \geq C_0.
\end{equation}
Moreover, on $[0,T]\ti \domE_0$ the functional $\calE$ complies with $\mathbf{<E>}$.
\end{hypothesis}

It is convenient to  introduce the functional 
\begin{equation}
\label{frakE}
\mfE: \domE_0 \to [0,\infty), \qquad \mfE(u): = \sup_{t\in [0,T]}\en tu,
\end{equation}
and observe that, by
\eqref{gpower-control}
 and Gr\"onwall's lemma, 
\begin{equation}
\label{Gronwall-est}
\mfE(u) \leq \rme^{ C_\# T} \en tu \qquad\text{for all } (t,u) \in [0,T]\ti \domE_0\,.
\end{equation}
We will work with the sublevel sets 
\begin{equation}
\label{sublevel-sets}
S_E: = \{ u \in \Spx 2\, :  
\mfE(u) \leq E\}\,, \qquad E>0.
\end{equation}
We can now formulate our second condition.

\begin{hypothesis}[Lower semicontinuity \&  continuity of the power]
\label{h:2}
We require that for every $j \in \{1,2\}$  there holds
\begin{equation}
\label{h:2.1}
\begin{aligned}
\forall\, E>0:\  \
 & \Big( (t_n,u_n) \weakto (t,u) \text{ in } [0,T]\ti \Spx 1 \text{ and } 
   \ u_n \in S_E \text{ for all }n \in \N  \Big) 
 \\
 &  \Longrightarrow  \ \begin{cases}
 \liminf_{n\to \infty} \calE(t_n,u_n) \geq \calE(t,u)\,,
 \\
 \lim_{n\to\infty} \pet {t_n}{u_n} = \pet tu\,.
 \end{cases}
 \end{aligned}
\end{equation}
\end{hypothesis}
Thanks to \eqref{h:2.1}, the functional $\mfE$ is weakly lower semicontinuous
in $\Spx 1$ and thus the sublevel sets $S_E$ are (sequentially) weakly closed
in $\Spx 1$. We highlight that, in \eqref{h:2.1} lower semicontinuity of
$\en t{\cdot}$ is, a priori, required with respect to the coarser topology
given by $\Spx 1$. Nonetheless, in concrete examples $\en t{\cdot}$ may turn
out to be coercive with respect to the norm $\norm{\cdot}2$,  recall
$ \domE_0 \subset \Spx2$.  Hence, the information that there exists $E>0$
with $ u_n \in S_E $ for all $n\in\N $ may yield additional compactness
information on the sequence $(u_n)_n$, and in fact weaken the above lower
semicontinuity/continuity requirements.

Hypothesis \ref{hyp:en-2} below involves the Fr\'echet subdifferentials of
$\en t{\cdot}$ with respect to the duality pairings
$\pairing{}{\Spx j}{\cdot}{\cdot}$, namely the operators
$\frnameopt j : [0,T]\ti \domE_j \rightrightarrows \Spx j^*$ defined by
\[
\xi \in \frsubopt j tu  \ \Longleftrightarrow \ \en tw - \en tu \geq
\pairing{}{\Spx j}\xi{w{-}u} {} + \rmo(\norm{w{-}u}j) \quad \text{as }
\norm{w{-}u}j \to 0\,, 
\]
where for all $t\in[0,T]$ we define
$\domE_j=\{u\in\domE_0: \frsubopt j tu\neq \emptyset\}$. We have
$\domE_1\subset \domE_2$ and
\begin{equation}
\label{relation-subdiff}
\frsubopt 1tu \subset \frsubopt 2tu \cap \Spx 1^* \qquad \text{for all } (t,u) \in [0,T]\ti \domE_1\,.
\end{equation}

\begin{hypothesis}[Closedness of $\frnameopt j$ on energy sublevels]
\label{hyp:en-2}
We require that for every $j\in \{1,2\}$  we have  
\begin{equation}
\label{h:2.2}
\forall\, E>0: \left.
\begin{cases}
 (t_n,u_n,\xi_n) \weakto (t,u,\xi) \text{ in } [0,T]\ti \Spx 1 \ti \Spx j^*, 
 \\
  u_n \in S_E, \ \xi_n \in \frsubopt j{t_n}{u_n} 
   \text{ for all } n \in \N  
   \end{cases}\!\right\}
    \ \   \Longrightarrow  \ \ \xi \in \frsubopt jtu\,.
\end{equation}
\end{hypothesis}

We emphasize that in \eqref{h:2.2} closedness is imposed along sequences weakly
converging in $\Spx 1$; thus, for $j=2$ condition \eqref{h:2.2} may also be
understood as a compatibility requirement between the duality pairing of
$\Spx 1$ and that of $\Spx 2$.  Again, we remark that the requirement
$(u_n)_n \subset S_E$ and the additional compactness information granted
by it may turn \eqref{h:2.2} into a (more standard) closedness condition of the
graph of $\partial^{\Spx j}\cE$ in $\Spx j\times \Spx j^*$.

Finally, along the footsteps of \cite{MRS2013,MR21} we will assume the validity
of the following chain-rule condition for the subdifferentials $ \frnameopt j$.

\begin{hypothesis}[Chain rules]
\label{h:chain-rule}
For $j=1,2$ the pair $(\calE, \frnameopt j)$ satisfies the chain rule property
$\mathbf{<CR>}$.
\end{hypothesis}

\subsubsection*{Dissipation potentials}
In what follows, we will work with two dissipation potentials satisfying the
following conditions.
\begin{hypothesis}
\label{hyp:dis}
For $j\in \{1,2\}$ the functionals $\calR_j: \Spx j \to [0,\infty)$  and their conjugates $\calR_j^*: \Spx j^* \to [0,\infty)$ comply with condition $\mathbf{<R>}$. 
\end{hypothesis}
For later use, we reformulate the superlinear growth conditions \eqref{gsuperlinear-growth} that we require for  $\calR_j$
and
$\calR_j^*$
 in terms of a unique convex, superlinear and monotone function $\unipsi$ providing a lower bound 
for  the primal and dual dissipation potentials.
\begin{lemma}
\label{l:unique-psi}
The dissipation potentials $\calR_j: \Spx j \to [0,\infty)$ fulfill  \eqref{gsuperlinear-growth} if and only if 
there exists a convex and increasing function $\unipsi: [0,\infty) \to [0,\infty)$, with superlinear growth,  such that, for 
$ j \in \{1,2\}$, 
\begin{equation}
\label{e:psi-minorant}
\calR_j(v) \geq \unipsi(\norm vj) \  \text{ and }\  \calR_i^* (\xi) \geq
\unipsi(\norms{\xi} j) \quad \text{for all } v \in \Spx j \text{ and } 
\xi \in \Spx j^*\,.
\end{equation}
\end{lemma}
\begin{proof}
Clearly, \eqref{e:psi-minorant} implies the superlinear growth
\eqref{gsuperlinear-growth} for the potentials $\calR_j$ and $\calR_j^*$. To
check the converse implication, observe that, since $\calR_j$ and  
$\calR_j^*$ are superlinear, for each $j\in \{1,2\}$  we have  
\begin{equation}
\label{e:K_sK}
\forall\, K\geq 0 \ \exists\, S_K^j,\, S_K^{j,*} \geq0 \ \forall\, v \in \Spx
j, \, \xi \in \Spx j^*:  \quad \begin{cases} 
\calR_j(v) \geq K \norm{v}j -S_K^j,
\\ 
\calR_j^*(\xi) \geq K \norms{\xi}j -S_K^{j,*},
\end{cases}
\end{equation}
with   $S_0^j=S_0^{j,*}=0$.
For fixed $K\geq 0$ set 
$\mathsf{S}_K: = \max_{j =1,2} \{ S_K^j,  S_K^{j,*}\}$
and  define 
\[
\unipsi(r): = \sup\{ Kr-\mathsf{S}_K\,: \ K \geq 0\}  \quad \text{for } r \in [0,\infty).
\]
 By construction,  $\calR_j$ and $\calR_j^*$ satisfy
\eqref{e:psi-minorant}.  It is immediate to check that $\unipsi$ is monotone
increasing, has superlinear growth at infinity, and since
$\unipsi(0)=0$, satisfies $\unipsi(r)\geq 0$ for all $r\geq 0$. Moreover,
$\unipsi$ is convex as it is given by the supremum of linear functions.
\end{proof}

We now introduce some specific conditions to deal with the split-step system
driven by the $\inf$-convolution of $\calR_1$ and $\calR_2$.  Since $\calR_1$
and $\calR_2$ are defined on two different Banach spaces, the effective
dissipation potential $\calR_\eff$ needs to be carefully specified.  It is more
straightforward to first define the dual potential $\calR_\eff^*$ on the
smaller dual space $\Spx 1^*$.  Therefore, let us introduce the functional
\[
  \Rtrial: \Spx 1^* \to [0,\infty), \qquad \Rtrial(\xi):= \calR_1^*(\xi)+
  \calR_2^*(\xi) \quad \text{for } \xi \in \Spx 1^*\,.
\]
We now identify $\Rtrial$ as the conjugate of the potential given by the
infimal convolution of $\calR_1$ and of the functional $\overline{\calR}_2$
that extends $\calR_2 $ to the whole of $\Spx 1$ by $\infty$ on
$\Spx 1{\setminus} \Spx 2$. We mention in advance that in \eqref{largest} below
we will directly define $ \calR_\eff$ via a minimum: in the proof of Lemma
\ref{l:Reff} we will show by the \emph{direct method} that the infimum is
attained.

\begin{lemma}[Properties of the inf-convolution]
\label{l:Reff}
Let $\overline{\calR}_2: \Spx 1\to [0,\infty)$ be defined by $\overline{\calR}_2(v): = \calR_2(v)$ if $v\in  \Spx 2$, and  by $\overline{\calR}_2(v): = \infty$ else.
Define 
  \begin{equation}
\label{largest}
 \calR_\eff : \Spx 1 \to [0,\infty) \qquad  \calR_\eff(v) : = \min_{v_1,\, v_2 \in \Spx 1, \ v = v_1+v_2 } \left( \calR_1(v_1)+\overline{\calR}_2(v_2) \right)\,.
 \end{equation}
 Then,   the following statements hold:\vspace{-0.5em}  
 \begin{enumerate}
 \item[\upshape(1)] 
 $\calR_\eff $ is  lower semicontinuous and convex, with $\calR_\eff^*= \Rtrial$.
 \item[\upshape(2)] 
   With $\Psi$ from Lemma \ref{l:unique-psi} there  holds
 \begin{equation}
 \label{crucial-bounds}
 2\unipsi\big( \frac1{2C_\rmN} \| v\|_1\big) \leq \calR_\eff(v) \leq \calR_1(v)
\leq \unipsi^*\big( \| v\|_1\big) \quad \text{for all } v \in \Spx 1. 
\end{equation}\vspace{-2.5em}
\end{enumerate}
\end{lemma}
 Estimate \eqref{crucial-bounds}
highlights that, ultimately,  the relevant Banach space for $\calR_\eff$ is
 $\Spx 1$. That is why, from now on we will use the notation 
 $\Spx {\eff}  := \Spx 1$.  
\\[0.3em]
\begin{proof}
Preliminarily, observe that also the extended potential $ \overline{\calR}_2$
is lower semicontinuous on $\Spx 1$. To see this, take a sequence $v_n\to v$ in
$\Spx 1$ such that $\liminf_n\overline{\calR}_2(v_n)<\infty$. Then $v_n\in \Spx
2$ and $\overline{\calR}_2(v_n)=\calR_2(v_n)$. By coercivity of $\calR_2$ and
reflexivity of $\Spx 2$ we have $v_n\rightharpoonup v$ in $\Spx 2$ for a
non-relabeled subsequence. Hence, $v\in \Spx 2$. Since $\calR_2$ is convex and
strongly lower semicontinuous, it is also weakly lower semicontinuous,  which yields that 
$\liminf_n\overline{\calR}_2(v_n) = \liminf_n \calR_2(v_n) \geq \calR_2(v) =
\overline{\calR}_2(v) $.\medskip  

\noindent
\emph{Proof of Claim} (1):
Applying  \cite[Thm.\,1, p.\,178]{IofTih79TEPe} we immediately find that 
\[
  \calR_\eff^*= \calR_1^*+ \overline{\calR_2}^* = \calR_1+\calR_2^*|_{\Spx 1^*}
  = \Rtrial
\]
(where the duality pairing with respect to all the conjugates above is that
between $\Spx 1^*$ and $\Spx 1$).  In turn, thanks to Hypothesis \ref{hyp:dis},
$\mathrm{dom}(\calR_j^*)=\Spx j^*$ for $j\in \{1,2\}$, hence both conjugate
potentials $\calR_j^*$ are continuous on their common domain $\Spx 1^*$.
Therefore, \cite[Thm.\,1]{IofTih79TEPe} again applies, yielding that
$ (\calR_1^*{+}\calR_2^*)^* = (\calR_1^{**} \INFCONV \calR_2^{**} ) $ (where
the have simply written $\calR_2^*$ in place of $\calR_2^*|_{\Spx 1^*}$).
Thus,
\[
\calR_\eff^{**} = \Rtrial^* =  (\calR_1^*{+}\calR_2^*)^* = (\calR_1^{**}
\INFCONV \calR_2^{**} ) = (\calR_1 \INFCONV\, \overline{\calR_2}) = \calR_\eff\,,
\]
which ensures that $\calR_\eff $ is convex and lower semicontinuous on $\Spx 1$. \\

Moreover, we see that in \eqref{largest} the minimum is attained. Indeed, 
taking, for a given $v\in \Spx 1$ with $\calR_\eff(v)<\infty$,   minimizing sequences $(v^1_n)_n,\, (v^2_n)_n \subset  \Spx 1$ 
 such that  $v^1_n+v^2_n=v$, we again observe that, by
coercivity, up to a (non-relabeled) subsequence we have
$v_n^1\rightharpoonup v^1$ in $\Spx 1$ and $v_n^2\rightharpoonup v^2$ in
$\Spx 2$ for some $v^i \in \Spx i$. Hence, we obtain that
\[
\cR_\eff (v) \geq \liminf_n \calR_1(v_n^1) + \liminf_n \calR_2(v_n^2) 
 \geq \calR_1(v^1) +\calR_2(v^2),
\] 
and the claim follows.\medskip

\noindent
\emph{Proof of Claim} (2): For all $\xi \in \Spx 1^*$
we observe the following chain of inequalities: 
\begin{align}
\unipsi\big(\norms{\xi}{1}\big) &\stackrel{(1)}{\leq} \calR_1^*(\xi) \leq \calR_\eff^*(\xi)  
\stackrel{(2)}{\leq}  \unipsi^*\big( \norms{\xi}1\big) + \unipsi^*\big( \norms \xi 2\big) 
 \\
& \stackrel{(3)}{\leq}  \unipsi^*\big( \norms{\xi}1 \big) + \unipsi^*\big( C_\rmN  \norms{\xi}1\big)
\stackrel{(4)}{\leq}  2\unipsi^*(C_\rmN \norms{\xi}1\big)\,,
\end{align}
where \recnote1 and \recnote2 follow from \eqref{e:psi-minorant}, while
\recnote3 and \recnote 4 are due to \eqref{norm-control}, also taking into
account the monotonicity of $\unipsi$ and $C_\rmN \geq1$.  Then,
\eqref{crucial-bounds} follows by conjugation.
\end{proof}

Our last, crucial condition on the dissipation potentials $\calR_j$, and on
$\calR_\eff$ as well, is that they all satisfy the QYE (with estimate
\eqref{eq:QuYouEst} for $\calR_\eff$ involving the norms $\norm{\cdot}1$ and
$\norms{\cdot}1$, and the space $\Spx \eff$ to be understood as $\Spx 1$).

\begin{hypothesis}
\label{hyp:QYE}
We require that the potentials $\calR_j: \Spx j \to [0,\infty)$ for
$j \in\{1,2,\eff\}$ satisfy the Quantitative Young estimate $\mathbf{<QYE>}$.
\end{hypothesis}
Of course, the Quantitative Young estimate for $\cR_\eff$ is not independent of
the analogous requirements for $\cR_1$ and $\cR_2$. We refer to Appendix
\ref{s:QYE} for some results addressing this connection.

\subsubsection*{Existence results}
It follows from the Hypotheses listed above that the existence result
\cite[Thm.\,2.2]{MRS2013} applies to the (generalized) gradient systems
$(\Spx j, \calE, \calR_j)$, $j=1,2$. In particular, observe that, since weak
convergence in $\Spx 2$ implies weak convergence in $\Spx 1$, Hypothesis
\ref{h:2} guarantees (sequential) lower semicontinuity of $\calE$ and
continuity of $\partial_t \calE$ with respect to the weak topology of $\Spx 2$;
likewise, Hypothesis \ref{hyp:en-2} ensures (sequential) closedness of
$\frnameopt 2$ with respect to the weak-weak topology of $\Spx 2 \ti \Spx
2^*$. Therefore,  we conclude the existence of solutions to the Cauchy
problems for
\begin{equation}
\label{Cauchy-pb}
\partial \calR_j(u'(t)) + \frsubopt j t{u(t)} \ni 0  \text{ in } \Spx j^*
\  \text{ a.e.\ in } (0,T), \quad u(0)=u_0\in \domE_0,  
 \quad j \in \{1,2\}
\end{equation}
also satisfying the associated energy-dissipation balance. In fact, by Lemma \ref{l:Reff} the effective dissipation potential $\calR_\eff$
enjoys the same properties of $\calR_1$ and $\calR_2$, and we can likewise directly conclude an existence result for the Cauchy problem associated with 
\begin{equation}
\label{Cauchy-pb-eff}
\partial \calR_\eff(u'(t)) + \frsubopt 1 t{u(t)} \ni 0  \text{ in } \Spx 1^*  
\  \text{ a.e.\ in } (0,T), \quad u(0)=u_0\in \domE_0\, .  
\end{equation}
Theorem \ref{exist:gflows} below states the existence of solutions to the
Cauchy problems for \eqref{Cauchy-pb} and \eqref{Cauchy-pb-eff}.  In fact, the
main objective of the analysis carried out in the ensuing Sections \ref{se:TimeSplit}
and \ref{s:4} will be to demonstrate that a solution to \eqref{Cauchy-pb-eff}
can be constructed via the time-splitting method.  

\begin{theorem}
\label{exist:gflows}
Assume Hypotheses \ref{hyp:en-1}, \ref{h:2}, \ref{hyp:en-2},
\ref{h:chain-rule}, \ref{hyp:dis}, and \ref{hyp:QYE}.  Then, for every
$u_0\in \domE_0$ and each $j\in \{1,2,\eff\}$ the subdifferential inclusions
\eqref{Cauchy-pb} and \eqref{Cauchy-pb-eff} admit a solution
$u_j\in \AC([0,T];\Spx j)$  (where  $\Spx \eff = \Spx 1$)
satisfying $u_j(0)=u_0$, and indeed there exist measurable selections
$(0,T) \ni t \mapsto \xi_j(t) \in \frsubopt j t{u_j(t)} \in \Spx j^*$ with
$- \xi_j(t) \in \partial \calR_j(u_j'(t))$ for a.a.\ $t\in (0,T)$, such
that $(u_j,\xi_j)$ fulfills the energy-dissipation balance
\begin{equation}
\label{EDB-i}
\en t{u_j(t)} + \int_s^t \left(  \calR_j (u_j'(r)){+} \calR_j^*({-}\xi_j(r))
\right) \dd r  = \en s{u_j(s)} +\int_s^t \pet r{u_j(r)} \dd r   
\end{equation}
for every  $0 \leq s \leq t \leq T$.
\end{theorem}

It is important to mention  that our conditions on the generalized gradient
systems $(\Spx j,\calE, \calR_j)$ are slightly weaker than  those required
in \cite{MRS2013}. There, compactness of sublevels was required of the energy
functional and, accordingly, the lower semicontinuity and closedness conditions
were imposed along sequences converging with respect to the \emph{strong}
topology. Here, we work in the more general setup of Hypotheses \ref{h:2} and
\ref{hyp:en-2}: in fact, a close perusal of the proof of
\cite[Thm.\,2.2]{MRS2013} reveals that its arguments can be adapted to the
present setup,  see also \cite{MR21}.

\section{The time-splitting  approach}
\label{se:TimeSplit}

For the time-splitting method let us consider a (possibly non-uniform)
partition of the interval $[0,T]$
\begin{equation}
  \label{non-uniform-partition}
  \begin{aligned}
    & \mathscr{P}_{\bftau} = \{ \dis t \bftau 0 = 0 < \dis t \bftau 1< \ldots <
    \dis t \bftau k < \dis t \bftau{k+1} < \ldots < \dis t\bftau{N_\bftau} =
    T\}
    \\
    &  \text{with }  \tau_k = \dis t\bftau{k} {-} \dis t \bftau{k-1}
      \text{ and } \ttau : = \max\bigset{\tau_k}{ k=1,\ldots, N_\bftau}\,.
  \end{aligned}
\end{equation}
We also introduce the `left' and `right'  semi-intervals generated  by
$\mathscr{P}_{\bftau}$, namely
\begin{equation}
  \label{left-right-intervals}
  \lint{k}{\bftau} : =  \left(\dis t{\bftau}{k-1}, \dis t{\bftau}{k-1}+\tfrac
    {\tau_k}2 \right], \qquad  
  \rint{k}{\bftau}: = \left( \dis t{\bftau}{k-1}+\tfrac {\tau_k}2, \dis t
    \bftau k \right]\,. 
\end{equation}
 In what follows, we will use the short-hand 
 \begin{equation}
 \label{thalf}
 \thalf : = \dis t{\bftau}{k-1}+\tfrac {\tau_k}2= \tfrac12\big(\dis t{\bftau}{k-1} +
 \dis t{\bftau}{k}\big)  \qquad \text{for } k=1,\ldots, N_\bftau\,. 
\end{equation}
To simplify notation, we introduce the piecewise constant interpolants
associated with the nodes of the partition
\begin{equation}
\label{time-interpolants}
\begin{aligned}
& \pwc \sft \bftau: [0,T]\to [0,T],  \quad &&   \pwc \sft\bftau(0): = 0,  \quad
&&  \pwc \sft\bftau(t): = \dis t\bftau k  && \text{ for } t \in (\dis
t\bftau{k-1}, \dis t \bftau k]; 
\\
 & \upwc \sft \bftau: [0,T]\to [0,T], \quad   && \upwc \sft\bftau(T): = T,
 \quad  && \upwc \sft\bftau(t): = \dis t\bftau {k-1}  && \text{ for } t \in
 [\dis t\bftau{k-1}, \dis t \bftau k). 
\end{aligned}
\end{equation}
We will also use the notation
\begin{equation}
\label{tau-dependent-indexes}
\mt{\bftau}t  : = k   \  \text{ for }  t \in \lint{k}{\bftau} \cup
\rint{k}{\bftau} \quad \text{and} \quad \wt\tau(t)  :=\tau_{\mt{\bftau}t} 
\quad \text{for } t \in (0,T]\,,
\end{equation} 
 with $\wt\tau(0):=\tau_1$. 

\subsubsection*{Repetition operators}
A  key tool  for our analysis are the following operators,
  defined on the space $ \rmL^1([0,T];\Spgx) $ with  a (general) separable and
  reflexive Banach space   $\Spgx$:
\begin{subequations}
\label{repetition-operators}
\begin{align}
\TT^{(1)}_\bftau{:} &\  \rmL^1([0,T];\Spgx)  \to  \rmL^1([0,T];\Spgx); \ \ 
\big(\TT^{(1)}_\bftau g\big) (t) := \left\{ \ba{c@{\ }l} g(t)& \text{for } t
  \in \lint{\bftau}{\mt {\bftau}t}\,, 
 \\
   g\big(t{-}\tfrac{\wt\tau(t)}2\big)& \text{for }  t \in  \rint{\bftau}{\mt {\bftau}t}\,,
 \ea \right .
 \\
\TT^{(2)}_\bftau{:} &\   \rmL^1([0,T];\Spgx)  \to  \rmL^1([0,T];\Spgx); \ \ 
\big(\TT^{(2)}_\bftau g\big) (t) := \left\{ \ba{c@{\ }l} 
 g\big(t{+}\tfrac{\wt\tau(t)}2\big)& \text{for } t  \in \lint{\bftau}{\mt {\bftau}t}\,,
 \\
   g(t)& \text{for }  t \in  \rint{\bftau}{\mt {\bftau}t}\,.
 \ea \right .
\end{align}
\end{subequations}
We shall refer to  $\TT^{(1)}_\bftau$ and $\TT^{(2)}_\bftau$  as
\emph{repetition operators}, since $\TT^{(1)}_\bftau g$  repeats,  on
the `right'  semi-intervals, the values of the function $g$ from the `left'
semi-intervals,  while $\TT^{(2)}_\bftau g$ does the converse, see Figure
\ref{fig:Repetition} for an illustration.

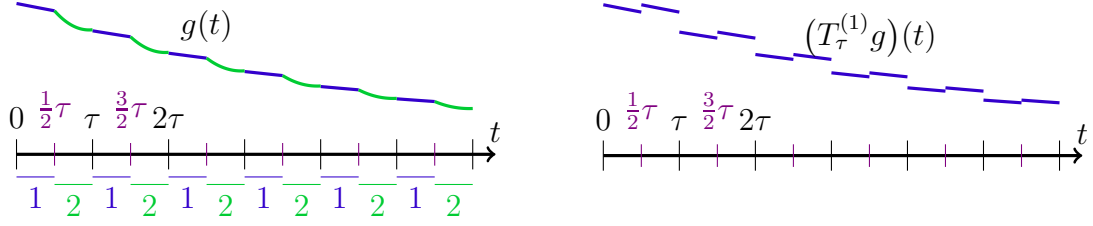
\begin{figure}
\centering
\begin{tikzpicture}
\draw[very thick, -> =triangle45] (0,0)--(6.3,0) node[above]{$t$};
\foreach \x in {0,...,6} \draw (\x,0.2)--(\x,-0.2);
\foreach \x in {0,...,5} \draw[blue!50!red] (0.5+\x,0.15)--(0.5+\x,-0.15);
\node at (0,0.2) [above]{$0$};
\node at (0.5,0.2) [above,blue!50!red]{$\frac12\tau$};
\node at (1,0.2) [above]{$\tau$};
\node at (1.5,0.2) [above,blue!50!red]{$\frac32\tau$}; 
\node at (2,0.2) [above]{$2\tau$};

\foreach \nnn in {0,...,5}
     \draw[domain=0:0.5,color=blue!80!red, very thick ] 
         plot ({\nnn+\x},{2*exp(-0.2*\nnn-0.1*\x))});

\foreach \nnn in {0,...,5}
    \draw[domain=0:0.5,color=green!80!blue, very thick ] plot 
        ({\nnn+\x+0.5},{(2-\x+1.3*\x*\x)*exp(-0.2*\nnn-0.1*(\x+0.5))});

\foreach \nnn in {0,...,5} \draw[color=blue!80!red] (\nnn,-0.3) 
                          -- node[pos=0.5,below]{1}  (\nnn+0.5,-0.3) ;

\foreach \nnn in {0,...,5} \draw[color=green!80!blue] (\nnn+.5,-0.4) 
                          -- node[pos=0.5,below]{2}  (\nnn+1.0,-0.4) ;

\node at (2.5,1.7) {$g(t)$};
\end{tikzpicture} 
\qquad
\begin{tikzpicture}
\draw[very thick, -> =triangle45] (0,0)--(6.3,0) node[above]{$t$};
\foreach \x in {0,...,6} \draw (\x,0.2)--(\x,-0.2);
\foreach \x in {0,...,5} \draw[blue!50!red] (0.5+\x,0.15)--(0.5+\x,-0.15);
\node at (0,0.2) [above]{$0$};
\node at (0.5,0.2) [above,blue!50!red]{$\frac12\tau$};
\node at (1,0.2) [above]{$\tau$};
\node at (1.5,0.2) [above,blue!50!red]{$\frac32\tau$}; 
\node at (2,0.2) [above]{$2\tau$};

\foreach \nnn in {0,...,5}
    \draw[domain=0:0.5,color=blue!80!red, very thick ] 
          plot ({\nnn+\x},{2*exp(-0.2*\nnn-0.1*\x))});

\foreach \nnn in {0,...,5}
         \draw[domain=0:0.5,color=blue!80!red, very thick ] 
              plot ({\nnn+\x+0.5},{2*exp(-0.2*\nnn-0.1*\x))});

\node at (3.5,1.6) {$\big(T^{(1)}_\tau g\big)(t)$};

\node[color=white] at (0,-0.8){$\ $};
\end{tikzpicture} 
\caption{Left: Schematic sketch of a function $g$ generated by a split-step approach
  with equidistant partition. Right: The repetition  operator
  $\TT^{(1)}_\tau$  
  selects the left semi-intervals and repeats them in the right
  semi-intervals.}\label{fig:Repetition} 
\end{figure}

The following result collects some properties of the operators
$\TT^{(j)}_\bftau $: though straightforwardly checked, they will play a key
role in our analysis. Ahead of the statement, we recall that the
convergence in the space $\rmC^0([0,T]; \Spgxw)$ is, by definition, given by
the convergence of $\rmC^0([0,T]; (\Spgx,d_\mathrm{weak}))$, where the metric
$d_\mathrm{weak}$ induces the weak topology on a closed bounded set of the
reflexive space $\Spgx$.
 
\begin{lemma}
\label{lem:repetition-operator}
For $j\in \{1,2\}$ we have  the following properties: 
\begin{enumerate}
\item For all $\bftau\in\Lambda$  we have  
\[
  \| \TT^{(j)}_\bftau \|_{\mathrm{Lin}} \leq 2\, , 
  \]
   where
  $\| \cdot \|_{\mathrm{Lin}}$ is the norm of
  $\mathrm{Lin} ( \rmL^1([0,T];\Spgx); \rmL^1([0,T];\Spgx))$.

\item $\TT^{(j)}_\bftau g \to g$ for every $g\in \rmL^1([0,T];\Spgx)$ as
  $|\bftau|\to 0$.
\item For any family $(g_{\bftau})_{\bftau} \subset \rmC^0([0,T]; \Spgxw)$,
   we have   
  \begin{align}
    \label{ContinuityT.}
 \hspace{-0,5cm}     g_{\bftau} \to g \text{ in } \rmC^0([0,T]; \Spgxw) \text{
   as } |\bftau|\to 
    0 \ \Longrightarrow \ \TT^{(j)}_\bftau g_{\bftau} \to g  \text{ in }
    \rmL^\infty ([0,T]; \Spgxw) \,. 
  \end{align}
\item For any family $(g_{\bftau})_{\bftau} \subset \rmL^1([0,T]; \Spgx)$, 
  we have, in the limit $|\bftau|\to 0 $,   
  \begin{align}
   \label{ContinuityT.b}
 \hspace{-0,5cm}  g_{\bftau} \weakto g \text{ in }  \rmL^1([0,T]; \Spgx) 
  \ \ \Longrightarrow \ \ \frac12 \left( \TT^{(1)}_\bftau {+}
     \TT^{(2)}_\bftau \right) (g_{\bftau}) \weakto g \text{ in }  \rmL^1([0,T];
   \Spgx)\,.  
 \end{align}
\end{enumerate}
\end{lemma}

The  proof of Lemma \ref{lem:repetition-operator}  will be carried out in
Appendix \ref{s:appB}.

\subsubsection*{The time-splitting algorithm}
 We recall the rescaled  dissipation potentials
$\tdis j: \Spx j \to [0,\infty)$  given by  
\begin{equation}
\label{tilde-dissipations}
\tdis j(v) : = 2 \calR_j\left( \tfrac 1{2} v \right) \ 
\text{ with conjugates } \ \tdis{j}^* (\xi) = 2 \calR_j^* (\xi) \quad \text{for
  } j\in \{1,2\}\,.
\end{equation}   
For $k \in \{ 1, \ldots, N_\bftau\}$, we  consider the Cauchy problems
associated with $(\Spx j,\calE,\tdis j)$: 
\begin{enumerate}
\item the Cauchy problem  for $(\Spx 1,\calE,\tdis 1)$  with  some
  initial datum $\underline{u} \in \domE_0$:
\begin{equation}
\label{Cauchy-pb-1}
\left\{ 
\begin{array}{r@{\:}c@{\:}l}
\partial \tdis 1(u'(t)) + \frsub t{u(t)}  &\ni& 0 \  \text{ in } \Spx 1^*
\quad \foraa\ t \in   \lint{k}{\bftau} \,,
\smallskip
\\
u(\dis t{\bftau}{k-1})&=& \underline{u} \,. 
\end{array}
\right.
\end{equation}
\item 
 the Cauchy problem  for $(\Spx 2,\calE,\tdis 2)$  with some
  initial datum $\ol{u} \in \domE_0$:

\begin{equation}
\label{Cauchy-pb-2}
\left\{ 
\begin{array}{r@{\:}c@{\:}l}
\partial \tdis 2(u'(t)) +  \frsub t{u(t)} &\ni & 0 \ \text{ in } \Spx 2^*
\quad \foraa\ t \in   \rint{k}{\bftau} \,,
\smallskip
\\
u ( \thalf )&=& \overline{u}\, ,
\end{array}
\right.
\end{equation}
(cf.\ \eqref{thalf} for the definition of $\thalf$).
\end{enumerate}
Thanks to Theorem \ref{exist:gflows}, both Cauchy problems admit a solution.
We are now in a position to detail the time-splitting algorithm.  Starting from
an initial datum $u_0 \in \domE_0$,  we recursively define the 
approximate  solution $\pws U\bftau: [0,T]\to \domE_0$ in the following
way:
\begin{equation}
\label{def-pws}
\begin{aligned}
  & \pws U\bftau(0): = u_0 \qquad \text{ and, for } t \in (\dis t \bftau {k-1},
  \dis t \bftau k ] \text{ with } k =1, \ldots, N_\bftau, \text{ we define}
  \hspace*{-2em} &&
  \\
  & \pws U\bftau \text{ as a solution of Cauchy problem \eqref{Cauchy-pb-1}
    with $\underline{u} = \pws U\bftau(\dis
    t\bftau{k-1})$} 
  && \text{on } \lint k{\bftau},
  \\
  & \pws U\bftau \text{ as a solution of Cauchy problem \eqref{Cauchy-pb-2}
    with
    $\overline{u} = \pws U\bftau\left(\dis t\bftau{k-1}{+}\tfrac{\tau_k}
      2\right)$} && \text{on } \rint k{\bftau}\,.
\end{aligned}
\end{equation}
By construction, we have $\pws U{\bftau} \in \AC([0,T];\Spx 1)$.

\subsubsection*{Our main convergence result}

We will prove the convergence of (a subsequence of) the family of curves
$(\pws U{\bftau} )_{\bftau}$ to a solution $U$ of the Cauchy problem for the
generalized gradient system \eqref{effective-dne}, under the additional
\emph{singleton condition} on $\frnameopt 2$. We mention in advance that the
fact that $\frsubopt 2tu$ is a singleton does not imply Fr\'echet
differentiability of $\calE$ at $(t,u)$, cf.\ e.g.\ the counterexample in
\cite[\S\,1.3, p.\,90]{Mordu-I}.

\begin{hypothesis}[Singleton condition]
\label{hyp:singleton}
We assume that 
\begin{equation}
\label{singleton}
\frsubopt  2tu \text{ is a singleton for all } (t,u)\in [0,T]\ti \domE_2.
\end{equation}
\end{hypothesis}
Obviously, due to \eqref{relation-subdiff}, Hypothesis \ref{hyp:singleton}
implies that, whenever it is non-empty, $\frsubopt 1tu$ is also a singleton and
coincides with $\frsubopt 2tu$.  Indeed, in Section\ \ref{ss:3.2} we will
present a counterexample to convergence of the split-step scheme, in the case
of an energy with multi-valued Fr\'echet subdifferentials,  i.e.\ the
singleton condition fails.  

Our main convergence result states the convergence of (a subsequence of) the
curves $(\pws U{\bftau} )_{\bftau}$, as $|\bftau| \downarrow 0$, to a solution
of the Cauchy problem for \eqref{Cauchy-pb-eff}.  We highlight that we even
have convergence of the `repeated velocities'
$(\TT^{(j)}_\bftau \pws U{\bftau}')_{\bftau}$, to the \emph{optimal} velocities
contributing to $\calR_\eff(U')$.

\begin{theorem}[Convergence of time-splitting method]
\label{th:conv-split-step}
In addition to the assumptions of Theorem \ref{exist:gflows}, assume  the
ordering condition \eqref{eq:ordering} and  the singleton condition of
Hypothesis \ref{hyp:singleton}.  Starting from an initial datum
$u_0 \in \domE_0$, define the curves $(\pws U{\bftau})_{\bftau \in \Lambda}$ as
in \eqref{def-pws}.  Then, for any sequence $(\bftau_n)_n$ with
$\lim_{n\to\infty} |\bftau_n | =0$ there exist a (non-relabeled) subsequence, a
curve $U \in \AC([0,T];\Spx 1)$, some $E>0$, and $V_j \in \rmL^1 ([0,T];\Spx
j)$ for $j =1,2$ such that $U(0)=u_0$, \  $U(t)\in\domE_1 \cap \subl E $ for all
$t\in [0,T]$, 
\begin{subequations}
\label{cvgs-Thm43}
\begin{align}
\label{pointwise-cvg}
&
 \pws U{\bftau_n}(t) \weakto U(t)  &&  \text{  in } \Spx 1  \quad \text{for
   all } t \in [0,T],  
 \\
 &  \label{cvg-repeated-velocities}
  \frac12  \TT^{(j)}_{\bftau_n}(\pws U{\bftau_n}') \weakto V_j  && \text{ in } \rmL^1([0,T];\Spx j) \quad \text{for } j=1,2,
 \\
 &  \label{cvg-U'}
  \pws U{\bftau_n}' \weakto U'  =V_1{+} V_2  && \text{ in } \rmL^1([0,T];\Spx 1),
 \end{align}
\end{subequations}
 and there exists a function
 $\xi \in\rmL^1([0,T];\Spx 1^*)$ such that the pair $(U,\xi)$ solves the subdifferential inclusion 
 \begin{equation}
 \label{solving-the-equn}
 \begin{cases}
 \displaystyle  \partial \calR_\eff (U'(t)) + \xi(t) \ni0 
 \\
  \displaystyle \frsubopt 1t{U(t)} = \{\xi(t)\} 
 \end{cases}
 \qquad \text{in } \Spx 1^* \quad  \text{for a.a.\ $t\in (0,T)$}
 \end{equation} 
and  fulfills the energy-dissipation balance, for every  $0 \leq s \leq t \leq T$,
\begin{equation}
\label{EDB-effect}
\en t{U(t)} + \int_s^t \!\!\big(  \calR_\eff (U'(r)){+} \calR_\eff^*({-}\xi(r)) 
 \big)   \dd r  = \en s{U(s)} +\int_s^t \!\! \pet r{U(r)} \dd r  \,.
\end{equation}

Moreover, $(V_1,V_2) $ provides an \emph{optimal decomposition} for $U'$, namely
\begin{equation}
 \label{optimal-decomposition}
 \left.
 \begin{array}{@{}c@{}}
  \Vopt 1(t) + \Vopt 2(t) = U'(t)
 \smallskip
 \\
 \calR_1(\Vopt 1(t)) {+}  \calR_2(\Vopt 2(t)) = \calR_\eff (U'(t)) =
 \min\limits_{v_1+v_2=v} \calR_1(v_1){+} \calR_2(v_2) 
 \end{array}
 \right\}
  \ \foraa\  t \in (0,T)\,.
\end{equation}	
\end{theorem}

The proof of Theorem \ref{th:conv-split-step} will be carried out in Section
\ref{s:4}. We only mention that, as suggested by Proposition \ref{p:EDP}, our
argument for proving that the curves $(\pws U{\bftau_n})_n$ converge to a
solution of \eqref{solving-the-equn} will be tightly related to the proof of
the energy-dissipation balance \eqref{EDB-effect}. As a by-product of this, in
Section\ \ref{s:4} we will obtain enhanced convergence information for
$(\pws U{\bftau_n})_n$, in addition to the convergences \eqref{cvgs-Thm43}. In
order to state these properties, we need to bring into play the two
integral terms that will encode the dissipative contributions in the
approximate version of \eqref{EDB-effect} (cf.\ \eqref{EDB-approx}
ahead). Their definition features the characteristic function of (the union of)
the  semi-intervals  $\lint{k}{\bftau}$, i.e.\
\begin{equation}
\label{charactetistic-functions}
\begin{aligned}
& 
\chi_ \bftau: (0,T)  \to \{0,1\}  && \quad  \chi_ \bftau(t): =  \left
\{ 
\begin{array}{ll}
1 & \text{if } t \in \lint{\mt{\bftau}t}{\bftau},
\\
0 & \text{otherwise.}
\end{array}
\right.
\end{aligned}
\end{equation}
 Recalling $\tdis j= 2\calR_j(\frac12\,\cdot)$ and $\tdis j^*=2\calR_j^*$, on
subintervals  $[s,t]\subset [0,T]$ we introduce
\begin{subequations}
\label{rate+slope-terms}
\begin{enumerate}
\item the \emph{rate term}   
\begin{equation}
\label{rate-term}
\mathcal{D}_{\bftau}^{\mathrm{rate}} ([s,t]) := 
 \int_{s}^t\big(\,
\chi_{\bftau} (r)\, \tdis 1 \left(\pws U{\bftau}'(r)\right) +
(1{-}\chi_{\bftau}(r) ) \, \tdis 2 \left(\pws U{\bftau}'(r)\right) \big) \dd r 
\end{equation}
\item and the \emph{slope term}
\begin{equation}
\label{slope-term}
\mathcal{D}_{\bftau}^{\mathrm{slope}} ([s,t]) : =    \int_{s}^t
\big(\, \chi_{\bftau}(r)\, \tdis1^* ({-} \xi_ {\bftau} (r))  +
(1{-}\chi_{\bftau}(r) )\,\tdis2^* ({-} \xi_ {\bftau} (r)) \,  \big)  \dd r \,.
\end{equation}
\end{enumerate}
\end{subequations}
\par Now, we are in a position to state the enhanced convergences
that we will succeed in proving for the sequence  $(\pws U{\bftau_n})_n$,
namely
\begin{subequations}
\label{enh-cvg}
\begin{align}
& 
\en t{\pws U{\bftau_n}(t)}  \longrightarrow 
\en t{U(t)} && \text{for all } t \in [0,T],
\\
&
\left\{ 
\begin{array}{ll}
\displaystyle \mathcal{D}_{\bftau_n}^{\mathrm{rate}} ([s,t])   \longrightarrow   \int_s^t \calR_\eff (U'(r)) \dd r, 
\\
\displaystyle 
\mathcal{D}_{\bftau_n}^{\mathrm{slope}} ([s,t])  \longrightarrow    \int_s^t 
 \calR_\eff^*({-}\xi(r)) \dd r
 \end{array}
 \right.
 && \text{for all } [s,t]\subset [0,T]\,.
 \end{align}
\end{subequations}
In fact, for the velocities we will even obtain the following, more precise,
convergence statement along any interval $[s,t]\subset [0,T]$
\begin{equation}
 \label{optimal-rates}
   \int_s^t \calR_j(\tfrac12 \TT^{(j)}_{\bftau_n}\pws U{\bftau_n}'(r))  
\longrightarrow  \int_s^t \calR_j(\Vopt j(r))  \dd r \qquad \text{for } j =1,2.
 \end{equation}

\section{Two  examples} 
\label{s:examples}

In the upcoming Section \ref{ss:3.2} we exhibit a counterexample to Theorem
\ref{th:conv-split-step} in the case in which the Fr\'echet subdifferential
does not comply with the singleton condition from Hypothesis
\ref{hyp:singleton}. In Section\ \ref{ss:5.2} we provide two gradient systems
$(\Spx 1, \calE, \calR_1) $ and $(\Spx 2, \calE, \calR_2) $ fulfilling all
assumptions of Theorem\ \ref{th:conv-split-step}.  

\subsection{Non-convergence for multi-valued $\pl\calE$}
\label{ss:3.2}

Following \cite[Section\,3.1]{Miel22RRIS} we consider the simple case 
$\Spx 1=\Spx 2=\R^2$  and a one-homogeneous energy
$\calE(t,u) =  \calE(u)=\max\{|u_1|,|u_2|\}$. Clearly, $\calE$ is
convex and its convex subdifferential is multi-valued, with 
\[
\pl\calE(u)= \begin{cases} 
  \{\pm (1,0)^\top\}&\text{for } \pm u_1 >|u_2|, \\
  \{\pm (0,1)^\top\}&\text{for } \pm u_2 >|u_1|, \\
  \set{\pm (\theta,1{-}\theta)^\top }{\theta\in [0,1]} 
          &\text{for } u_1=u_2, \ \pm u_1>0,\\ 
  \set{\pm (\theta,\theta{-}1)^\top }{\theta\in [0,1]} 
          &\text{for } u_1=-u_2, \ \pm u_1>0,\\ 
  [-1,1]\ti[-1,1] &\text{for } u=0.
              \end{cases}
\]
For a general dissipation potential of the form
\begin{equation}
\label{gen-disspot-EX}
\calR^*(\xi)= \frac a2 \xi_1^2 + \frac b2 \xi_2^2
\end{equation}
we can solve the gradient-flow equation for $(\R^2,\calE,\calR^*)$ and obtain a
contraction semigroup on the Hilbert space $\R^2$ (cf.\
\cite[Theorem\,3.1]{Brez73OMMS}).  We consider the solution with the initial
condition $u^0=u(0)=(2,1)^\top$, which leads to the piecewise affine solution 
\[
\begin{cases}
u(t)=\binom{2{-}at}{1}  &  \text{ for }t\in [0,\tfrac 1a], 
\smallskip
\\ 
u(t)=
(1-(t{-}\tfrac1a)\frac{ab}{a{+}b}) \binom11  & \text{ for }t\in {\big]\tfrac1a,\tfrac2a+\tfrac1b\big[},
\smallskip
\\
 u(t)=\binom00 &\text{ for } t\geq \tfrac2a+\tfrac1b.
 \end{cases}
\]  
Note that in the middle regime the solution satisfies $u_1=u_2>0$ and hence the
subdifferential $\pl\calE(u(t))$ is set-valued. This multi-valuedness is
necessary as the choice for $(\theta,1{-}\theta)\in \pl\calE(u)$ indeed depends
on $\calR^*$, namely $\theta=b/(a{+}b)$. 
For later use, it is nice to observe  that $- u_1'(t)$ only takes
three values, namely $a$, $ab/(a{+}b)$, and $0$.

We now apply the split-step algorithm for the two dissipation potentials 
\[
\calR_1^*(\xi)= \frac{a_1}2 \xi_1^2 + \frac{b_1}2 \xi_2^2 \quad \text{and} \quad 
\calR_2^*(\xi)= \frac{a_2}2 \xi_1^2 + \frac{b_2}2 \xi_2^2.
\]
Clearly, we have the effective potentials 
\[
\calR_\eff^*(\xi)=\frac a2 \xi_1^2 + \frac b2 \xi_2^2 \ \text{ and } \ 
\calR_\eff(\xi)=\frac 1{2a} \xi_1^2 + \frac 1{2b} \xi_2^2 \ \text{ with  }a =
a_1{+}a_2 \text{ and } b=b_1{+}b_2. 
\]
The three gradient systems $(\R^2,\calE, \calR^*_\eff)$,
$(\R^2,\calE, 2\calR^*_1)$, and $(\R^2,\calE, 2\calR^*_2)$ have the same form
as the gradient system $(\R^2,\calE, \calR^*)$ with the general potential
$\calR$ from \eqref{gen-disspot-EX}.  Hence, the obtained solutions are again
piecewise affine and $V:=- u'(t)$ only takes three values.

In the first regime $u_1(t)> |u_2(t)|$ we have the three velocities:
\[
\ts
\calR^*_\eff:\ V_\eff=\binom{a_1{+}a_2}0, \quad 2\calR^*_1: \ V_1=\binom{2a_1}0,  
\quad 2\calR^*_2: \ V_2=\binom{2a_2}0. 
\]
Hence, we observe that the split-step solutions $u_\tau$ oscillate between the
two velocities $V_1$ and $V_2$ in such a way that the weak limit of $- u_\tau'$ 
equals $\frac12(V_1{+}V_2)=\binom{a_1{+}a_2}0$. Hence, in this regime (where
$\pl\calE(u)$ is single-valued), we have convergence of $u_\tau$ to $u_\eff$ as
$V_\eff= \frac12(V_1{+}V_2)$. 

However, for $t\in {\big] \frac1a,\frac2a+\frac1b\big[}$ the situation is
different, because for $u_1(t)=u_2(t)>0$ the subdifferential is
multi-valued. The three velocities are
\[
 \ts
\calR^*_\eff:\ V_\eff=\frac{ab}{a{+}b}\binom11, \quad 2\calR^*_1: \
V_1 = \frac{2a_1b_1}{a_1{+}b_1} \binom11,  \quad 2\calR^*_2: \ 
V_2 = \frac{2a_2b_2}{a_2{+}b_2} \binom11. 
\]
Clearly, for $t\in {\big] \frac1a,\frac2a+\frac1b\big[}$ we obtain 
\[
\ts
u_\tau(t) \to u_{\lim}(t):=\binom11-(t{-} \tfrac1a)\frac{1} 2 \big(V_1{+}V_2\big), \quad
\text{while } u_\eff(t)= \binom11-(t{-} \tfrac1a)\,V. 
\]
We always have $|V|\geq \big|\frac12(V_1{+}V_2) \big| $ but in general with a
strict inequality, e.g.\ for $(a_1,b_1)=(1,3)$ and $(a_2,b_2)=(3,1)$ we have
$a=b=4$ and obtain $V_\eff = \binom22$ and $V_1=V_2= \binom{3/2}{3/2}$.

We observe that the effective solution has a higher speed and is reaching
$u(t)=0$ already at $t=1/4+1/2=0.75$. However, the split-step solution
$u_\tau$, which is actually not oscillating with $\tau$ because of $V_1=V_2$,
is slowed down as it reaches $u_\tau(t)=0$ only for
$t=1/4+2/3\approx 0.917$.

\begin{remark} In this example one can follow the limiting procedure in the 
energy-dissipation balance. One observes that the liminf estimate for the
velocities always works. However, because of the limiting rate being too
small, there is a true drop. With $t_1=1/a$ and $t_2=2/a+ 1/b$ and
$(a_1,b_1)=(1,3)$ and $(a_2,b_2)=(3,1)$ we have 
\[
\frac1{t_2{-}t_1} \int_{t_1}^{t_2}  \tdis{J_\tau(t)} (u_\tau')  \dd \tau =
\frac12\Big(\tdis1(V_1) + \tdis2(V_2) \Big) 
= \frac12\Big(\frac{a_1b_1}{a_1{+}b_1} + \frac{a_2b_2}{a_2{+}b_2}  \Big)
   = \frac34 \,.
\]
However, the limit $u_{\lim}$ of $u_\tau$ satisfies $  u_{\lim}' =V_1=V_2=
\frac32\binom11$    and 
\[
\frac1{t_2{-}t_1}\int_{t_1}^{t_2}\calR_\eff(u_{\lim}'  (t))\dd t =
\calR_\eff\big( V_j\big) =\frac18 |V_j|^2 =  \frac{9}{16}\ 
\lneqq \ \frac34 \,. 
\]
Such a drop cannot be recovered if we treat 
the rate part of the dissipation and the slope part of the dissipation 
separately,  as is done in  \cite{Sandier-Serfaty}, see also the discussion
in \cite[Sec.\,5.4]{Miel23IAGS}. 
The approach of EDP convergence as studied in \cite{DoFrMi19GSWE, 
MiMoPe21EFED, MiPeSt21EDPC} may be capable to pass to the limit as well, but
this lies outside the range of this paper. 
\end{remark}

\subsection{A doubly nonlinear PDE example}
\label{ss:5.2}
In this section we consider a doubly nonlinear PDE of Allen-Cahn type where the
above split-step method applies.  Let $\Omega\subset\R^{d}$ be a bounded
Lipschitz domain. On the
 \begin{subequations}
 \label{setup-nonl-ex}
 \begin{align}
 \label{spaces-ex}
&  \textbf{Banach spaces} &&   \Spx{1}=\rmL^{p}(\Omega) \text{ with } p \in
(1,p_d],  
 \ \text{ and } \   \Spx{2}=\rmH_{0}^{1}(\Omega)
 \intertext{with $p_d= \tfrac{2d}{d-2}$ for $d\geq 3$, and $p_d $ \emph{arbitrary} in 
   $(1,\infty)$ for $d\in \{1,2\}$, so that $\Spx 2 \subset \Spx 1$ densely and
   continuously, we consider the} 
\label{potentials-ex}
& \textbf{dissipation potentials}  && 
\cR_{1}(v):=\frac{1}{p}\|v\|_{\rmL^{p}(\Omega)}^{p} \ \text{ and } \ 
   \cR_{2}(v):=\frac{1}{2}\|\nabla v\|_{\rmL^{2}(\Omega)}^{2} \,,
\intertext{and the}
\label{energy-ex}
& \textbf{energy functional } 
&&  \cE:[0,T]\times\rmL^{p}(\Omega)\to(-\infty,\infty] \ \text{ defined via }
 \end{align}
\[
\cE(t,u) : =\left\{
\begin{array}{cl}
\! \int_{\Omega}\frac{1}{2}|\nabla u|^{2} {+} W(u)\dd x-\langle\ell(t),u\rangle_{\rmH_0^1(\Omega)} & \text{for }u\in\rmH_{0}^{1}(\Omega)\text{ and }W(u)\in\rmL^{1}(\Omega)\,,
\\
\infty & \text{otherwise}\,.
\end{array}
\right.
\]
\end{subequations}
In what follows, we will suppose that 
\begin{subequations}
\label{ass-ex-E}
\begin{equation}
\label{ass-ell}
\ell\in\mathrm{C}^{1}([0,T];\rmH^{-1}(\Omega)),
\end{equation}
and, following \cite[Sec.\,7]{RMS08}, we will require that 
$W: \R \to \R $ satisfies $ W\in\mathrm{C}^{2}(\R) $ and 
\begin{equation}
\label{ass-W}
\begin{aligned}
\exists\, C_{W,1},C_{W,2},C_{W,3}>0  \ 
 \exists\, s_p\in(1,\tfrac{p_d}{p_*}) \ \forall \, r\in\R:\quad 
\left\{ 
\begin{array}{ll}
& \! \! \! \! \!
 W''(r)\geq-C_{W,1}\,,
 \\
 & \! \! \! \! \!
 W(r)\geq-C_{W,2}\,,
 \\
 &\! \! \! \! \!  |W'(r)|\leq C_{W,3} (1{+} |r|^{s_{p}})\,,
 \end{array}
 \right.
 \end{aligned}
\end{equation}
\end{subequations}
where with $p^*$ is the dual exponent to $p$.

Theorem \ref{th:ex-4} below addresses the validity of the assumptions for
Theorem\ \ref{th:conv-split-step} in the context of the two gradient systems
$(\Spx 1, \calE, \calR_1) $ and $(\Spx 2, \calE, \calR_2) $. Observe
that the corresponding evolutionary equations are
\[
\begin{aligned}
&
(\Spx 1, \calE, \calR_1)\,:   &&  \quad   |u'|^{p-2}u' - \Delta u + W'(u) && = &&  \ell && \quad  \text{in } (0,T)\ti\Omega,
\\
& (\Spx 2, \calE, \calR_2)\,:   &&\quad  -\Delta u' - \Delta u +W'(u)  && =  && \ell && \quad  \text{in } (0,T)\ti\Omega\,.
\end{aligned}
\]

\begin{theorem}
\label{th:ex-4}
Under conditions \eqref{ass-ex-E}, the gradient systems
$(\Spx 1, \calE, \calR_1) $ and $(\Spx 2, \calE, \calR_2) $ from
\eqref{setup-nonl-ex} comply with Hypotheses \ref{hyp:en-1}--\ref{hyp:dis} and
Hypothesis \ref{hyp:singleton}. The QYE from Hypothesis \ref{hyp:QYE} holds if
 and only if  $p\leq 2$. In particular, for $p\leq 2$ the solution of
the gradient system $(\Spx 1,\calE,\calR_\eff)$ can be constructed by the
time-splitting method.
\end{theorem}
The proof will be carried out in the two following sections, starting with a
discussion of the properties of the energy functional $\calE$.

\subsubsection*{Properties of the energy}  
The lower bound on $W''$ required in \eqref{ass-W} will be used to derive
$\lambda$-convexity of $\calE$, the bound on $W$ shall provide the lower bound
on the energy, the third bound on $W$ will be exploited for proving the
closedness of the subdifferentials of $\calE$.

Indeed, it is immediate to check that $\calE$, whose proper domain
$\mathrm{dom}(\calE)$ is of the form $[0,T]\ti \domE_0$ with
$\domE_0 \subset \rmH_0^1(\Omega)$, is bounded from below and complies with the
power-control condition \eqref{gpower-control}. Since for every $E>0$ the
corresponding sublevel set $S_E$ (cf.\ \eqref{sublevel-sets}) is contained in a
bounded set in $\rmH_0^1(\Omega)$, we also immediately verify Hypothesis
\ref{h:2}.

The same arguments as in the proof of \cite[Lem.\,7.3]{RMS08} (cf.\ also
\cite[Rmk.\,7.4]{RMS08}) yield the following representations for the Fr\'echet
subdifferentials $\frnameopt j$ of $\calE(t,\cdot)$:
\[
\begin{aligned}
& \text{for } \Spx 1= \rmL^p(\Omega): && 
\frnameopt 1(t,u)=\begin{cases}\!
-\Delta u{+}W'(u){-}\ell(t) & \text{ if  } -\Delta u{+}W'(u){-}\ell(t) \in
\rmL^{p^*}(\Omega),\\ 
\emptyset & \text{ else};
\end{cases}
\\
& \text{for } \Spx 2= \rmH_0^1(\Omega): && 
\frnameopt 2(t,u)=\begin{cases}\!
-\Delta u{+}W'(u){-}\ell(t) & \text{ if  } -\Delta u{+}W'(u){-}\ell(t) \in
\rmH^{-1}(\Omega),\\ 
\emptyset & \text{ else}.
\end{cases}
\end{aligned}
\]
Hence, the \emph{singleton condition} from Hypothesis \ref{hyp:singleton} is
satisfied. As for the closedness requirement from Hypothesis \ref{hyp:en-2},
let $t_n \to t$ in $[0,T]$ and let us consider a sequence
$(u_{n})_{n} \subset S_E$ for some $E>0$, weakly converging to some $u$ in
$\rmL^p(\Omega)$.  Hence, $(u_n)_n$ is bounded in $\rmH_{0}^{1}(\Omega)$ and
thus $u_n \weakto u$ in $\rmH^1(\Omega)$, so that
$-\Delta u_{n}\rightharpoonup-\Delta u$ in $\rmH^{-1}(\Omega)$, and $u_n \to u$
strongly in $\rmL^{p_d-\eps}(\Omega)$ (as $p_d$ is the critical exponent from
the Rellich-Kondrachov theorem).  In particular, $u_{n}\to u$, and thus
$W'(u_{n})\to W'(u)$, a.e.\ in $\Omega$.  Furthermore,
$W'(u_{n})\leq C|u_{n}|^{s_{p}}$ a.e.\ in $\Omega$.  Since
$s_{p}<\frac{p_d}{p^*}$, by dominated convergence we conclude that
$W'(u_n) \to W'(u)$ in $\rmL^{p^*}(\Omega)$.  Also using that
$\ell(t_n) \to \ell(t) $ in $\rmH^{-1}(\Omega)$, we immediately conclude the
closedness of both subdifferentials $\frnameopt j$.

Finally, it was shown in \cite[Lem.\,7.3]{RMS08} that 
$\cE(t,\cdot)$ is $\lambda$-convex in $\rmL^1(\Omega)$, i.e.\
\begin{align}
\label{lambda-convexity}
&
\exists\,\lambda<0 \  \forall\, t \in [0,T] \ \forall\, u_0, u_1 \in \domE_0 
 \ \forall\, \theta \in [0,1]\, : 
\\
& \nonumber
\calE(t,u_\theta) \leq (1{-}\theta) \calE(t,u_0)+ \theta \calE(t,u_1) -
\frac\lambda 2 \theta(1{-}\theta) \|u_0{-}u_1\|_{\rmL^1(\Omega)}^2 \text{
  with } u_\theta= (1{-}\theta) u_0+\theta u_1
\end{align}
Then, $\cE(t,\cdot)$ are $\lambda$-uniformly convex in $ \Spx 1$ and in
$\Spx 2$ (namely, estimate \eqref{lambda-convexity} holds with
$\| \cdot\|_{\rmL^1(\Omega)}$ replaced by $\| \cdot\|_{\rmL^p(\Omega)}$ and
$\| \cdot\|_{\rmH_0^1(\Omega)}$, respectively,  and $\lambda$  suitably
adjusted).  Then, we are in a position to apply \cite[Prop.\,A.1]{MR21} and
conclude the validity of the chain rule property $\mathbf{<CR>}$ for
$(\calE, \frnameopt j)$.

\subsubsection*{Properties of the dissipation potentials $\calR_1$, $\calR_2$ 
and $\calR_\eff$}
Finally, we discuss the validity of Hypotheses \ref{hyp:dis} and \ref{hyp:QYE}
for $\calR_1$ and $\calR_2$.  Obviously, the dissipation potential $\calR_1$
and its conjugate $\calR_1^*$ comply with condition
\eqref{gsuperlinear-growth}; also $\calR_2$ has superlinear growth, since on
$\Spx 2= \rmH^1_0(\Omega)$ the function $\|\nabla \cdot\|_{\rmL^{2}(\Omega)}$
provides a norm equivalent to the $\rmH^1(\Omega)$-norm. Finally, we observe
that
 \[
   \cR_{2}^{*}(\xi)=\sup_{v\in\rmH_{0}^{1}(\Omega)}\left\{ \langle
     v,\xi\rangle-\frac{1}{2}\|\nabla v\|_{\rmL^{2}(\Omega)}^{2}\right\}
   =\frac{1}{2}\|\nabla v_{*}\|_{\rmL^{2}(\Omega)}^{2},
\]
where $v_{*}\in\rmH_{0}^{1}(\Omega)$ is the unique solution of
$\xi=-\Delta v_{*}$ in the $\rmH_{0}^{1} \ti \rmH^{-1}$ duality.  Hence,
\begin{equation}
\label{R2-star-example}
\exists\, c, C>0 \ \forall\, \xi \in \rmH^{-1}(\Omega){:} \ \ 
c \|\xi\|_{\rmH^{-1}(\Omega)}^2  \leq   \cR_{2}^{*}(\xi)=\frac{1}{2}\|\nabla(-\Delta)^{-1}\xi\|_{\rmL^{2}(\Omega)}^{2}  \leq C \|\xi\|_{\rmH^{-1}(\Omega)}^2.
\end{equation}
We now discuss the validity of the Quantitative Young Estimate
\eqref{eq:QuYouEst} for $\calR_\eff$,
\[
  \cR_{\eff}(u)=\min_{v\in\rmL^{p}(\Omega)} \left(
    \frac{1}{p}\|v{-}u\|_{\rmL^{p}(\Omega)}^{p} {+} \frac{1}{2}\|\nabla
    v\|_{\rmL^{2}(\Omega)}^{2} \right)
\]
(where we omit to detail that the term $\|\nabla v\|_{\rmL^{2}(\Omega)}^{2}$ is
replaced by $\infty$ if $v \in L^{p}(\Omega){\setminus} \rmH_0^1(\Omega)$, cf.\
\eqref{largest}). We will distinguish the cases $p\in (1,2)$, $p=2$ and $p>2$,
and show that the QYE holds if and only if $p\leq 2$.   For this, recall
$\Spx\eff= \Spx1=\rmL^p(\Omega)$. 

\paragraph{QYE for  $p \in (1,2)$.}
Using $\|\nabla u\|_{\rmL^{2}(\Omega)}\geq C \|u\|_{\rmL^{p_d}(\Omega)}$ 
on $\Spx2=\rmH^1_0(\Omega)$,  we have
\[
\cR_{2}(v)\geq\frac{1}{2}\|\nabla v\|_{\rmL^{2}(\Omega)}^{2}\geq \frac C2
\|v\|_{\rmL^{p_d}(\Omega)}^{2}. 
\]
Since $p_d>2>p$, we obtain 
\[
\exists\, \overline{c}>0  \ \forall\, v \in \rmH_0^1(\Omega): \quad
\cR_{2}(v)\geq \overline{C}\|v\|_{\rmL^{p_d}(\Omega)}^{2} \geq
\overline{c}\|v\|_{\rmL^{p}(\Omega)}^{2}  \,. 
\]
Therefore, $\calR_1$ and $\calR_2$ comply with condition
\eqref{overline-growth} of Lemma \ref{l:QYE-more general}, which guarantees the
validity of the QYE for $\calR_\eff$.

\paragraph{QYE for $p=2$.} In this case, 
\[
\cR_{\eff}^{*}(\xi) = \cR_{1}^{*}(\xi) + \cR_2^*(\xi) = \frac12 \|\xi\|^2_{\rmL^2(\Omega)} + \frac{1}{2}\|\nabla(-\Delta)^{-1}\xi\|_{\rmL^{2}(\Omega)}^{2} \quad \text{for all } \xi \in \Spx 1^* = \rmL^2(\Omega)\,.
\]
Hence, also for $\calR_\eff$ we have both a quadratic upper bound and a
quadratic lower bound. Therefore, we may apply Lemma \ref{l:2.2} and conclude
that $\calR_\eff$ complies with the QYE.

\paragraph{Failure of QYE for  $p>2$.} Below we will establish the following
two statements:
\begin{align*}
\text{(A) } & \exists\, v\neq 0\ \exists\, C_\rmA>0\ \forall\, \lambda >0: \quad \cR_\eff(\lambda
v) \leq C_\rmA \lambda^2 \,,
\\
\text{(B) } & \exists\, \big( \xi_n\big)_{n\in \N} \text{ in } \Spx1^*\ 
 \exists\, C_\rmB>0: \quad
\|\xi_n\|_{\rmL^{p^*}(\Omega)} \to \infty \ \text{ and } \ 
\cR_\eff^*(\xi_n) \leq C_\rmB\| \xi_n\|_{\rmL^{p^*}(\Omega)}^{p^*} \,.
\end{align*}

\noindent\underline{\emph{Step 1:} (A) and (B) imply that QYE does not hold.}
In (A) we can choose $\lambda=\lambda_n=\|\xi_n\|_{\rmL^{p^*}}^{p^*/2}$. Then,
(A) and (B) imply the upper bound
\begin{equation}
  \label{eq:Reff.bound}
  \cR_\eff(\lambda_n v) + \cR_\eff^*(\xi_n) \leq \big( C_\rmA + C_\rmB\big)
  \|\xi_n\|_{\rmL^{p^*}(\Omega)}^{p^*} \quad \text{for all } n\in \N. 
\end{equation}
However, QYE would imply the lower bound 
\[
  \cR_\eff(\lambda_n v) + \cR_\eff^*(\xi_n) \geq c\| \lambda_n v\|_{\rmL^p(\Omega)}
  \|\xi_n\|_{\rmL^{p^*}(\Omega)} -C = c\|v\|_{\rmL^p(\Omega)}\,  \|\xi_n\|_{\rmL^{p^*}(\Omega)}^{1+p^*/2}-C. 
\]
Because of $p>2$ we have $p^*=p/(p-1) \in (1,2)$ and hence $p^* \lneqq
1+p^*/2$. Since, by (B) we can take $\|\xi_n\|_{\rmL^{p^*}}$ arbitrarily large,
we see that the lower bound derived from QYE contradicts the upper bound
\eqref{eq:Reff.bound}. Hence, QYE is false. \smallskip

\noindent\underline{\emph{Step 2:} (A) holds.} We choose any $v \in
\rmL^p(\Omega) \cap \rmH^1_0(\Omega)$ with $v\neq 0$. Then, $\cR_\eff(\lambda
v) \leq \cR_2(\lambda v) = C_\rmA\lambda^2$ with $ C_\rmA =  \frac12\|\nabla
v\|^2_{\rmL^2(\Omega)}$.\smallskip 

\noindent\underline{\emph{Step 3:} (B)  holds.} We set $\xi_n(x) = n
\sin\big( n^{1-p^*/2} x_1\big)$ and observe
\[
\| \xi_n\|_{\rmL^{p^*}}^{p^*} = \int_\Omega n^{p^*} \big| \sin( n^{1-p^*/2}
x_1)\big|^{p^*} \dd x \geq n^{p^*} \int_\Omega  \big| \sin( n^{1-p^*/2}
x_1)\big|^{2} \dd x \geq \frac{n^{p^*}}2 \,\sqrt{|\Omega|}  \,,
\]
where we used $|\sin \alpha|^{p^*} \geq |\sin \alpha |^2$ because of $p^*\in
(1,2)$. 
Moreover, using $\xi_n = - \pl_{x_1} \Xi_n$ for $\Xi_n(x) = n^{p^*/2} \cos(
n^{1-p^*/2} x_1) $ we find 
\begin{align*}
\|\xi_n\|_{\rmH^1_0(\Omega)^*}&
 = \sup_{\|v\|_{\rmH^1_0}\leq 1} \int_\Omega \xi_n  \,v \dd x  
 = \sup_{\| v\|_{\rmH^1_0}\leq 1} \int_\Omega -\partial_{x_1}\Xi_n  \,v \dd x  
\\
&=  \sup_{\| v\|_{\rmH^1_0}\leq 1} \int_\Omega \Xi_n  \,\pl_{x_1} v \dd x  
\leq \| \Xi_n\|_{\rmL^2(\Omega)} \leq  n^{p^*/2}  \sqrt{|\Omega|} \,.
\end{align*}
With the above estimates, we arrive at
$ \cR_\eff^*(\xi_n)=\cR_1^*(\xi_n)+\cR_2^*(\xi_n) = \frac1{p^*} \|
\xi_n\|_{\rmL^{p^*}(\Omega)}^{p^*} + \frac12\|\xi_n\|_{\rmH^1_0(\Omega)^*}^2 \leq \wt C
\,n^{p^*}$, and (B) follows with $C_\rmB= 2\wt C/\sqrt{|\Omega|}$.

\section{Convergence proof of the time-splitting method}
\label{s:4}
Our argument for the proof of Theorem \ref{th:conv-split-step} is carried out
in the following steps, tackled in the upcoming Sections
\ref{ss:4.1}--\ref{ss:4.4}.  It follows the classical existence theory for
solutions to gradient flow equations, see the survey \cite{Miel23IAGS}. 
\begin{description}
\item[(1) A priori estimates and compactness:] From the
  energy-dissipation balances for the Cauchy problem \eqref{Cauchy-pb-1} on the
   semi-intervals  $(\lint{k}{\bftau})_{k=1}^{N_\bftau} $, and for
  \eqref{Cauchy-pb-2} on the  semi-intervals 
  $(\rint{k}{\bftau})_{k=1}^{N_\bftau} $, we will deduce an overall
  energy-dissipation balance satisfied on the interval $[0,T]$ by the curves
  $\pws U\bftau$ from \eqref{def-pws}.  Therefrom we will derive all the a
  priori estimates on the family $(\pws U\bftau)_\bftau$, cf.\ Proposition
  \ref{prop:aprio} ahead. Consequently, we will deduce suitable compactness
  properties for a sequence $(\pws U{\bftau_n})_n$.
\end{description}
We then pass to the limit in the energy-dissipation balance, by
separately addressing 
\begin{description}
\item[(2) the limit passage in the rate term]
  $\calD^\text{rate}_\bftau([0,T])$, where we use
  $\cR_\eff = \cR_1 \INFCONV \cR_2$, and
\item[(3) the limit passage in the slope term]
  $\calD^\text{slope}_\bftau([0,T])$, where we exploit the singleton condition \eqref{singleton}. 
\end{description}
With Steps (1)--(3) we will thus show that (along a subsequence) the curves
$(\pws U{\bftau_n})_n$ converge to a curve $U \in \AC([0,T];\Spx 1)$ for which
there exists $\xi \in \rmL^1([0,T];\Spx 1^*)$ such that the pair $(U,\xi)$
complies with the upper energy-dissipation estimate
\begin{equation}
\label{EDB-0T-lim}
\en T{U(T)} + \int_0^T \!\! \big(  \calR_\eff (U'(r)){+}
\calR_\eff^*({-}\xi(r)) \big) \dd r  \leq \en 0{U(0)} +\int_0^T \!\! \pet
r{U(r)} \dd r \,. 
\end{equation}
\begin{description}
\item[(4) Conclusion of the proof:] We will apply  energy-dissipation
  principle from  Proposition \ref{p:EDP} to conclude that 
  $(U,\xi)$ fulfills the subdifferential inclusion \eqref{solving-the-equn} and
  the energy-dissipation balance \eqref{EDB-effect}.  With a careful
  argument based on the limit passage from the approximate energy-dissipation
  balance to \eqref{EDB-effect}, we  then derive the optimal
  decomposition \eqref{optimal-decomposition} and the enhanced
  convergences \eqref{enh-cvg} and \eqref{optimal-rates}.
\end{description}

\subsection{Approximated energy-dissipation balance and a priori estimates}
\label{ss:4.1}

To state the approximate energy-dissipation balance we introduce a curve
$\xi_\bftau:[0,T]\to \Spx 2^*$ encompassing the force terms that appear in the
subdifferential inclusions \eqref{Cauchy-pb-1} \& \eqref{Cauchy-pb-2}. We will
separately define $\xi_\bftau$ on the sets
$\bigcup_{k=1}^{N_\bftau} \lint{k}{\bftau}$ and
$\bigcup_{k=1}^{N_\bftau} \rint{k}{\bftau}$ whose union gives $[0,T]$.  Recall
that, for every $k \in \{1,\ldots, N_\bftau\}$, on the  semi-interval 
$\lint{k}{\bftau}$ the curve $\pws U\bftau$ fulfills the energy dissipation
balance
\begin{subequations}
  \begin{align}
    \nonumber
    \en t{\pws U\bftau(t)} + \int_{\dis t \bftau {k-1}}^t \!\! \left( \tdis 1
      (\pws U{\bftau}'(r)){+} \tdis{1}^*({-} \xi_\bftau(r)) \right) \dd r & =
    \en{\dis t \bftau{k-1}}{\pws U{\bftau}(\dis t{\bftau}{k-1})} + \int_{\dis
      t\bftau{k-1}}^t \!\! \partial_t \en r{\pws U{\bftau}(r)} \dd r
    \\
    \label{en-diss-discr-1}
    & \quad \text{for } \dis t {\bftau}{k-1} \leq t \leq \thalf 
  \end{align}
(recall $\thalf = \dis t {\bftau}{k-1} + \frac{\tau_k}{2}$),  where
  $\xi_\bftau: \bigcup_{k=1}^{N_\bftau} \lint{k}{\bftau} \to \Spx 1^*$
  satisfies
 \begin{equation}
 \label{upwc-xi}
 \xi_{\bftau}(r) \in \frsubopt 1 {r}{\pws U{\bftau}(r)} \cap ({-}\partial \tdis
 1 (\pws U{\bftau}'(r))) \quad \foraa\ r \in    \bigcup_{k=1}^{N_\bftau}
 \lint{k}{\bftau}\,. 
  \end{equation}
  \end{subequations}
Likewise, 
on each  interval $\rint{k}{\bftau}$ an energy-dissipation balance involving the dissipation potential $\tdis 2$ holds, namely
 \begin{subequations}
 \begin{equation}
  \label{en-diss-discr-2}
 \begin{aligned}
 &
 \en t{\pws U\bftau(t)} + \int_{\thalf}^t 
 \left( \tdis 2 (\pws U{\bftau}'(r)){+} \tdis{2}^*({-} \xi_\bftau(r)) \right) \dd r 
 \\
 &  = \en{\thalf}{\pws U{\bftau}(\thalf)} + \int_{\thalf}^t \partial_t \en r{\pws U{\bftau}(r)} \dd r 
\quad \text{ for } \thalf \leq t \leq\dis t {\bftau}{k} \,,
 \end{aligned}
\end{equation}
where $\xi_{\bftau} : \bigcup_{k=1}^{N_\bftau} \rint{k}{\bftau} \to \Spx 2^*$
satisfies
\begin{equation}
  \label{pwc-xi}
  \xi_{\bftau}(r) \in \frsubopt 2 {r}{\pws U{\bftau}(r)} \cap ({-}\partial
  \tdis 2 (\pws U{\bftau}'(r))) \qquad \foraa\  r  \in \bigcup_{k=1}^{N_\bftau}
  \rint{k}{\bftau}\,. 
\end{equation}
\end{subequations}

Combining \eqref{en-diss-discr-1} and \eqref{en-diss-discr-2} we deduce the
overall energy-dissipation balance  satisfied by
the curves $\pwc U{\bftau}$, featuring the \emph{rate} and \emph{slope} terms
$\calD_{\bftau}^{\mathrm{rate}} $ and $ \calD_{\bftau}^{\mathrm{slope}} $ from
\eqref{rate+slope-terms}, which are defined by alternating between $\tdis1$ and
$\tdis 2$. 
This energy balance is the starting point for the derivation of the first set
of a priori estimates on the curves $(\pws U\bftau)_{\bftau \in \Lambda}$. 

\begin{proposition}
\label{prop:aprio}
The functions 
$(\pws U\bftau)_{\bftau \in \Lambda}$ and  $(\xi_\bftau )_{\bftau \in \Lambda}$ satisfy the energy-dissipation balance
\begin{equation}
\label{EDB-approx}
\begin{aligned}
& 
\en t{\pws U\bftau(t)}   + \mathcal{D}_{\bftau}^{\mathrm{rate}} ([s,t])  + \mathcal{D}_{\bftau}^{\mathrm{slope}} ([s,t])   
 = 
\en s{\pws U\bftau(s)} +\int_s^t \partial_t 
\en r{\pws U\bftau(r)} \dd r  
\end{aligned}
\end{equation}
 for every $[s,t]\subset [0,T]$. 
Furthermore, there exists a positive constant $\overline{C}>0$ such that the following estimates are valid for all $\bftau \in \Lambda$:
\begin{subequations}
\label{aprio-est}
\begin{align}
\label{aprio-est-en+pow}
&
\sup_{t\in [0,T]} \mfE(\pws U\bftau(t))  \leq \overline{C},  \qquad \qquad  \sup_{t\in [0,T]}| \pet t{\pws U\bftau(t)} | \leq \overline{C}, 
\\
\label{aprio-diss}
& 
 \mathcal{D}_{\bftau}^{\mathrm{rate}} ([0,T])   \leq \overline{C}, 
\\
&
\label{aprio-conj}
 \mathcal{D}_{\bftau}^{\mathrm{slope}} ([0,T]) \leq \overline{C}, 
\end{align}
\end{subequations}
and the families
$\big(\chi_\bftau \pws U{\bftau}' \big)_{\bftau \in \Lambda}\subset
\rmL^1([0,T]; \Spx 1)$ and
$\big( (1{-}\chi_\bftau) \pws U{\bftau}' \big)_{\bftau \in \Lambda}\subset
\rmL^1([0,T]; \Spx 2)$ are uniformly integrable; in particular,
$(\pws U{\bftau}')_{\bftau \in \Lambda}\subset \rmL^1([0,T]; \Spx 1)$ is
uniformly integrable.
\end{proposition}
\begin{proof} The energy-dissipation balance \eqref{EDB-approx} follows on
$[0,t]$ simply adding \eqref{en-diss-discr-1} and \eqref{en-diss-discr-2}
over all relevant intervals.  By subtracting the result for $[0,s]$ from that
for $[0,t]$ the desired result for $[s,t]$ follows.

Estimates \eqref{aprio-est} follow from standard arguments (cf., e.g.,
\cite[Prop.\,6.3]{MRS2013}), which rely on the power-control estimate
\eqref{gpower-control} giving
$\int_s^t \pl_t \en r{\pws U\bftau(r)} \dd r \leq C_\# \int_s^t \en r{\pws
  U\bftau(r)} \dd r$. Hence, via Gr\"onwall's lemma, from \eqref{EDB-approx} we
derive the energy estimate in \eqref{aprio-est-en+pow}; the power estimate
immediately follows via \eqref{gpower-control}.  From \eqref{EDB-approx} we
then immediately conclude \eqref{aprio-diss} and \eqref{aprio-conj}.

  From \eqref{aprio-diss}, taking into account that the terms
  $\calR_j \left(\tfrac12 \pws U{\bftau}'\right) $ contribute to
  $\mathcal{D}_{\bftau}^{\mathrm{rate}}$ and that each $\calR_j$ are
  superlinear growth, we deduce that the families
  $ (\chi_\bftau\pws U\bftau')_{\bftau \in \Lambda}$ and
  $ ((1{-}\chi_\bftau)\pws U\bftau')_{\bftau \in \Lambda}$ are uniformly
  integrable in $\rmL^1([0,T];\Spx 1)$ and $\rmL^1([0,T];\Spx 2)$,
  respectively.
\end{proof}

It is now convenient to rewrite the `rate' and `slope' terms featuring in
\eqref{EDB-approx} in terms of the repetition operators $ \TT^{(j)}_{\bftau}$
from \eqref{repetition-operators}.  Their role, in the context of the present
time-splitting scheme, is now clear: $\TT^{(1)}_\bftau$ repeats ``$1$-steps''
and omits ``$2$-steps'', while $\TT^{(2)}_\bftau$ does the converse. Trivial
calculations based on the definition of the repetition operators identify the
contributions to $ \mathcal{D}_{\bftau}^{\mathrm{rate}} $ and
$\mathcal{D}_{\bftau}^{\mathrm{slope}} $ with quantities involving the
`repeated rates' and the `repeated forces', namely
 \begin{equation}
 \label{useful-identities}
\begin{aligned}
   &  \int_0^{t^k_\bftau} 
\chi_{\bftau}^{(j)} (r) \:  \tdis j  \left(V(r)\right)  \dd r = 
\int_0^{t^k_\bftau}  \calR_j(\tfrac12  \TT^{(j)}_\bftau V(r)) \dd r
&& \text{for } V \in \rmL^1([0,T]; \Spx j)\,, 
\\
   &  \int_0^{t^k_\bftau}
 \chi_{\bftau}^{(j)} (r) \:  \tdis j^*  (\Xi (r)) \dd r = 
 \int_0^{t^k_\bftau}  \calR_j^*(\TT^{(j)}_\bftau 
 \Xi(r)) \dd r  && \text{for } \Xi \in \rmL^1([0,T]; \Spx j^*)\,,
 \end{aligned}
 \end{equation}
where we have used the place-holder
 \[
 \chi_{\bftau}^{(1)} := \chi_{\bftau}  \quad \text{and} \quad 
   \chi_{\bftau}^{(2)} :=  1{-}\chi_{\bftau} \,. 
 \]
Therefore, the rate and slope parts of the dissipation take the form 
\begin{subequations}
\label{rephrasing-via-repet}
\begin{align}
&
\label{rephrasing-via-repet-a}
 \mathcal{D}_{\bftau}^{\mathrm{rate}} \big( [0,t^k_\bftau]  \big)=
  \int_0^{t^k_\bftau} \left\{  \calR_1(\tfrac12
   \TT^{(1)}_{\bftau} \pws U{\bftau}'(r)) {+} 
   \calR_2(\tfrac12  \TT^{(2)}_{\bftau}\pws U{\bftau}'(r)) \right\} \dd r\,, 
 \\
 & 
 \label{rephrasing-via-repet-b}
 \mathcal{D}_{\bftau}^{\mathrm{slope}} \big( [0,t^k_\bftau]  \big) =
  \int_0^{t^k_\bftau}  \left\{  \calR_1^*({-} \TT^{(1)}_{\bftau} \pws \xi{\bftau}(r)) {+}
     \calR_2^*({-}\TT^{(2)}_{\bftau} \pws \xi{\bftau} (r))  
 \right\} \dd r\,. 
\end{align}
\end{subequations} 
 We stress that, in \eqref{rephrasing-via-repet-a} the terms $\TT^{(j)}_{\bftau} \pws U{\bftau}'$ are the `repeated rates' $\TT^{(j)}_{\bftau} (\pws U{\bftau}')$, not to be confused with the rates of the repeated curves $\TT^{(j)}_{\bftau}\pws U{\bftau}$. 
As  a straightforward consequence of estimates \eqref{aprio-diss} and \eqref{aprio-conj}, combined with \eqref{rephrasing-via-repet} and the superlinear growth of $\calR_j$ and
$\calR_j^*$,  we have the following
\begin{corollary}
\label{cor:new-aprio}
For $j \in \{1,2\}$ the families
$( \TT^{(j)}_{\bftau} \pws U{\bftau}')_{\bftau \in \Lambda}\subset
\rmL^1([0,T]; \Spx j)$ and
$( \TT^{(j)}_{\bftau} \pws \xi{\bftau})_{\bftau \in \Lambda}\subset
\rmL^1([0,T]; \Spx j^*)$ are uniformly integrable.
\end{corollary}

Relying on Proposition \ref{prop:aprio} and Corollary \ref{cor:new-aprio} we
obtain the following compactness result. In \eqref{cvg-1} below we refer to the
convergence in the space $ \rmC^0([0,T];\Spxw 1)$, whose meaning has been
specified prior to the statement of Lemma \ref{lem:repetition-operator}.

\begin{corollary}
\label{cor:compactness}
Let $(\bftau_n)_{n} \subset \Lambda$ fulfill
$\lim_{n\to\infty} |\bftau_n | =0$. Then, there exist a (non-relabeled)
subsequence and a limit curve $U \in \AC ([0,T];\Spx 1)$  such that the
following convergences hold as $n\to\infty$:
\begin{subequations}
\label{convergences}
\begin{align}
& 
\label{cvg-1}
\pws U{\bftau_n} \to U   && \text{in }  \rmC^0([0,T];\Spxw 1), 
\\
& 
\label{cvg-2}
\pws U{\bftau_n}' \weakto U' && \text{in }  \rmL^1 ([0,T];\Spx 1), 
\end{align}
\end{subequations}
and 
\begin{subequations}
\label{cv:en+power}
\begin{align}
\label{cv-en1}
&
\liminf_{n\to \infty} \en t{\pws U{\bftau_n} (t)} \geq \en t{U(t)} && \text{for all } t \in [0,T],
\\
&
\label{cv-power1}
\partial_t \en t{\pws U{\bftau_n} (t)}  \to \partial_t \en t{U (t)}   && \text{for all } t \in [0,T].
\end{align}
\end{subequations} 
 Furthermore,  for $j \in \{1,2\}$ there exist $V_j \in \rmL^1([0,T];\Spx j)$  and  $\xforcel_{j} \in \rmL^1([0,T];\Spx j^*)$ 
such that, up to a further subsequence, we have as $n\to\infty$
\begin{subequations}
\label{cvg-for-the-halves}
\begin{align}
\label{cvg-for-the-halves-velocities}
&
\frac12  \TT^{(j)}_{\bftau_n} \pws U{\bftau_n}' \weakto V_j  &&   \text{ in }\rmL^1([0,T];\Spx j),
\\
& 
\label{cvg-for-the-halves-forces}
 \TT^{(j)}_{\bftau_n} \pws \xi{\bftau_n} \weakto \xforcel_j   &&  \text{ in }\rmL^1([0,T];\Spx j^*),
 \end{align}
\end{subequations}
and there holds
\begin{equation}
\label{crucial-appB}
V_1(t)+V_2(t) = U'(t) \qquad \foraa\  t \in (0,T). 
\end{equation}
\end{corollary}
\begin{proof}
Convergence \eqref{cvg-2} follows from the uniform integrability of the
family $(\pws U{\bftau}')_{\bftau \in \Lambda}\subset \rmL^1([0,T]; \Spx 1)$,
while \eqref{cvg-1} ensues, e.g., from the Arzel\`a-Ascoli compactness type
result in \cite[Prop.\,3.3.1]{AGS08}. Then,
 the energy and power convergences \eqref{cv:en+power} follow from
condition \eqref{h:2.1}.

Now, Corollary \ref{cor:new-aprio} ensures that, up to a subsequence, the
sequences $ \big(\tfrac12 \TT^{(j)}_{\bftau_n} \pws U{\bftau_n}'\big)_n$ and
$( \TT^{(j)}_{\bftau_n} \pws \xi{\bftau_n})_n$ have a weak limit in
$\rmL^1([0,T];\Spx j)$ and $\rmL^1([0,T];\Spx j^*)$, respectively.  Relation
\eqref{crucial-appB} follows from combining convergence \eqref{cvg-2} with item
(4) in Lemma \ref{lem:repetition-operator}. 
\end{proof}
 
\noindent In the following sections we will address the passage to the limit in
the `rate term' from \eqref{rate-term} and the `slope term' from
\eqref{slope-term}.

\subsection{Liminf estimate for the  rate term}
\label{ss:vel}
We are going to prove the following
\begin{equation}
\label{liminf-rate}
\text{\bfseries Claim $1$:}  \qquad \liminf_{n\to\infty}
\mathcal{D}_{\bftau_n}^{\mathrm{rate}} ( [0,T] )    \geq \int_0^T 
\calR_\eff (U'(r))  \dd r\,. 
\end{equation}
Indeed, 
\begin{equation}
\label{quoted-later}
\begin{aligned}
\liminf_{n\to\infty} \mathcal{D}_{\bftau_n}^{\mathrm{rate}} ([0,T])  &
\stackrel{(1)}{=} \liminf_{n\to\infty}  \int_0^T   \left\{
  \calR_1(\tfrac12  \TT^{(1)}_{\bftau_n} \pws U{\bftau_n}'(r)) {+}
  \calR_2(\tfrac12  \TT^{(2)}_{\bftau_n}\pws U{\bftau_n}'(r)) \right\} \dd r
\\
& \stackrel{(2)}{\geq }  \int_0^T   \left\{  \calR_1(V_1(r)){+}
  \calR_2(V_2(r)) \right\} \dd r 
 \stackrel{(3)}{\geq }  \int_0^T   \calR_{\eff} (U'(r)) \dd r\,,
\end{aligned}
\end{equation}
where {\footnotesize (1)} follows from \eqref{rephrasing-via-repet-a},
{\footnotesize (2)} is due to convergences
\eqref{cvg-for-the-halves-velocities} and the convexity and lower
semicontinuity of $\calR_j$, while {\footnotesize (3)} follows by property
\eqref{crucial-appB} and the definition of $\calR_\eff$  as an infimal
convolution. Hence, Claim 1 is established.

\subsection{Liminf estimate for the  slope term}
\label{ss:slope}

\textbf{Claim $2$:} \emph{There exists  $\xi  \in \rmL^1([0,T];\Spx 1^*)$ such that}
\begin{subequations}
\label{props-forces}
\begin{align}
\label{in-subdiff}
&
\frsubopt 1 {t}{U(t)} = \frsubopt 2 {t}{U(t)} = \{ \xi(t) \} \quad \foraa\ t \in (0,T), 
\\
\label{liminf-slope}
&
\liminf_{n\to\infty} \mathcal{D}_{\bftau_n}^{\mathrm{slope}}  ([0,T])  
 \geq  \int_0^T  \left\{  \calR_1^* ({-} \xi (r))  {+}   
\calR_2^* ({-} \xi (r))  \right\}  \dd r  \,.
\end{align}
\end{subequations}

Clearly, recalling \eqref{rephrasing-via-repet-b} and convergences
\eqref{cvg-for-the-halves-forces} we immediately have, by the convexity and
lower semicontinuity of $\calR_j^*$,
\begin{equation}
\label{lsc-already-here}
\begin{aligned}
 \liminf_{n\to\infty} \mathcal{D}_{\bftau_n}^{\mathrm{slope}} ([0,T])  &  = 
 \liminf_{n\to \infty}   \int_0^T   \left\{  \calR_1^*({-} \TT^{(1)}_{\bftau_n} \pws \xi{\bftau_n}(r)) {+}  \calR_2^*({-}\TT^{(2)}_{\bftau_n} \pws \xi{\bftau_n} (r)) 
 \right\} \dd r \\ & \geq   \int_0^T 
\left\{  \calR_1^* ({-}  \xforcel_1 (r))   {+}   \calR_2^* ({-} \xforcel_2
  (r))    \right\}  \dd r  \,,
\end{aligned}
\end{equation}
where $\xforcel_j \in \rmL^1([0,T];\Spx j^*)$  is the weak limit  of
the sequences $( \xforce{\bftau_n} j)_n$, cf.\ Corollary \ref{cor:compactness}.
In the following lines we demonstrate that $\xforcel_1$ and $\xforcel_2 $
coincide by resorting to the `singleton condition' from Hypothesis
\ref{hyp:singleton}.  For this we use  a \emph{Young measure} argument.

With this aim, it will be convenient to introduce the  `repeated curves'
\begin{equation}
\label{u-1}
 \xcur{\bftau}1: [0,T] \to  \domE_0,  
 \qquad  \xcur{\bftau}2: [0,T] \to  \domE_0,   
\end{equation}
with the convention that $\xcur{\bftau}1(t=0) = \xcur{\bftau}2(t=0) =
u_0$.  Note that $\xcur{\bftau}j$ may no longer be continuous, but  we
immediately observe that, since $\pws U{\bftau_n}\to U$ in
$\mathrm{C}^0([0,T];\Spxw 1)$, applying the continuity 
of the repetition operators (cf.\ item (3) in Lemma
\ref{lem:repetition-operator})  we have
\begin{equation}
\label{uniqueLimitU}
\xcur{\bftau_n}j \to U \text{ in }  \rmL^\infty  
 ([0,T];\Spxw 1)  \quad \text{for } j\in\{1,2\}.
\end{equation}
 By construction and using (\ref{upwc-xi} and \ref{pwc-xi}), it now 
follows that
\begin{equation}
\label{subdiff-prop}
\frsubopt j t{\xcur{\bftau_n}j (t)} = \{ \xforce{\bftau_n} j (t)  \} 
\quad \foraa\ t \in (0,T), \qquad \text{for } j \in \{1,2\}.
\end{equation}
Relying on \eqref{subdiff-prop} we will infer further information on the limits
$\xforcel_j$. 

For this, we resort to a Young-measure compactness result,
\cite[Thm.\,3.2]{RossiSavare06}, which states that, up to a (non-relabeled)
subsequence, the sequences $(\xforce{\bftau_n}1 )_n$ and
$(\xforce{\bftau_n}2 )_n$ admit two limiting Young measures
$(\mu_t^j)_{t\in (0,T)}$, with $\mu_t^1 \in \mathrm{Prob}(\Spx 1^*)$ and
$\mu_t^2 \in \mathrm{Prob}(\Spx 2^*) $ for a.a.\ $t\in (0,T)$, enjoying the
following properties for $j\in \{1,2\}$:
 \begin{subequations}
 \label{props-YM}
 \begin{enumerate}  
 \item the supports of the measures $\mu_t^j$ are contained in the set of the
   limit points of the sequences $(\xforce{\bftau_n}j(t))_n$ in the weak
   topology of $\Spx j^*$, i.e.\  for a.a.\ $t \in (0,T)$ we have 
   \begin{equation}
     \label{props-YM-1}
     \mathrm{supp}(\mu_t^j) \subset  \Ls {{\Spx j^*}}{ \{ \xforce{\bftau_n}j(t)
       \}_n  } :=\bigcap_{k\geq 1} 
     \overline{\{\xforce{\bftau_l}j(t):l\geq k\}}^\mathrm{weak} \,, 
   \end{equation}
   where the notation refers to the notion of $\limsup$ in the sense of the
   Kuratowski convergence of sets, cf.\ e.g.\ \cite{Att84VCFO};
 \item the weak limits $\xforcel_j$  of $(\xforce{\bftau_n}j)_n$ in
   $\rmL^1([0,T];\Spx j^*)$ coincide with the barycenters of the measures
   $\mu_t^j$, namely
   \begin{equation}
     \label{props-YM-2}
     \xforcel_j(t)   = \int_{X^*} \zeta \dd \mu_t^j(\zeta) \quad \foraa\ t
     \in (0,T)\,. 
   \end{equation}
 \end{enumerate}
\end{subequations}
It turns out that $\mathrm{supp}(\mu_t^j)$ is a singleton for $j=1,2$.  Indeed,
using the convergence of $\xcur{\bftau_n}j \to U$ in
$  \rmL^\infty  ([0,T]; \Spx {1,\mathrm{w}})$ by (\ref{uniqueLimitU}), the
subdifferential inclusions \eqref{subdiff-prop}, and the closedness property
\eqref{h:2.2}, we have that $ \Ls {{\Spx j^*}}{ \{ \xforce{\bftau_n}j(t) 
 \}_n } \subset \partial^{\Spx j}\cE( t,U(t))$ for a.a.\ $t\in (0,T)$
and $j \in \{1,2\}$. From \eqref{props-YM-1} we then infer 
\[
\mathrm{supp}(\mu_t^1) = \mathrm{supp}(\mu_t^2) =  
\Ls {{\Spx j^*}}{ \{ \xforce{\bftau_n}j(t)  \}_n  }  = 
\frsubopt j t{U(t)}= \{ \xi(t) \} \quad \foraa\ t \in (0,T).
\]
By \eqref{props-YM-2} we then conclude that $ \xforcel_1 = \xi = \xforcel_2$.
Then, the  liminf estimate \eqref{liminf-slope} follows from the
previously observed \eqref{lsc-already-here},  and Claim 2 is established.

\subsection{Conclusion of the proof  of Theorem \ref{th:conv-split-step}}
\label{ss:4.4}

Relying on convergences \eqref{convergences}, on the lower semicontinuity and
continuity properties of $\cE$ and $\partial_t \cE$ (cf. Hypothesis \ref{h:2}),
 and on the lower semicontinuity estimates \eqref{liminf-rate} from Claim
$1$ and \eqref{props-forces} from Claim $2$, we are in a position to take the
limit in the energy-dissipation balance \eqref{EDB-approx}, written on the
interval $[0,T]$, and thus conclude the validity of the energy-dissipation
inequality \eqref{EDB-0T-lim} along a curve $U \in \AC([0,T];\Spx 1)$ and
$\xi \in \rmL^1 ([0,T];\Spx 1^*)$.  By lower semicontinuity, we immediately
have $\sup_{t\in [0,T]} \mfE(U(t)) \leq \overline{C}$ with $\overline C>0$ from
\eqref{aprio-est-en+pow}, and thus
$\sup_{t\in [0,T]} |\partial_t \calE(t,U(t)) |  \leq C $.  Therefore,
from \eqref{EDB-0T-lim} we infer that
\[
 \int_0^T \! \big(  \calR_\eff (U'(r)){+} \calR_\eff^*({-}\xi(r)) \big) \dd r <\infty\,.
\]
Then, the quantitative Young estimate for $\calR_\eff$ yields
$\int_0^T \norm{U'(t)}1 \norms{\xi(t)}1 \dd t <\infty$.  Thus, we are in a
position to apply the chain rule for $(\calE, \frnameopt 1)$ from Hypothesis
\ref{h:chain-rule} and deduce that $t \mapsto \en t{U(t)}$ is absolutely
continuous on $[0,T]$, and that the chain rule formula \eqref{eq:48strong}
holds for the pair $(U,\xi)$. Then, Proposition \ref{p:EDP} allows us to
conclude that $(U,\xi)$ solves \eqref{solving-the-equn} and fulfills the
energy-dissipation balance \eqref{EDB-effect}.  

Now, it remains to show property \eqref{optimal-decomposition}, i.e.\ the weak
limits $V_j$ of $\big( \tfrac12 \TT^{(j)}_{\bftau_n}\pws U{\bftau_n}'\big)_n$
provide an \emph{optimal decomposition} of $U'$ in the sense that 
$V_1+V_2=U'$ as well as $\calR_1(V_1)+ \calR_2(V_2)= \calR_\eff(U')$ a.e.\ in $(0,T)$.

To see this, we first observe that the limit passage from the approximate
energy-dissipation balance \eqref{EDB-approx} to the limit energy-dissipation
balance \eqref{EDB-effect} on the interval $[0,T]$ ultimately implies that the
liminf estimates \eqref{liminf-rate} and \eqref{liminf-slope} turn into
convergences (see e.g.\ \cite[Thm.\,4.4]{MRS2013} or \cite[Thm.\,3.11]{MRS13}
for the standard argument).  In particular, combining the convergences
for the rate term with \eqref{quoted-later} we conclude 
\[
 \int_0^T  \!\! \calR_\eff(U'(r))  \dd r 
= \lim_{n\to\infty}\mathcal{D}_{\bftau_n}^{\mathrm{rate}}  ([0,T]) 
\geq  \int_0^T \!\! \{ \calR_1(V_1(r)) {+}  \calR_2(V_2(r)) \} \dd r  \geq
  \int_0^T  \!\! \calR_\eff(U'(r))  \dd r\,, 
\]which turns the above relations into a chain of 
equalities.
Hence, from 
\[
 \lim_{n\to\infty}  \int_0^T \!\! \left\{  \calR_1(\tfrac12
   \TT^{(1)}_{\bftau} \pws  U{\bftau}'(r)) {+}   \calR_2(\tfrac12
   \TT^{(2)}_{\bftau}\pws U{\bftau}'(r))  \right\} \dd r 
=  \int_0^T \!\! \left\{  \calR_1(V_1(r)){+} \calR_2(V_2(r)) \right\} \dd r
\]
we obtain the individual convergences \eqref{optimal-rates} on the interval
$[0,T]$, as the liminf of each integral term on the left-hand side is estimated
from below by the corresponding term on the right-hand side.

Finally, recall that $U'= V_1+V_2$ a.e.\ in $(0,T)$, so that
$\calR_\eff(U') \leq \calR_1(V_1) + \calR_2(V_2) $ a.e.\ in $(0,T)$. Combining
this with the fact that $\int \calR_\eff (U') \dd r $ equals
$ \int \{ \calR_1(V_1){+}\calR_2(V_2)\} \dd r$ ultimately leads to the desired
optimality property $\calR_\eff(U') = \calR_1(V_1) + \calR_2(V_2) $ a.e.\ in
$(0,T)$.  Hence, we conclude the validity of \eqref{optimal-decomposition}, and
thus, the proof of Theorem \ref{th:conv-split-step}.  \QED

Finally, we briefly comment on the enhanced convergences \eqref{enh-cvg}. It is
clear that it suffices to show the convergence results on intervals of the form
$[0,t]$. The technical issue arises because a general $t\in(0,T)$ is in general
not aligned with the partition $\mathscr{P}_{\bftau} $ of
\eqref{non-uniform-partition}.

To show the enhanced convergence, we recall the time-interpolants
$\pwc \sft \bftau: [0,T]\to [0,T]$ from \eqref{time-interpolants}, which
satisfy $\pwc \sft \bftau(t)\to t$ for $\bftau\to 0$. With this, we repeat
the argument from above for proving the convergence of the rate and
slope terms, while passing from the approximate energy-dissipation balance
\eqref{EDB-approx} on $[0,t^k_\bftau]$ to the limit energy-dissipation balance
\eqref{EDB-effect} on the interval $[0,t]$.  Using the liminf estimates for 
energy and powers in \eqref{cv-en1} and \eqref{cv-power1}, it again 
suffices to show a liminf estimate for the rate and slope terms on intervals
$[0,t]$. For convenience, we consider only the rate term, because the slope
can be treated analogously. In particular, it suffices to show the following 
liminf estimate:
\begin{align*}
\liminf_{\bftau\to 0}\int_0^{\pwc \sft \bftau(t)}\calR_j(\tfrac12
   \TT^{(j)}_{\bftau} \pws  U{\bftau}'(r))\dd r\geq \int_0^t \calR_j(V_j(r))\dd r, 
\end{align*}
where
$\tfrac12 \TT^{(j)}_{\bftau} \pws U{\bftau}'=: V_{\bftau}^{(j)}\weakto V_j$ in
$\rmL^1([0,T];\Spx j)$.  Introducing
$W_{\bftau}^{(j)}:=\chi_{[0,\pwc \sft \bftau(t)]}V_{\bftau}^{(j)}\in
\rmL^1([0,T];\Spx j)$, we get
\begin{align*}
  \int_0^{\pwc \sft \bftau(t)}\calR_j(\tfrac12
  \TT^{(j)}_{\bftau} \pws  U{\bftau}'(r))\dd r &= \int_0^T \chi_{[0,\pwc \sft
    \bftau(t)]}(r) \calR_j(V_{\bftau}^{(j)}(r))\dd r \\
  & =\int_0^T \calR_j(\chi_{[0,\pwc \sft \bftau(t)]}(r)V_{\bftau}^{(j)}(r))\dd r 
  = \int_0^T  \calR_j(W_{\bftau}^{(j)}(r))\dd r ,
\end{align*}
where we used $\cR_j(0)=0$. Moreover,
$\|W_{\bftau}^{(j)}(r)\|_{\Spx j} \leq \|V_{\bftau}^{(j)}\|_{\Spx j}$ allows us
to apply the compactness argument in Corollary \ref{cor:new-aprio} such that
there is $W_j\in \rmL^1([0,T];\Spx j)$ with $W_{\bftau}^{(j)}\weakto W_j$ in
$\rmL^1([0,T];\Spx j)$. Using that $\pwc \sft \bftau(t)\to t$ as $\bftau\to 0$,
we find $W_j=\chi_{[0,t]}V_j$ and obtain
\begin{align*}
  \liminf_{\bftau\to 0}&\int_0^{\pwc \sft \bftau(t)}\calR_j(\tfrac12
  \TT^{(j)}_{\bftau} \pws  U{\bftau}'(r))\dd r =  \liminf_{\bftau\to 0} \int_0^T
  \calR_j(W_{\bftau}^{(j)}(r))\dd r\\ 
  &\geq \int_0^T \calR_j(W_j(r))\dd r =\int_0^T \calR_j(\chi_{[0,t]}V_j(r))\dd r
  = \int_0^t \calR_j(V_j(r))\dd r \,.
   \end{align*}
With this, the enhanced convergences \eqref{enh-cvg} and \eqref{optimal-rates} 
are established. 

\section{Alternating Minimizing Movements}
\label{s:AltMinMov}

In this section we discuss an extension of our convergence result in Theorem
\ref{th:conv-split-step}.  Namely, we show that the splitting scheme can be
combined with the Minimizing-Movement approximations of the single-dissipation
gradient systems $(\Spx j,\calE,\tdis j)$ with
$\tdis j=2\cR_j(\frac12\,\cdot)$. 

\subsection{Setup and convergence result}
\label{su:AltMM.Setup} 

 More precisely, for each $j=1,2$ we set up
the Minimizing Movement scheme and construct discrete solutions to the
subdifferential inclusions \eqref{Cauchy-pb-1} and \eqref{Cauchy-pb-2} by
solving the time-incremental minimization problems involving the rescaled
potentials $\tdis1$ and $\tdis 2$ in an alternating manner.  Then, we define
approximate solutions by suitably interpolating the discrete solutions.

Hence, let $\mathscr{P}_{\bftau}$ be a (possibly non-uniform) partition as in
\eqref{non-uniform-partition}; recall that the sub-interval
$(\dis t{\bftau}{k-1}, \dis t{\bftau}{k})$ is split into semi-intervals via
\[
  (\dis t{\bftau}{k-1}, \dis t{\bftau}{k}) = \lint{k}{\bftau} \cup
  \rint{k}{\bftau} \quad \text{with } \lint{k}{\bftau} =
    \big(\dis t{\bftau}{k-1},\thalf \big] \ \text{ and } \ 
    \rint{k}{\bftau} = \big(\thalf , \dis t \bftau k \big],
\]
where $\thalf = \dis t{\bftau}{k-1}+\tfrac {\tau_k}2$.  Starting from an
initial datum $u_0 \in \domE_0$, we define the \emph{piecewise constant}
time-discrete solutions $\pwc U{\bftau} :[0,T] \to \domE_0$ in the following
way: We set $\pwc U{\bftau}(0): = u_0$ and, for $t\in (0,T)$, we define 
\begin{subequations}
\label{min-schemes-i}
\begin{align}
\label{min-schemes-1}
& 
\text{for } t \in  \lint{k}{\bftau} {:}  && 
\pwc U{\bftau}(t) : = U_k^1
\\
 & && \text{ with } U_k^1  \in \Argmin_{U \in  \Spx 1} \left\{
   \tfrac{\tau_k}2\, \widetilde{\calR}_1\! \left(  \tfrac2{\tau_k}\left(
       U{-}\pwc U{\bftau} ( \dis t{\bftau}{k-1}  
  )  \right) \right)  {+} \en {\thalf}U 
 \right\} \,, \hspace*{2em}
 \nonumber
 \\
 \label{min-schemes-2}
 & 
\text{for } t \in  \rint{k}{\bftau}  {:}  && 
\pwc U{\bftau}(t) : = U_k^2
\\
 & &&  \text{with } U_k^2  \in \Argmin_{U \in  \Spx 2} \left\{ \tfrac{\tau_k}2
   \,\widetilde{\calR}_2 \!\left(  \tfrac2{\tau_k} \left( U{-}\pwc
       U{\bftau}  (\thalf
 )  \right) \right)  {+} \en {\dis t{\bftau}{k}}U 
 \right\} \,.
 \nonumber
\end{align}
\end{subequations}
 
We also define the piecewise constant interpolant
$\upwc U{\bftau} :[0,T] \to \domE_0$ by
\begin{equation}
\label{right-continuous-pwc}
\upwc U{\bftau}(t): =
\left\{ 
\begin{array}{ll}
u_0  & \text{for } t \in (0,\tfrac{\tau_1}2],
\smallskip
\\
\pwc U{\bftau}\left(t - \tfrac{\wt\tau(t)}2 \right) & \text{for } t \in
(\tfrac{\tau_1}2, T] 
\end{array}
\right.
\end{equation}
(cf.\ \eqref{tau-dependent-indexes} for the notation $\wt\tau(t)$). 

Furthermore, we introduce the \emph{piecewise linear} interpolant
$\pwl U{\bftau}:[0,T]\to \domE_0$, i.e.\
\begin{equation}
\label{pwl}
\pwl U{\bftau}(t): =
\left\{
\begin{array}{cl}
\displaystyle 
 \frac{t{-} t_{\bftau}^{k-1}}{\tau_k/2}\: \pwc U{\bftau}(t)+  \frac{ \thalf
   {-}t}{\tau_k/2} \:\upwc U{\bftau}(t)   &  \text{for  } t \in [t_{\bftau}^{k-1},
  \, \thalf] \,,   \medskip
 \\
 \displaystyle 
 \frac{t{-} \thalf}{\tau_k/2}\: \pwc U{\bftau}(t)+  \frac{
   t_{\bftau}^{k}{-}t}{\tau_k/2} \: \upwc U{\bftau}(t)  
  &  \text{for } t \in [\thalf,\, t_{\bftau}^{k}] \,.
 \end{array}
 \right.
\end{equation}
Thus, the piecewise constant and linear interpolants satisfy the Euler-Lagrange
equations for the minimum problems \eqref{min-schemes-i}, namely
\[
\begin{aligned}
& 
\partial \widetilde{\calR}_1 \big(\pwl U{\bftau}'(t) \big) + 
 \frsubopt 1 {\thalf}
{\pwc U{\bftau}(t)} \ni 0  && \quad \text{in } \Spx 1^* 
 &&\foraa\   t  \in  \lint{k}{\bftau},
\\
& \partial \widetilde{\calR}_2 \big(\pwl U{\bftau}'(t) \big) + 
\frsubopt 2 {   t_{\bftau}^{k}  }
{\pwc U{\bftau}(t)} \ni 0   && \quad \text{in } \Spx 2^*  && \foraa\   t  \in
\rint{k}{\bftau}. 
\end{aligned}
\]
Nonetheless, like in the time-continuous setup we will not directly pass to the
limit in the above inclusions but instead resort to a discrete
energy-dissipation inequality that will act as a proxy of the
energy-dissipation balance \eqref{EDB-approx} and will be at the core of the
proof of the following convergence result.

\begin{theorem}[Alternating Minimizing Movements]
\label{th:conv-split-step-MM}
In addition to the assumptions of Theorem \ref{exist:gflows}, assume the
singleton condition Hypothesis \ref{hyp:singleton}.  Starting from an initial
datum $u_0 \in \domE_0$, define the curves
$(\pwc U{\bftau})_{\bftau \in \Lambda}$ and
$(\pwl U{\bftau})_{\bftau \in \Lambda}$ as in (\ref{min-schemes-i}) and
(\ref{pwl}).

Then, for any sequence $(\bftau_n)_n$ with $\lim_{n\to\infty} |\bftau_n | =0$
there exist a (non-relabeled) subsequence, a curve $U \in \AC([0,T];\Spx 1)$
and functions $V_j \in \rmL^1 ([0,T];\Spx j)$, $j =1,2$, such that $U(0)=u_0$,
the following convergences hold as $n\to\infty$
\begin{subequations}
\label{cvgs-Thm71}
\begin{align}
\label{pointwise-cvg-discr}
\hspace*{1em} &
 \pwc U{\bftau_n}(t) \weakto U(t) \ \text{ and } \ 
 \pwl U{\bftau_n}(t) \weakto U(t)  &&  
   \text{  in } \Spx 1  &&  \text{for all } t \in [0,T],\hspace*{1em} 
  \\
  & 
  \label{cvg-repeated-rates-discr}
 \hspace*{3em} \frac12  \TT^{(j)}_{\bftau_n}(\pwl U{\bftau_n}') \weakto V_j  &&
 \text{ in } \rmL^1([0,T];\Spx j)  &&  \text{for } j=1,2, 
 \\
 & 
 \label{cvg-U'-discr}
 \hspace*{6.1em}  \pwl U{\bftau_n}' \weakto U' =V_1{+}V_2   && \text{
   in } \rmL^1([0,T];\Spx 1), &&  
 \end{align}
\end{subequations}
and there exists a function $\xi \in\rmL^1([0,T];\Spx 1^*)$ such that the pair
$(U,\xi)$ solves the subdifferential inclusion \eqref{solving-the-equn} and
fulfills the energy-dissipation balance \eqref{EDB-effect}. Furthermore, the
functions $(V_1,V_2) $ provide an \emph{optimal decomposition} for $U'$ in the
sense of \eqref{optimal-decomposition}.
\end{theorem}

\subsection{Proof of Theorem \ref{th:conv-split-step-MM}}
\label{su:ProofThm6.1}

We start by deriving the discrete analogue (in fact, an \emph{inequality}) of
the energy-dissipation balance \eqref{EDB-approx}.  For this, we need to
bring into the picture a further interpolant, commonly known as the
\emph{variational interpolant}, which was first introduced in the framework of
the Minimizing Movement theory for metric gradient flows by \textsc{E.\ De
  Giorgi}, cf.\ \cite{Ambrosio95, AGS08}.  In the present context, the
interpolant $\pwM U{\bftau}: [0,T]\to \domE_0$ is defined in the following way:
$\pwM U{\bftau}(0):=u_0$ and, for $t>0$,
\begin{align}
\nonumber
\hspace*{0.5em}& 
\text{for } t \in  \lint{k}{\bftau}, \ \ t = \dis t
{\bftau}{k-1}{+}r :  && 
\pwM U{\bftau}(t)  \in \Argmin_{U \in  X} \left\{ r \widetilde{\calR}_1\!
  \left( \tfrac1r  \left(U{-}\pwc U{\bftau}  (t^{k-1}_\bftau) \right)
  \right)  {+} \en {t}U \right\};\hspace*{0.5em} 
\smallskip
\\[-0.3em]
& && && \label{def-var-interpolant}
\\[-0.2em]
& \nonumber
\text{for } t \in  \rint{k}{\bftau}, \ t =\thalf{+}r{:} && 
\pwM U{\bftau}(t)  \in \Argmin_{U \in  X} \left\{ r \widetilde{\calR}_2 \left(
    \tfrac1r \left( U{-}\pwc U{\bftau}  (\thalf)  
 \right) \right)  {+} \en {t}U \right\}.  
\end{align}
The existence of a \emph{measurable} selection in the sets of minimizers in 
 \eqref{def-var-interpolant}
 follows by, e.g., \cite[Cor.\,III.3, Thm.\,III.6]{Castaing-Valadier77}. Since, for $t= \thalf $ and for $t= \dis t{\bftau}k$
the minimum problems   in \eqref{min-schemes-i} coincide with those in \eqref{def-var-interpolant}, we may assume that 
\[
\pwc U{\bftau}(s) = \upwc U{\bftau}(s) =  
\pwl U{\bftau}(s)  =\pwM U{\bftau}(s) \qquad \text{for } s =
 \thalf, \  s= \dis t{\bftau}k \quad\text{and } k = 1,\ldots, N_{\bftau}\,.
\]
Furthermore, by \cite[Thm.\,8.2.9]{Aubin-Frankowska}, with $\pwM U{\bftau}$ we
can associate a measurable function $\pwM \xi{\bftau}: (0,T) \to \Spx 2^*$
fulfilling the Euler equation for the minimum problems
\eqref{def-var-interpolant}, i.e.\
\[
\pwM \xi{\bftau}(t) \in 
\left\{
\begin{array}{ll}
  \displaystyle
\frsubopt 1 t {\pwM U{\bftau}(t)} \cap \left( {-} \partial \widetilde{\calR}_1 
\left( \tfrac 1{t{-}\dis t{\bftau}{k-1}} \big( \pwM U{\bftau}(t){-} \upwc
    U{\bftau}  (\dis t{\bftau}{k-1})  \big)  \right) \right)  &
\text{for } t \in  \lint{k}{\bftau},
\smallskip
\\
  \displaystyle
\frsubopt 2 t {\pwM U{\bftau}(t)} \cap \left( {-} \partial \widetilde{\calR}_2 
\left( \tfrac1{t{-}\thalf} \big( \pwM U{\bftau}(t){-} \upwc U{\bftau} 
 (\thalf )  \big)
\right)  \right) &  
\text{for } t \in  \rint{k}{\bftau},
\end{array}
\right.
\]
for $k=1,\ldots, N_{\bftau}$.  Then, we may apply \cite[Lem.\,6.1]{MRS2013} (see
also \cite{MieRos20?DL}), and conclude that that interpolants $\pwc U{\bftau}$,
$\pwl U{\bftau}$, $\pwM U{\bftau}$, and $\pwM\xi{\bftau}$ fulfill, on the
 semi-intervals  $ \lint{k}{\bftau} = (\dis t{\bftau}{k-1}, \thalf]$, the
following estimate 
\begin{equation}
\label{discr-endiss-1}
\begin{aligned}
&
 \en {\thalf}{\pwc U{\bftau}(\thalf)} +
\int_{\dis t{\bftau}{k-1}}^{\thalf}\!\left\{ \widetilde{\calR}_1 (\pwl
  U{\bftau}'(r)) +\widetilde{\calR}_1^* ({-}\pwM
\xi{\bftau}(r)) \right\} \dd r  
\\
&
\leq  \en {\dis t{\bftau}{k-1}}{\pwc U{\bftau}(\dis t{\bftau}{k-1})} +
\int_{\dis t{\bftau}{k-1}}^{\thalf}  \partial_t \en r{\pwM U{\bftau}(r)} \dd
r\,. 
\end{aligned}
\end{equation}
 The analogue holds on 
$\rint{k}{\bftau}= (\thalf, \dis t{\bftau}k]$, involving the dissipation
potential $\widetilde{\calR}_2$.

Relying on \eqref{discr-endiss-1} and its analogue on the intervals
$\rint k{\bftau}$ we can deduce  a discrete energy-dissipation inequality
\eqref{EDB-discr} below, replacing the time-continuous energy-dissipation
balance \eqref{EDB-approx}.   In order to state it in a compact form, we
introduce the discrete analogues of the \emph{rate} and \emph{slope} terms from
\eqref{rate+slope-terms}.  With slight abuse, we will denote them with the same
symbols used in \eqref{rate+slope-terms}: 
\begin{align*}
\mathcal{D}_{\bftau}^{\mathrm{rate}}  ([0,T]) & \ := \ 
    \int_0^T  \left\{
\chi_{\bftau} (r)\  \tdis1 \big(\pwl U{\bftau}'(r)\big) +
(1{-}\chi_{\bftau}(r) ) \,\tdis2 \big(\pwl U{\bftau}'(r)\big) \right\} \dd r 
\\
 &   \stackrel{\eqref{rephrasing-via-repet-a}}{=}     \int_0^T  
\left\{  \calR_1(\tfrac12  \TT^{(1)}_{\bftau} \pwl U{\bftau}'(r)) +
  \calR_2(\tfrac12  \TT^{(2)}_{\bftau}\pwl U{\bftau}'(r)) \right\} \dd r, 
\\
\mathcal{D}_{\bftau}^{\mathrm{slope}} ([0,T])  & \ : = \       \int_0^T  
\left\{ \chi_{\bftau}(r) \, \tdis1^* ({-}\pwM \xi{\bftau} (r))  +
  (1{-}\chi_{\bftau}(r) ) \,\tdis2^* ({-}\pwM \xi{\bftau} (r))  \right\}  \dd r 
\\
  & \stackrel{\eqref{rephrasing-via-repet-b}}{=}     \int_0^T    
\left\{  \calR_1^*({-} \TT^{(1)}_{\bftau} \pwM \xi{\bftau}(r)) +
  \calR_2^*({-}\TT^{(2)}_{\bftau} \pwM \xi{\bftau} (r)) \right\} \dd r\,.
\end{align*}
%
The following result (to be compared with Proposition \ref{prop:aprio} and
Corollary \ref{cor:new-aprio}) collects all of the a priori estimates stemming
from \eqref{EDB-discr}.

\begin{proposition}
\label{pr:AlTMinAprio}
The interpolants $\pwc U{\bftau} $, $\pwl U{\bftau}$, and $\pwM U{\bftau}$ and
$\pwM\xi{\bftau}$ fulfill the \emph{discrete energy-dissipation inequality}
\begin{equation}
\label{EDB-discr}
\begin{aligned}
  & \en t{\pwc U\bftau(t)} + \mathcal{D}_{\bftau}^{\mathrm{rate}} ([s,t])
  +\mathcal{D}_{\bftau}^{\mathrm{slope}} ([s,t])  \leq   \en s{\pwc U\bftau(s)}
  +\int_s^t \partial_t \en r{\pwM U\bftau(r)} \dd r
\end{aligned}
\end{equation}
for all $ 0 \leq s \leq t \leq T$.  Moreover, there exists a positive constant
$\overline{C}>0$ such that the following estimates hold for all
$\bftau \in \Lambda$:
\begin{subequations}
\label{aprio-est-discr}
\begin{align}
\label{discr-est-en+pow}
&
\sup_{t\in [0,T]} \mfE(\pwc U\bftau(t))  \leq \overline{C},  \qquad   \sup_{t\in [0,T]}  \mfE(\pwM U\bftau(t))  \leq \overline{C},  \qquad   \sup_{t\in [0,T]} 
|\partial_t \en t{\pwM U\bftau(t)} |  \leq \overline{C}, 
\\
& 
\mathcal{D}_{\bftau}^{\mathrm{rate}} ([0,T]) +\mathcal{D}_{\bftau}^{\mathrm{slope}} ([0,T]) \leq \overline{C}.
\end{align}
\end{subequations}
Thus, the families
$(\pwl U{\bftau}')_{\bftau \in \Lambda} \subset \rmL^1([0,T];\Spx1)$, \
$( \TT^{(j)}_{\bftau} \pwl U{\bftau}')_{\bftau \in \Lambda} \subset
\rmL^1([0,T];\Spx j)$, and
$( \TT^{(j)}_{\bftau} \pwM \xi{\bftau})_{\bftau \in \Lambda} $ $ \subset
\rmL^1([0,T];\Spx j^*)$, for $j\in \{1,2\}$, are uniformly integrable.
Finally, as $ |\bftau| \downarrow 0 $, we have  
\begin{equation}
\label{e:consistency}
\sup_{t\in [0,T]} \norm{ \pwc U{\bftau}(t){-}  \upwc U{\bftau}(t)}1+ \sup_{t\in
  [0,T]} \norm{ \pws U{\bftau}(t){-}  \pwc U{\bftau}(t)}1 + \sup_{t\in [0,T]}
\norm{ \pwM U{\bftau}(t){-}  \upwc U{\bftau}(t)}1   = \mathrm{o}(1) \,. 
\end{equation}
\end{proposition}
\begin{proof}
  Clearly, estimates \eqref{aprio-est-discr} follow from \eqref{EDB-discr} via
  the same arguments as in the proof of Proposition \ref{prop:aprio}. The
  estimate for $\sup_{t\in [0,T]} \mfE(\pwM U\bftau(t)) $ can be retrieved from
  the  discrete  energy-dissipation inequality \eqref{discr-endiss-1}
  and its analogue on the  semi-intervals  $\rint{\bftau}k$.  Let us
  just comment on the proof of \eqref{e:consistency}:  the estimates for
  $ \norm{ \pwc U{\bftau}{-} \upwc U{\bftau}}1$ and
  $ \norm{ \pws U{\bftau}{-} \upwc U{\bftau}}1$ derive from the uniform
  integrability of the family $(\pws U{\bftau}')_{\bftau \in \Lambda}$ in
  $ \rmL^1([0,T];\Spx1)$, while we refer to the proof of \cite[Prop.\
  6.3]{MRS2013} for the estimate of
  $ \norm{ \pwM U{\bftau}{-} \upwc U{\bftau}}1$.
\end{proof}

\paragraph{\bf Sketch of the passage to the limit in the energy-dissipation
  inequality \eqref{EDB-discr}.}
By repeating the very same arguments as in the proof of Corollary
\ref{cor:compactness} and relying on \eqref{e:consistency}, we show that for
any sequence $(\bftau_n)_n$ with $|\bftau_n| \downarrow 0$ as $n\to\infty$
there exist a (non-relabeled) subsequence of $(\pwc U{\bftau_n})_n$,
$(\upwc U{\bftau_n})_n$, $(\pws U{\bftau_n})_n$ and $(\pwM U{\bftau_n})_n$ and
a curve $U \in \AC([0,T];\Spx 1)$ such that for all $t\in [0,T]$ there holds
\[
  \pwc U{\bftau_n}(t), \upwc U{\bftau_n}(t), \pwl U{\bftau_n}(t), \pwM
  U{\bftau_n}(t) \rightharpoonup U(t) \qquad \text{in } \Spx 1\,.
\]
Convergences \eqref{cvg-U'-discr} for $(\pwl U{\bftau_n}')_n$ and
\eqref{cvg-repeated-rates-discr} for
$( \tfrac12 \TT^{(j)}_{\bftau} \pwl U{\bftau}' )_n$ (to functions $V_j$ such
that $V_1+V_2=U'$), hold, too. Likewise, we conclude the analogues of
convergences \eqref{cvg-for-the-halves-forces} for the sequences
$( \TT^{(j)}_{\bftau} \pwM \xi{\bftau_n})_n$. Therefore, we are in a position
to take the limit as $n\to\infty$ in \eqref{EDB-discr}.  A straightforward
adaptation of the arguments from Section\ \ref{ss:vel} and \ref{ss:slope} leads
us to conclude that there exists $\xi \in \rmL^1 ([0,T];\Spx 1^*)$ such that
the pair $(U,\xi)$ complies with the energy-dissipation inequality
\eqref{EDB-0T-lim}. Then, we repeat the very same arguments from Section
\ref{ss:4.4} and establish that $(U,\xi)$ in fact fulfills the
energy-dissipation balance \eqref{EDB-effect}, that it solves the
subdifferential inclusion \eqref{solving-the-equn}, and that $(V_1,V_2)$
provide an optimal decomposition of the rate $U'$.

This finishes the proof of Theorem \ref{th:conv-split-step-MM}.  \QED

\section{The time-splitting method for systems with a block structure}
\label{s:block}
In this section we tackle the application of the splitting method to generalized gradient systems with a \emph{block structure}. In such systems, 
\[
  \text{the state variable $u$ is a vector } \binom yz
\in \blockspace: = \Spy \ti \Spz,
\]
with $\Spy$ and $\Spz$ (separable) reflexive Banach spaces.  The evolution of
the system is governed by an energy functional
\[
  \calE: [0,T]\ti \blockspace \to (-\infty, \infty], \qquad \calE = \calE(t,u)=
  \calE(t,y, z)
\]
(we will use both notations with slight abuse), whereas dissipation mechanisms
are encoded by two dissipation potentials $\disblo y$ and $ \disblo z$, each
acting on  one of  the components of the rate vector
$u'= \binom\yvel\zvel$, namely
\[  
\cR(v,w)=\big(\disblo y {\oplus} \disblo z\big)(v,w):= \disblo y
(\yvel)+\disblo z (\zvel) \ \text{ with } 
\disblo y{:}\: \Spy \to [0,\infty) \text{ and } 
  \disblo z{:}\: \Spz \to [0,\infty)\,.
\]

The analysis of these systems can be carried out in the context of the
splitting approach in the  previous case of  Section\ \ref{se:TimeSplit} by
introducing the dissipation potentials $\calR_j: \blockspace \to [0,\infty]$
 via  
\begin{equation}
\label{dissip-potent-Rj}
\begin{aligned}
\calR_1(u') =  \calR_1(\yvel,\zvel)=\disblo y(\yvel) + \mathcal{I}_{\{0\}}(\zvel),  \qquad  \calR_2(u') =  \calR_2(\yvel,\zvel)=\calI_{\{0\}}(\yvel) + \disblo z(\zvel),
\end{aligned}
\end{equation}
where $\calI_{\{0\}}$ is the indicator function of the singleton $\{0\}$ with
$\calI_{\{0\}}(0)=0$ and $\infty$ else. 
Let 
\[
\partial^{\blockspace} \calE : [0,T]\ti \blockspace \rightrightarrows
\blockspace^* \text{ be the Fr\'echet subdifferential of } \calE(t,\cdot)
\]
in  the duality pairing $\pairing{}{\blockspace}{\cdot}{\cdot}$. 
Our time-splitting scheme  will be  based on the
Cauchy problems for the subdifferential inclusions
\[
\partial \calR_j(u'(t)) + \partial^{\blockspace} \calE(t,u(t))  \ni 0 \quad
\text{in } \blockspace^* \qquad \foraa\  t \in (0,T),
\]
in which we either freeze the variable $z$ (for $j=1$), or the variable $y$
(for $j=2$).  It can be easily calculated that, in this setup, the effective
dissipation potential is
\begin{equation}
  \label{effective-block}
  \begin{aligned}
    & \calR_\eff: \blockspace \to [0,\infty), && \calR_\eff(u') = \disblo y
    (\yvel) + \disblo z (\zvel) \qquad \text{with }
    \\
    & \calR_\eff^*: \blockspace^* \to [0,\infty), && \ \calR_\eff^*(\xi) =
    \disblo y^* (\yfor) + \disblo z^* (\zfor).
  \end{aligned}
\end{equation}

\subsection{Assumptions}
\label{ss:8.1}

Our conditions on $\calE$ and on the dissipation potentials $\disblo y$ and
$\disblo z$ mimic the setup of Section\ \ref{s:2}, but with some significant
differences. We start by settling the properties of the energy functional.

\begin{hypothesis}
\label{hyp:E-block}
The functional $\calE: [0,T]\ti \blockspace \to (-\infty,\infty]$ has the
proper domain $\mathrm{dom}(\calE)$ $ = [0,T]\ti \domE_0 $, on which $\calE$ is
bounded from below \eqref{bded-below} and complies with the
time-differentiability condition $\mathbf{<E>}$.  Along all sequences
$ (t_n,u_n) \weakto (t,u) \text{ in } [0,T]\times \blockspace$ with
$ (u_n)_{n\in \N}$ contained in an energy sublevel $ S_E$
\begin{compactitem}
\item[-]
  the lower semicontinuity and power continuity conditions from \eqref{h:2.1} hold;
  \item[-] the closedness condition \eqref{h:2.2} for $\partial^{\blockspace} \calE: [0,T]\ti \blockspace \rightrightarrows \blockspace^*$ holds.
  \end{compactitem}
\end{hypothesis}
The following subsumes our requirements on the dissipation potentials $\disblo y$ and $\disblo z$.
\begin{hypothesis}
\label{hyp:diss-block}
For $\mathrm{x} \in \{\mathrm{y},\mathrm{z}\}$ 
and $\mathbf{X} \in \{\Spy, \Spz\}$
the functionals $\disblo {\mathrm{x}}: \mathbf{X} \to [0,\infty)$ 
  and their conjugates $\disblo {\mathrm{x}}^*: \mathbf{X}^* \to [0,\infty)$ comply with condition $\mathbf{<R>}$ and with the Quantitative Young Estimate \eqref{eq:QuYouEst}. 
\end{hypothesis}
We note that $\cR_\eff =\cR_y+\cR_z$ does not necessarily satisfy the QYE. 
To see  this, we define the norm on the product space $\Spx{}$ by
$\|v\|_{\Spx{}} = \|\yvel\|_\Spy+\|v_z\|_\Spz$. Then
$\|\xi\|_{\Spx{}^*} = \max\{\|\xi_y\|_{\Spy^*},\|\xi_z\|_{\Spz^*}\}$, and we
have
\[
\begin{aligned}
\cR_\eff(v)+\cR_\eff^*(\xi)  & = \cR_y(\yvel)+\cR_z(v_z) +  \cR_y^*(\xi_y)+\cR_z^*(\xi_z)
\\
 & \geq c_\Spy \|\yvel\|_\Spy  \|\xi_y\|_{\Spy^*} + c_\Spz \|v_z\|_\Spz  \|\xi_z\|_{\Spz^*} -c.
\end{aligned}
\]
However, the term above does not bound
$c_{\Spx{}} \{\|\yvel\|_\Spy+\|v_z\|_\Spz\} \cdot
\max\{\|\xi_y\|_{\Spy^*},\|\xi_z\|_{\Spz^*}\} -\tilde c$, in general.

Although Hypotheses \ref{hyp:E-block} and \ref{hyp:diss-block} mimic the setup
of Section\ \ref{se:TimeSplit}, it is clear that the overall gradient system
$(\Spx, \calE, \calR_\eff)$ is different from that considered therein.  The
first, striking difference is that after extending the dissipation potentials
$\disblo y$ and $\disblo z$ to the dissipation potentials $\calR_1$ and
$\calR_2$ on the \emph{common} space $\blockspace= \Spy\ti\Spz$ (cf.\
\eqref{dissip-potent-Rj}), we lose the coercivity of $\calR_1^*= \disblo y^*$
and $\calR_2^*=\disblo z^*$ required via \eqref{gsuperlinear-growth} on
$\blockspace^*=\Spy^*\ti\Spz^* $.  Moreover, we emphasize that, here,
the Fr\'echet subdifferential of $\calE(t,\cdot)$ is considered in the
duality pairing of the common space $\blockspace$.  Nonetheless, in what
follows we are going to show that the techniques at the core of the analysis in
Section\ \ref{s:4} carry over to the present setting. For this, a crucial role
will be played by the condition that $\partial^{\blockspace} \calE$ has a
`cross product structure', which is weaker than the singleton condition needed
in the setup of Section \ref{se:TimeSplit}.

\begin{hypothesis}[Cross-product condition]
\label{h:cross-product condition}
For all $(t,u) = (t,y,z) \in\dom(\partial \cE)$ we have 
\begin{equation}
\label{cross-product-condition}
\partial^{\blockspace}
\calE(t,y,z) =\blosub y (t,y,z) \ti \blosub z (t,y,z)  
\end{equation}
where $\blosub x : [0,T]\ti \blockspace \rightrightarrows \mathbf{X}^*$, for
$\mathrm{x} \in \{\mathrm{y},\mathrm{z}\}$ and correspondingly 
$\mathbf{X} \in \{\Spy, \Spz\}$, is the \emph{partial subdifferential} of
$\calE$ with respect to the variable $ \mathrm{x}$,  while fixing the
other variable. 
\end{hypothesis}
We note that this condition is more general than the singleton condition
because multi-valued subdifferentials are still possible. We also remark for
later use that, in view of \eqref{cross-product-condition}, the closedness
condition for $ \partial^{\blockspace} \calE$ is indeed equivalent to
\begin{equation}
\label{h:2.2-partial}
\begin{aligned}
\forall\, E>0: \ 
 & \left\{
\begin{array}{ll}
\displaystyle
 (t_n,y_n, z_n, \eta_n,\zeta_n ) \weakto (t,y, z, \eta, \zeta)  \text{ in
 } [0,T]\ti \Spy \ti \Spz \ti \Spy^*\ti \Spz^*,  
 \\
   (y_n,z_n) \in S_E, \ \eta_n \in \blosub y (t_n,y_n,z_n), \    \zeta_n \in
   \blosub z (t_n,y_n,z_n) \ \forall\,n \in \N  
   \end{array}
   \right.
 \\  
   &  \  \quad  \Longrightarrow  
\quad \eta \in  \blosub y (t,y,z) \text{ and } \zeta  \in  \blosub z (t,y,z)\,.
\end{aligned}
\end{equation}
Obviously,   in view of \eqref{effective-block} and \eqref{cross-product-condition}
the subdifferential inclusion
\begin{equation}
\label{subdiff-incl-eff-block}
 \partial \cR_\eff (\dot u(t)) +\partial^{\blockspace} \cE(t, u(t)) \ni 0 \quad  \text{ in } \blockspace^* \quad \foraa\  t \in (0,T)
\end{equation}
is indeed  equivalent to the system
\[
\left\{
\begin{array}{lll}
&
\displaystyle  \! \! \! \!   \! \!  \partial \disblo y (\dot y(t)) +  \blosub y
(t,y(t),z(t))  \ni 0  &  \text{ in } \Spy^* 
\\ 
&
\displaystyle \! \! \! \!   \! \!    \partial \disblo z (\dot z(t)) + \, \blosub
z (t,y(t),z(t))  \ni 0  &  \text{ in } \Spz^* 
\end{array}
\right. \quad \foraa\  t \in (0,T)\,.
\]

Last but not least, we need to specify our chain-rule assumption on
$(\calE, \partial^{\blockspace} \calE)$. As we have previously remarked,
Hypothesis \ref{hyp:diss-block} does not guarantee the validity of the QYE for
$\calR_\eff$. Instead of imposing it as an additional assumption, in this setup
we will compensate for the possible lack of this property by strengthening the
chain rule. In fact, the condition
$\int_0^T \|\xi\|_{\blockspace^*} \|u'\|_{\blockspace} \dd t<\infty $ is now
replaced by (the weaker)
$\int_0^T \left( \calR_\eff(u'){+} \calR_\eff^*({-}\xi)\right) \dd t <\infty$.
Hence, Hypothesis \ref{hyp:CR-block} also encodes a `compatibility condition'
between $\calR_\eff$ and $\calE$.

\begin{hypothesis}[Abstract chain rule]
\label{hyp:CR-block}
The quadruple $(\blockspace,\calE, \partial^{\blockspace} \calE, \calR_\eff)$
satisfies the following property: for 
$(u,\xi) \in \AC([0,T];\blockspace) \ti \rmL^1([0,T];\blockspace^*)$ such that
$\xi(t) \in \partial^{\blockspace} \calE(t,u(t))$ for a.a.\  $t\in (0,T)$,
\begin{equation}
\label{first-part-CR-block}
\sup_{t\in [0,T]} |\calE(t,u(t))|<\infty, \text{ and } \int_0^T \left(
  \calR_\eff(u'(t)){+} \calR_\eff^*({-}\xi(t))\right) \dd t  <\infty, \,
\end{equation}
 the function $t\mapsto \calE(t,u(t))$ is absolutely continuous and 
the chain rule \eqref{eq:48strong} holds. 
\end{hypothesis}

\noindent
Clearly, if $\calR_\eff$ satisfies the $\QYE$, then the chain rule property
$\mathbf{<CR>}$ for $(\calE, \partial^{\blockspace} \calE)$ is sufficient for
Hypothesis \ref{hyp:CR-block}. In the general case, a sufficient condition for
the validity of $\mathbf{<CR>}$ is that the chain rule \eqref{eq:48strong}
holds for all pairs $u = (v,z) \in \AC([0,T];\Spy\ti \Spz)$ and
$\xi =(\eta,\zeta) \in \AC([0,T];\Spy^*\ti \Spz^*)$, with
$\eta(t) \in \blosub y (t,y(t),z(t)) $ and
$\zeta(t) \in \blosub z (t,y(t),z(t)) $ for a.a.\ $t\in (0,T)$ (so that
$\xi(t) \in \partial^{\blockspace} \calE(t,u(t))$ by Hypothesis
\ref{h:cross-product condition}), fulfilling the following estimate
\begin{equation}
\label{weaker-estimate}
 \int_0^T \|\eta(t)\|_{\Spy^*} \, \|y'(t)\|_{\Spy} \dd t  +  \int_0^T \|\zeta(t)\|_{\Spz^*} \, \|z'(t)\|_{\Spz} \dd t <\infty\,.
\end{equation}
Obviously, \eqref{weaker-estimate} is weaker than
$ \int_0^T \|\xi(t)\|_{\blockspace^*} \, \|u'(t)\|_{\blockspace} \dd t <\infty$,
and it follows from the estimate $\int_0^T \left( \calR_\eff(u') 
{+} \calR_\eff^*({-}\xi)\right) \dd t <\infty$  if the individual
dissipation potentials $ \disblo y $ and $ \disblo z $ satisfy the QYE.

\subsection{Time-splitting for block structure systems}
\label{ss:8.2}

As in Section \ref{se:TimeSplit} we introduce the rescaled dissipation
potentials $\tdis j: \blockspace \to [0,\infty)$
\[
\widetilde \cR_1(u')= 2\cR_1 \!\left( \tfrac12 u' \right) =
2\disblo y \!\left( \tfrac12 \yvel \right) + \mathcal{I}_{\{0\}}(\zvel),\quad
\widetilde \cR_2(u')=  2\cR_2 \!\left( \tfrac12 u' \right) =2\disblo z \!\left(
  \tfrac12 \zvel \right) + \mathcal{I}_{\{0\}}(\yvel)\,. 
\]
We will construct our approximate solution to the Cauchy problem for
\eqref{subdiff-incl-eff-block} by solving, in suitable sub-intervals $I_j$ of
$[0,T]$, the Cauchy problems for the doubly nonlinear equations
\[
\partial \tdis j(u'(t)) + \partial^{\blockspace}  \calE(t,u(t))  \ni 0 \quad \text{in } \blockspace^*  \quad  \foraa\  t \in I_j, \quad  j \in \{1,2\},
\]
which, thanks to the cross product condition \eqref{cross-product-condition}, reformulate 
in the same way as  the subdifferential inclusion \eqref{subdiff-incl-eff-block} for $\calR_\eff$. 
\par
More precisely, let $\mathscr{P}_{\bftau} $ be a non-uniform partition of
$[0,T]$ (cf.\ \eqref{non-uniform-partition}) and let
$ (\lint{k}{\bftau})_{k=1}^{N_{\bftau}}$ and
$( \rint{k}{\bftau})_{k=1}^{N_{\bftau}}$ be the associated `left' and `right'
 semi-intervals,  see \eqref{left-right-intervals}.  Starting from an
initial datum $u_0 = (y_0,z_0)\in \domE_0$, we define the approximate solutions
 $\pws U\bftau=  (\pws Y\bftau, \pws Z\bftau): [0,T]\to \domE_0 \subset \Spy\ti
 \Spz$  to \eqref{subdiff-incl-eff-block}
in the following way (cf.\ \eqref{def-pws}).  We start with 
\begin{subequations}
  \label{def-pws-block}
  \begin{equation}
    \label{initial-block}
    (\pws Y\bftau(0),\pws Z\bftau(0)) : = (y_0,z_0) 
  \end{equation}
and  proceed   for $ t \in (\dis t \bftau {k-1}, \dis t \bftau k ] $ and 
$k \in \{ 1, \ldots, N_\bftau\}$   as follows: 
\begin{itemize}
\item[-] On the  semi-interval  $ \lint{k}{\bftau}$, we define
  $(\pws Y\bftau, \pws Z\bftau)$ to be a solution of the Cauchy problem:
  \begin{equation}
    \label{left-Cauchy-block}
    \left\{ 
      \begin{array}{cll}
        \displaystyle
        \partial \disblo y(\yvar'(t)) +  \blosub y (t,y(t),z(t))  \ni 0  &  \text{ in }
        \Spy^* & \text{ and } 
        \smallskip
        \\
        \zvar'(t) \equiv 0 & \text{ in } \Spz &\ \foraa\ t \in   \lint{k}{\bftau}\,;
        \smallskip
        \\
        (y(\dis t{\bftau}{k-1}),z(\dis t{\bftau}{k-1})) = (\pws
        Y\bftau(\dis t\bftau{k-1}),  \pws Z\bftau(\dis t\bftau{k-1})) \,. 
        \hspace{-3em} & &  
      \end{array}
    \right.
  \end{equation}
\item[-] On the  semi-interval  $ \rint{k}{\bftau}$, we define
  $(\pws Y\bftau, \pws Z\bftau)$ to be a solution of the Cauchy problem
  \begin{equation}
    \label{right-Cauchy-block}
    \left\{ 
      \begin{array}{cll}
        \yvar'(t) \equiv 0 & \text{ in } \Spy &\text{ and }
        \smallskip
        \\
        \disblo z(\zvar'(t)) +  \blosub z (t,y(t),z(t)) \ni 0  &  \text{ in } \Spz^*
        &  \ \foraa\ t \in   \rint{k}{\bftau} \,;
        \smallskip
        \\
        (y ( \thalf),z ( \thalf) )  = (\pws Y{\bftau}(\thalf),  \pws
        Z{\bftau}(\thalf))  \,. \hspace{-5em}&  &
      \end{array} 
    \right.
  \end{equation}
\end{itemize}
\end{subequations}
Existence of solutions for the Cauchy problems \eqref{left-Cauchy-block}
follows by noting that for fixed $z\in \Spz$ the triple
$(\Spy, \calE(\cdot,\cdot, z), \disblo y)$ satisfies the required assumptions
of Theorem \ref{exist:gflows}. Analogously, the same holds for the triple
triple $(\Spz, \calE(\cdot, y, \cdot), \disblo z)$ if $y\in\Spy$ is fixed,
which provides existence of solutions of the Cauchy problems
\eqref{right-Cauchy-block}.  Hence, we conclude that the solution curve
$\pws U{\bftau} = (\pws Y{\bftau}, \pws Z{\bftau})\colon [0,T] \to \blockspace
$ fulfills $\pws U{\bftau} \in \AC([0,T];\blockspace)$.  We also introduce the
functions $\pws \eta{\bftau} \in \rmL^1([0,T];\Spy^*)$ and
$\pws \zeta{\bftau} \in \rmL^1([0,T];\Spz^*)$ featuring in the force terms in
the subdifferential inclusions \eqref{left-Cauchy-block} and
\eqref{right-Cauchy-block}. Namely,  for $k \in \{ 1, \ldots, N_\bftau\}$
we set  
\begin{equation}
  \label{interpolant-selections}
  \begin{aligned}
    & \pws \eta{\bftau} (t) \left\{
      \begin{array}{cl}
         \in   \blosub y (t,\pws Y{\bftau}(t),\pws Z{\bftau}(t))
        \cap ({-} \partial \disblo y(\pws Y{\bftau}'(t))) & \text{for }   t \in
        \lint{k}{\bftau} \,, 
        \smallskip
        \\
          \equiv 0 & \text{ for }   t \in \rint{k}{\bftau}\,,
      \end{array}
    \right.
    \\[0.4em]
    & \pws \zeta{\bftau} (t) \left\{
      \begin{array}{cl}
        \equiv 0 & \text{for }   t \in \lint{k}{\bftau} \,, 
        \smallskip
        \\
        \in   \blosub z (t,\pws Y{\bftau}(t),\pws Z{\bftau}(t))  \cap ({-}
        \partial \disblo z(\pws Z{\bftau}'(t))) & \text{for }   t \in
        \rint{k}{\bftau}  \,.
      \end{array}
    \right.
  \end{aligned}
\end{equation}
Finally, for tailoring analysis to the block structure context, in addition to
the `overall' repetition operators
$\TT^{(j)}_\bftau : \rmL^1([0,T];\mathscr{U}) \to \rmL^1([0,T];\mathscr{U})$
for $\mathscr{U} \in \{ \blockspace, \blockspace^*\}$ we will resort to the
operators (denoted by the same symbols)
 \begin{subequations}
\label{repetition-ops-block}
\begin{align}
\TT^{(1)}_\bftau &{:}\  \rmL^1([0,T];\mathscr{Y})  \to  \rmL^1([0,T];\mathscr{Y}); \quad
\big(\TT^{(1)}_\bftau g\big) (t) := \left\{ \ba{@{}cl@{}} 
     g(t)& \text{if } t \in \lint{\bftau}{\mt {\bftau}t}\,,
 \\
   g\big(t{-}\tfrac{\wt\tau(t)}2\big)& \text{if }  t\in\rint{\bftau}{\mt {\bftau}t}\,,
 \ea \right.
 \\
 \TT^{(2)}_\bftau &{:}\  \rmL^1([0,T];\mathscr{Z})  \to  \rmL^1([0,T];\mathscr{Z}); \quad
\big(\TT^{(2)}_\bftau g\big) (t) := \left\{ \ba{@{}cl@{}} 
g\big(t{+}\tfrac{\wt\tau(t)}2\big)& \text{if } t\in \lint{\bftau}{\mt {\bftau}t}\,,
 \\
   g(t)& \text{if }  t \in  \rint{\bftau}{\mt {\bftau}t} \,,
 \ea \right.
\end{align}
\end{subequations}
with $\mathscr{Y} \in \{ \Spy, \Spy^*\}$ and
$\mathscr{Z} \in \{ \Spz, \Spz^*\}$.

Now, we are in a position to give our convergence result for the time-splitting
scheme in the setup with block structure.  Observe that, due to the block
structure we will succeed in relating the weak limits of the repeated rates
$ (\TT^{(1)}_{\bftau_n}( \pws Y{\bftau_n}')) $ and
$( \TT^{(2)}_{\bftau_n}( \pws Z{\bftau_n}') )$ to the limiting rates $Y'$ and
$Z'$, cf.\ \eqref{cvg-repeated-rates-81-statement} below.

\begin{theorem}[Convergence of time-splitting method for block systems]
\label{thm:sec8}
Under Hypotheses \ref{hyp:E-block}, \ref{hyp:diss-block}, \ref{h:cross-product
  condition}, and \ref{hyp:CR-block}, starting from an initial datum
$u_0=(y_0,z_0)\in \domE_0 $, define the curves $\pws U{\bftau}$ as in
\eqref{def-pws-block}.

Then, for any sequence $(\bftau_n)_n$ with $\lim_{n\to\infty} |\bftau_n | =0$
there exist a (non-relabeled) subsequence and a curve
$U=(Y,Z) \in \AC([0,T];\blockspace)$ with $U(0)=u_0$ such that the following
convergences hold for the sequences $(\pws Y{\bftau_n})_n$ and
$(\pws Z{\bftau_n})_n$ as $n\to \infty$:
\begin{subequations}
\label{cvgs-Thm81}
\begin{align}
\label{pointwise-cvg-81}
&\pws Y{\bftau_n}(t) \weakto Y(t) \text{  in } \Spy
 \quad  \text{and}\quad  \pws Z{\bftau_n}(t) \weakto Z(t)  \text{  in } \Spz
  \qquad \text{for all } t \in [0,T], 
 \\
 & 
 \label{cvg-U'-81}
  \pws Y{\bftau_n}' \weakto Y'  \text{ in } \rmL^1([0,T]; \Spy) 
    \ \text{ and } \  \pws Z{\bftau_n}' \weakto Z' \text{ in } \rmL^1([0,T]; \Spz),
 \\
 &
   \label{cvg-repeated-rates-81-statement}
  \tfrac12  \TT^{(1)}_{\bftau_n}( \pws Y{\bftau_n}') \weakto Y'  \text{ in }
  \rmL^1([0,T]; \Spy)  \  \text{ and }\ 
  \tfrac12  \TT^{(2)}_{\bftau_n}( \pws Z{\bftau_n}') \weakto Z' \text{ in }
  \rmL^1([0,T]; \Spz)\,, 
\end{align}
\end{subequations}
and there exists a function $\xi =(\eta,\zeta)\in\rmL^1([0,T];\blockspace^*)$
such that the pair $(U,\xi)$ solves the subdifferential system, for a.a.\ $t\in (0,T)$,
 \begin{equation}
 \label{solving-the-equn-81}
 \begin{array}{cccc}
  \partial \disblo y(Y'(t)) + \eta(t) \ni 0  & \text{and} &
  \eta(t) \in  \blosub y (t,Y(t),Z(t))   &  \text{ in } \Spy^* \,,
  \smallskip
 \\
  \partial \disblo z(Z'(t)) + \zeta(t) \ni 0  & \text{and} &
  \zeta(t) \in  \blosub z (t,Y(t),Z(t))   &  \text{ in } \Spz^* \,,
 \end{array}
 \end{equation} 
and  fulfills the energy-dissipation balance
\begin{equation}
\label{EDB-effect-81}
\begin{aligned}
&
\calE(t,Y(t),Z(t))  + \int_s^t\! \big\{ \disblo y (Y'(r)){+}  \disblo z (Z'(r)) 
+   \disblo y^* ({-}\eta(r)){+}  \disblo z^* ({-}\zeta(r)) \big\} \dd r 
\\
&
 = \calE(s,Y(s),Z(s))  + \int_s^t  \!\partial_t \calE(r,Y(r),Z(r)) \dd r  \qquad
 \text{for }\  0 \leq s \leq t \leq T \,.
 \end{aligned}
\end{equation}
\end{theorem}

Like for Theorem \ref{th:conv-split-step}, we obtain the enhanced convergences
\begin{subequations}
\label{enh-cvg-81}
\begin{align}
& 
\en t{\pws U{\bftau_n}(t)}  \longrightarrow 
\en t{U(t)} \qquad  \text{for all } t \in [0,T],
\\
&
\left. 
\begin{array}{ll}
  \displaystyle
  \int_s^t \disblo y(\tfrac12  \TT^{(1)}_{\bftau_n} \pws Y{\bftau_n}'(r)) \dd r   \longrightarrow   \int_s^t \disblo y (Y'(r)) \dd r,  
  \smallskip
  \\
  \displaystyle
  \int_s^t \disblo z(\tfrac12  \TT^{(2)}_{\bftau_n} \pws Z{\bftau_n}'(r)) \dd r   \longrightarrow   \int_s^t \disblo z (Z'(r)) \dd r, 
  \smallskip
  \\
  \displaystyle
  \int_s^t   \disblo y^*({-} \TT^{(1)}_{\bftau_n} \pws \eta{\bftau_n}(r)) \dd r  \longrightarrow   \int_s^t \disblo y^* ({-}\eta(r)) \dd r,   \smallskip
  \\
  \displaystyle
  \int_s^t   \disblo z^*({-} \TT^{(2)}_{\bftau_n} \pws \zeta{\bftau_n}(r)) \dd r
  \longrightarrow   \int_s^t \disblo y^* ({-}\zeta(r)) \dd r 
 \end{array}
 \right\} \quad \text{for all }\ [s,t]\subset [0,T].
 \end{align}
\end{subequations}

\begin{remark}[Non-convergence]
  It can be easily checked in our ``multi-valued'' counterexample  of
  Section \ref{ss:3.2}  that (i) the ``cross-product condition'' does not
  hold and (ii) that the solutions of the split-step algorithm  using the
  block structure  of $\Spy\ti\Spz = \R\ti\R$ get stuck completely
  when reaching the diagonal $u_1=u_2$.  Hence, we have non-convergence
  because solutions for the effective problem move along the diagonal until they
  reach $u=0$.
\end{remark}

\subsection{Alternating Minimizing Movements for block structures}
\label{su:AltMinMovBlock}
Last but not least, we point out that the analogue of Theorem
\ref{th:conv-split-step-MM} holds for  systems with block structure,
assuming  Hypotheses \ref{hyp:E-block}, \ref{hyp:diss-block},
\ref{h:cross-product condition}, and \ref{hyp:CR-block}.  Let us briefly
illustrate the  Alternating  Minimizing Movement approach to
block-structured systems. As in Section\ \ref{s:AltMinMov}, we construct
approximate solutions by solving the time incremental minimization schemes for
the subdifferential inclusions \eqref{left-Cauchy-block} and
\eqref{right-Cauchy-block}. This results in the \emph{alternating minimization
scheme} \eqref{BLOCK-min-schemes-i} below.

More precisely, starting from an initial datum $u_0=(y_0,z_0)\in \domE_0 $, we
define the piecewise constant solutions
$\pwc U{\bftau} :[0,T] \to \blockspace$,
$\pwc U{\bftau}=( \pwc Y{\bftau}, \pwc Z{\bftau})$ , by setting
$\pwc Y{\bftau}(0): = y_0$, $\pwc Z{\bftau}(0): = z_0$, and for $t\in (0,T)$
 and $k \in \{1,\ldots, N_\bftau\}$  we define
\begin{subequations}
\label{BLOCK-min-schemes-i}
\begin{align}
\label{BLOCK-min-schemes-1}
\text{for } t \in  \lint{k}{\bftau}\ \:{:}\quad  & 
\pwc Y{\bftau}(t) : = Y_k, \quad \pwc Z{\bftau}(t) : = Z_{k-1}, 
\\
 &\text{with } Y_k  \in \Argmin_{Y \in  \Spy} \left\{
   \tfrac{\tau_k}2 \,\widetilde{\calR}_{\mathrm{y}} \big(  \tfrac2{\tau_k}
     ( Y{-}  Y_{k-1}  ) \big)  +
   \calE(\thalf, Y, Z_{k-1})  \right\}, 
 \nonumber
 \\
 \label{BLOCK-min-schemes-2}
 \text{for } t \in  \rint{k}{\bftau}{:}\quad  &
\pwc Y{\bftau}(t) : = Y_k, \quad \pwc Z{\bftau}(t) : = Z_{k},
\\
& \text{with } Z_k  \in \Argmin_{Z \in  \Spz} \left\{ \tfrac{\tau_k}2
  \,\widetilde{\calR}_{\mathrm{z}} \big(  \tfrac2{\tau_k} ( Z{-} 
      Z_{k-1} ) \big)  + \calE(\dis t
  {\bftau}k, Y_k, Z) \right\}.
 \nonumber
\end{align}
\end{subequations}
We also introduce the `delayed' piecewise constant interpolant
$\upwc U{\bftau} = (\upwc Y{\bftau}, \upwc Z{\bftau})$ via
\eqref{right-continuous-pwc}, and the piecewise linear interpolant
$\pwl U{\bftau}:[0,T]\to \blockspace$ of the discrete solutions by setting
\begin{equation}
\label{pwl-block}
\pwl U{\bftau}(t): = 
(\pwl Y{\bftau},  \upwc Z{\bftau}) \text{ for  } t \in [t_{\bftau}^{k-1},
\thalf] \ \text{ and } \ \pwl U{\bftau}(t): = 
(\pwc Y{\bftau},  \pwl Z{\bftau})  \text{ for  } t \in [\thalf,
t_{\bftau}^{k}] \,.
\end{equation}
where 
\[
\left\{
\begin{array}{cl}
\pwl Y{\bftau}(t) =  \frac{t{-} t_{\bftau}^{k-1}}{\tau_k/2} \:\pwc Y{\bftau}(t)
+ \frac{ \thalf {-}t}{\tau_k/2} \;\upwc Y{\bftau}(t)   &  \text{for } t \in
[t_{\bftau}^{k-1}, \thalf],   \medskip
 \\
 \pwl Z{\bftau}(t) = \frac{t{-} \thalf}{\tau_k/2} \: \pwc Z{\bftau}(t)+  \frac{ t_{\bftau}^{k}{-}t}{\tau_k/2} \:\upwc Z{\bftau}(t)   &  \text{for } t \in [\thalf, t_{\bftau}^{k}].
 \end{array} 
 \right.
\]
Finally, the variational interpolant $\pwM U{\bftau}$ can be defined by
replacing \eqref{def-var-interpolant} by an alternate minimization scheme, in
analogy with \eqref{BLOCK-min-schemes-i}.

After these preparations, we can state the result corresponding to Theorem
\ref{th:conv-split-step-MM}, in the context of the block system from Section\
\ref{ss:8.1}.  We omit its proof  because  it follows  easily by
 adapting the proof of Theorem \ref{th:conv-split-step-MM} to the 
case with  block structure, in the
same way as  as we will tailor the proof of Theorem
\ref{th:conv-split-step} to provide a proof of Theorem \ref{thm:sec8} 
 in the upcoming Section \ref{ss:proof-81}.

\begin{theorem}[ Alternating Minimizing Movements for block systems]
\label{thm:sec8-MM} \mbox{} \hfill
Under \linebreak[3] Hypotheses \ref{hyp:E-block}, \ref{hyp:diss-block}, \ref{h:cross-product
  condition}, and \ref{hyp:CR-block}, starting from an initial datum
$u_0=(y_0,z_0)\in \domE_0 $, define the curves
$\pwc U{\bftau} = (\pwc Y{\bftau}, \pwc Z{\bftau})$ and
$\pwl U{\bftau} = (\pwl Y{\bftau}, \pwl Z{\bftau})$ as in
\eqref{BLOCK-min-schemes-i} and \eqref{pwl-block}.
 
Then, for any sequence $(\bftau_n)_n$ with $\lim_{n\to\infty} |\bftau_n | =0$
there exist a (non-relabeled) subsequence and a curve
$U=(Y,Z) \in \AC([0,T];\blockspace)$ with $U(0)=u_0$ such that the following
convergences hold as $n\to \infty$:
\begin{subequations}
\label{cvgs-Thm81-MM}
\begin{align}
\label{pointwise-cvg-81-MM}
&
\left.
 \begin{array}{@{}cc}
   \pwc Y{\bftau_n}(t),\, \pwl Y{\bftau_n}(t) \weakto Y(t)  &  \text{in } \Spy
   \smallskip
 \\
    \pwc Z{\bftau_n}(t),\, \pwl Z{\bftau_n}(t)  \weakto Z(t) &  \text{in } \Spz
 \end{array}\right\}
 \quad  \text{for all } t \in [0,T], 
 \\
 & 
 \label{cvg-U'-81-MM}
 \pwl Y{\bftau_n}' \weakto Y'    \text{  in } \rmL^1([0,T]; \Spy) \quad
 \text{and} \quad 
    \pwl Z{\bftau_n}' \weakto Z'    \text{  in } \rmL^1([0,T]; \Spz),
 \\
 &
 \label{cvg-repeated-rates-81-statement-MM}
 \tfrac12  \TT^{(1)}_{\bftau_n}( \pwl Y{\bftau_n}') \weakto Y'    \text{  in }
 \rmL^1([0,T]; \Spy) \quad \text{and} \quad 
  \tfrac12  \TT^{(2)}_{\bftau_n}( \pwl Z{\bftau_n}') \weakto Z'    \text{  in } \rmL^1([0,T]; \Spz),
 \end{align}
\end{subequations}
and there exists a function $\xi =(\eta,\zeta)\in\rmL^1([0,T];\blockspace^*)$
such that the pair $(U,\xi)$ solves the subdifferential system
\eqref{solving-the-equn-81} and fulfills the energy-dissipation balance
\eqref{EDB-effect-81}.
\end{theorem}

\subsection{An application to linearized visco-elasto-plasticity}
\label{su:BlockApplic}
In this section we discuss the applicability of Theorems \ref{thm:sec8} and
\ref{thm:sec8-MM} to a prototypical class of coupled systems, also considered
in \cite[Sec.\,2]{MR21}. These systems include a model combining linearized
viscoelasticity and viscoplasticity, cf.\ Example \ref{ex:sec8} ahead.

Let 
\begin{subequations}
\label{example-S8}
\begin{equation}
\label{Hilbert-S8}
\text{$\Spy$ and $\Spz$ \  be Hilbert spaces}.
\end{equation}
The dissipation potential $\disblo y: \Spy \to [0,\infty)$ is quadratic, while
$\disblo z: \Spz \to [0,\infty)$ consists of a $1$-homogeneous part and of a
contribution with $p$-growth for some $p>1$, namely
\begin{equation}
\label{diss-S8}
\disblo y(\yvel) = \frac12 \langle \mathbb{V}_{\mathrm{y}} \yvel, \yvel\rangle, \qquad \disblo z(\zvel) = \Psi_{1}(\zvel) +  \Psi_{p}(\zvel) 
\end{equation}
where $\mathbb{V}_{\mathrm{y}} : \Spy \to \Spy^*$ is a bounded, linear,
symmetric operator, $ \Psi_{1}: \Spz \to [0,\infty)$ is positively
$1$-homogeneous and $\Psi_{p}(\zvel) = \psi(\|\zvel\|_{\Spz}) $ for some
convex,  increasing  
$ \psi: [0,\infty) \to [0,\infty) $ with $c r^p \lesssim \psi(r) \lesssim C r^p$ 
giving $p$-growth.    The energy
functional is of the form
\begin{equation}
\label{en-S8}
\calE(t,y,z): = \frac12 \pairing{}{\Spy}{\bbA y}{y}  +
\pairing{}{\Spz}{\bbB y}{z} +\frac12
\pairing{}{\Spz}{\bbG z}{z}  -  \pairing{}{\Spy}{f(t)}{y}  -
 \pairing{}{\Spz}{g(t)}{z} ,
\end{equation}
\end{subequations}
where $\bbA : \Spy \to \Spy^* $ and $\bbG: \Spz \to \Spz^*$ are linear,
bounded, and symmetric, $\bbB : \Spy \to \Spz^*$ is linear and bounded 
such that $\binom{\bbA\ \bbB^*}{\bbB\ \bbG}$ is positive definite. Moreover, we
assume that  $(f,g)\in \rmC^1([0,T]; \Spy^*\ti \Spz^*)$ are time-dependent
applied forces.  In this setup, the subdifferential inclusion
\eqref{subdiff-incl-eff-block} translates into the system
\begin{subequations}
\label{prot-van-visc-syst}
\begin{align}
&
\label{prot-van-visc-u}
\qquad \quad \qquad \ \,   \mathbb V_{\mathrm{y}} y' + 
\bbA y+ \bbB^*z &= f(t) && \text{in } \Spy^*
&& \foraa\  t \in (0,T),
\\
&
\label{prot-van-visc-z}
\partial \Psi_1(z')  + \partial \Psi_p(z')+
\bbB y +  \bbG z &=g(t) && \text{in } \Spz^*
&& \foraa\  t \in (0,T). 
\end{align}
\end{subequations}
A concrete model that falls in this class of systems is provided by the
following example.

\begin{example}
\label{ex:sec8} 
\upshape We consider an elastoplastic body in a bounded Lipschitz domain
$\Omega\subset\R^{d}$. Linearized elastoplasticity is described in terms of the
displacement $y:\Omega\to\R^{d}$, with $y\in \Spy = \rmH_{0}^{1}(\Omega)$, and
the symmetric, trace-free plastic strain tensor
$z:\Omega\to\R_{\mathrm{dev}}^{d\times d}:=\left\{
  z\in\R_{\mathrm{sym}}^{d\times d}:\mathrm{tr}(z)=0\right\} .$ Let
$\mathbf{Z}=\rmL^{2}(\Omega,\R_{\mathrm{dev}}^{d\ti d})$. The energy functional
$\cE:[0,T]\ti \Spy \ti \Spz\to\R$ is defined by
\[
\cE(t,y,z)=\int_{\Omega}\left\{   \frac{1}{2}(e(y){-}z):\mathbb{C}(e(y){-}z) 
 +\frac{1}{2}z:\mathbb{H}z\right\} \!\dd x-\langle f(t),y\rangle 
\]
where $e(y)  = \frac12(\nabla u{+}\nabla u^\top)$ is the linearized
symmetric strain tensor,
$\mathbb{C} \in \mathrm{Lin}(\R_\mathrm{sym}^{d\ti d}) $ and
$\mathbb{H}\in \mathrm{Lin}(\R_\mathrm{dev}^{d\ti d})$ are the positive
definite and symmetric elasticity and hardening tensors, respectively, and
$f: [0,T]\to \rmH^{-1}(\Omega;\R^d)$ a time-dependent volume loading.  The
dissipation potentials are
\[
\disblo y(y')=\int_{\Omega}\frac{1}{2}e(y') : \mathbb{D}e(y')\dd x, \qquad
\disblo z(z')=\int_{\Omega}\sigma_{\mathrm{yield}}|z'|+ \tfrac{\varrho}2 |z'|^2  \dd x
\] 
with $\mathbb{D} \in \mathrm{Lin}(\R^{d\times d}_{\mathrm{sym}})$ the positive
definite viscoelasticity tensor, $\sigma_{\mathrm{yield}}>0$ the yield stress
and $\varrho>0$ a positive coefficient. System \eqref{prot-van-visc-syst}
rephrases as
\[
\begin{aligned}
  -\mathrm{div}\big(\bbD e(y') + \bbC (e(y){-}z)\big) \ \ \qquad
  & = f(t) && \text{ in } \Omega \ti (0,T), 
  \\
  \sigma_\mathrm{yield} \mathrm{Sign} (z') + \varrho  z' + \mathrm{dev}\big(\bbC
  (z{-}e(y))\big) + \mathbb{H} z & \ni\ \ 0 && \text{ in } \Omega \ti (0,T).
\end{aligned}
\]
where 
$\mathrm{dev}\,A = A - \frac1d (\mathrm{tr}\!\; A)\,I$ is the deviatoric part of a tensor $A\in \R^{d\times d}$. 
\end{example}
\par
It is very easy to check that  the energy functional $\calE$ from \eqref{en-S8} satisfies Hypotheses \ref{hyp:E-block} and \ref{h:cross-product condition}. 
Likewise, it is immediate to check that the dissipation potentials $\disblo y$ and $\disblo z$ in \eqref{diss-S8}
both 
comply with condition $\mathbf{<R>}$. Furthermore,  each of them  fulfills an  individual  Quantitative Young Estimate thanks to Lemma \ref{l:2.2}. 
It remains to discuss the validity of the chain-rule Hypothesis \ref{hyp:CR-block}. For this, let us consider a  curve $u = (y,z) \in \AC([0,T];\Spy\ti\Spz)$ with 
$\sup_{t \in [0,T]} |\calE(t,y(t)),z(t))|<\infty$
 such that
\eqref{weaker-estimate} holds, i.e.\ 
\[
\int_0^T \| \bbA y(t){+} \bbB^*z(t) {-} f(t)\|_{\Spy^*} \| y'(t)\|_{\Spy} \dd t + \int_0^T \|\bbB y(t) {+}  \bbG z(t) {-}g(t) \|_{\Spz^*} \| z'(t)\|_{\Spy} \dd t <\infty\,.
\]
Now, since $(y,z) \in \rmL^\infty ([0,T];\Spy \ti \Spz)$ we readily infer that
$\bbA y \in \rmL^\infty ([0,T];\Spy^*)$,
$\bbB^* z \in \rmL^\infty ([0,T];\Spy^*)$,
$\bbB y \in \rmL^\infty ([0,T];\Spz^*)$, and
$\bbG z \in \rmL^\infty ([0,T];\Spz^*)$. Therefore, the individual
contributions to $\calE$ from \eqref{en-S8}, evaluated along the curve $u$, are
absolutely continuous, and for them the following chain rules hold:
\[
\left\{
\begin{array}{@{}cll}
\frac{\dd}{\dd t} \left( \tfrac12 \pairing{}{\Spy}{\bbA y(t)}{y(t)} \right) &=  \pairing{}{\Spy}{\bbA y(t)}{y'(t)}, 
\smallskip
\\
\frac{\dd}{\dd t} \left( \pairing{}{\Spz}{\bbB y(t)}{z(t)} \right)  &= \pairing{}{\Spy}{\bbB^* z(t)}{y'(t)} +\pairing{}{\Spz}{\bbB y(t)}{z'(t)},
\smallskip
 \\
\frac{\dd}{\dd t} \left( \tfrac12\pairing{}{\Spz}{\bbG z(t)}{z(t)} \right) &=  \pairing{}{\Spz}{\bbG z(t)}{z'(t)}\,.
\end{array}
\right.
\]
From that we immediately conclude the chain rule from Hypothesis
\ref{hyp:CR-block}.
	
All in all, we have proved that the  block-structured system
$(\Spy \ti \Spz, \calE, \disblo y{\oplus} \disblo z)$  from
\eqref{example-S8} complies with Hypotheses \ref{hyp:E-block},
\ref{hyp:diss-block}, \ref{h:cross-product condition}, and
\ref{hyp:CR-block}. Thus, Theorems \ref{thm:sec8} and \ref{thm:sec8-MM} are
applicable.

\subsection{Proof of Theorem \ref{thm:sec8}}
\label{ss:proof-81}
We split the argument in the following steps.\medskip

\noindent\STEP{1. A priori estimates:}
As in the proof of Theorem \ref{th:conv-split-step}, the starting point for our analysis is the approximate energy-dissipation balance 
\begin{equation}
\label{EDB-approx-81}
\begin{aligned}
& 
\calE(t,\pws Y\bftau(t), \pws Z\bftau(t))    + \mathcal{D}_{\bftau}^{\mathrm{rate}} ([s,t])  + \mathcal{D}_{\bftau}^{\mathrm{slope}} ([s,t])  
\\
 & 
 = \calE(s,\pws Y\bftau(s), \pws Z\bftau(s)) +\int_s^t \partial_t \calE(r,\pws
 Y\bftau(r), \pws Z\bftau(r))  \dd r 
\end{aligned}
\end{equation}
along all subintervals $[s,t]\subset [0,T]$. With $\cR_j$ given by
\eqref{dissip-potent-Rj} and $\tdis j=2\cR_j(\frac12\,\cdot\,)$, 
  the rate and slope terms from \eqref{rate+slope-terms} now read 
\begin{subequations}
\label{rate+slope-terms-81}
\begin{align}
&
\label{rate-term-81}
\begin{aligned}
\mathcal{D}_{\bftau}^{\mathrm{rate}} ([0,T])  & := 
   \int_0^T   
  \left\{
\chi_{\bftau} (r) \,  \tdis\rmy   \left(\tfrac12 \pws Y{\bftau}'(r)\right) + 
 (1{-}\chi_{\bftau}(r) )\,  \tdis\rmz  \left(\tfrac12 
 \pws Z{\bftau}'(r)\right) \right\} \dd r 
\\
 &  \stackrel{\eqref{rephrasing-via-repet-a}}{=} 
  \int_0^T  \left\{  \disblo y(\tfrac12  \TT^{(1)}_{\bftau} \pws
   Y{\bftau}'(r)) +   \disblo z(\tfrac12  \TT^{(2)}_{\bftau}\pws
   Z{\bftau}'(r)) \right\} \dd r, 
\end{aligned}
\\
&
\label{slope-term-81}
\begin{aligned}
\mathcal{D}_{\bftau}^{\mathrm{slope}}  ([0,T]) & : =    \int_0^T 
\left\{ \chi_{\bftau}(r) \, \tdis\rmy^*   ({-} \pws \eta{\bftau} (r))
  + (1{-}\chi_{\bftau}(r) )\,  \tdis\rmz^*  ({-}\pws\zeta{\bftau} (r))
\right\}  \dd r  
\\ & 
\stackrel{\eqref{rephrasing-via-repet-b}}{=}  \int_0^T   \left\{
  \disblo y^*({-} \TT^{(1)}_{\bftau} \pws \eta{\bftau}(r)) +  \disblo
  z^*({-}\TT^{(2)}_{\bftau} \pws \zeta{\bftau} (r))  \right\} \dd r\,.
 \end{aligned}
\end{align}
\end{subequations}
%
Then, we can mimic the arguments from from Proposition \ref{prop:aprio} and
Corollary \ref{cor:new-aprio} and derive the analogues of the a priori
estimates therein.  \medskip
\par\noindent
\STEP{2. Compactness:}
We may prove the analogue of 
Corollary \ref{cor:compactness}. Namely, there exists $U \in \AC([0,T];\blockspace)$,
such that, up to a subsequence, 
\begin{equation}
\label{cvgs-U-81}
 \pws U{\bftau_n}(t) \weakto U(t) \text{ in } \blockspace \text{ for all } t \in [0,T], \qquad \pws U{\bftau_n}' \weakto U' \text{ in }  \rmL^1([0,T]; \blockspace),
\end{equation}
whence
convergences \eqref{pointwise-cvg-81}, \eqref{cvg-U'-81}. Furthermore, there exist $V$ and $Z$ such that
\begin{equation}
  \label{cvg-repeated-rates-81}
   \left\{
     \begin{array}{ll}
       \displaystyle \tfrac12  \TT^{(1)}_{\bftau_n}( \pws Y{\bftau_n}') \weakto
       V  &  \text{  in } \rmL^1([0,T]; \Spy), 
       \medskip
       \\
       \displaystyle  \tfrac12  \TT^{(2)}_{\bftau_n}( \pws Z{\bftau_n}')
       \weakto W  &  \text{  in } \rmL^1([0,T]; \Spz)\,. 
 \end{array}
\right.
\end{equation}
In order to identify $V$ and $W$ we observe that, since the $z$ (respectively,
the $y$) variable is frozen in the Cauchy problem \eqref{left-Cauchy-block}
(\eqref{right-Cauchy-block}, resp.), we have that
\[
  \TT^{(1)}_{\bftau} ( \pws U{\bftau_n}' ) = \left(\begin{array}{cc}
      \TT^{(1)}_{\bftau} ( \pws Y{\bftau_n}') \\ 0 \end{array} \right) \quad
  \text{and} \quad \TT^{(2)}_{\bftau} (\pws U{\bftau_n}' ) =
  \left(\begin{array}{cc}0 \\
      \TT^{(2)}_{\bftau} ( \pws Z{\bftau_n}') \end{array}\right)\,.
\]
In turn, \eqref{cvgs-U-81} and Lemma \ref{lem:repetition-operator} yield that
\[
\frac12 \left( \TT^{(1)}_{\bftau} (\pws U{\bftau_n}'){+}  \TT^{(2)}_{\bftau}
  (\pws U{\bftau_n}')\right)  \weakto U'\,. 
\]
Therefore, by \eqref{cvg-repeated-rates-81} we find that
$\binom{ Y'}{ Z'} =U'= \binom{ V}{ 0} +\binom0W$, whence
\eqref{cvg-repeated-rates-81-statement} holds true.

Finally, there exist $\eta \in \rmL^1([0,T];\Spy^*)$
and  $\zeta \in \rmL^1([0,T];\Spz^*)$ such that, as $n\to \infty$, 
\begin{equation}
\label{cvgs-forces-81}
 \TT^{(1)}_{\bftau_n} \pws \eta{\bftau_n} \weakto \eta \ \text{ in }
 \rmL^1([0,T];\Spy^*) \quad \text{and} \quad 
  \TT^{(2)}_{\bftau_n} \pws \zeta{\bftau_n} \weakto \zeta  \ \text{ in }
  \rmL^1([0,T];\Spz^*). 
\end{equation}
We now adapt the Young measure argument carried out in Section \ref{ss:slope}:  up to a  (non-relabeled) subsequence, 
the sequences $( \TT^{(1)}_{\bftau_n} \pws \eta{\bftau_n} )_n$
 and  $( \TT^{(2)}_{\bftau_n} \pws \zeta{\bftau_n} )_n$ admit two limiting Young measures $(\mu_t)_{t\in (0,T)}$ and $(\nu_t)_{t\in (0,T)}$,
  with $\mu_t \in \mathrm{Prob}(\Spy^*)$
  and $\nu_t \in  \mathrm{Prob}(\Spz^*) $  for a.a.\ $t\in (0,T)$, such that 
 \begin{subequations}
 \label{props-YM-81}
 \begin{equation}
\label{props-YM-1-81}
\begin{aligned}
&\mathrm{supp}(\mu_t) \subset \Ls{\Spy^*}{\{  \TT^{(1)}_{\bftau_n} \pws
  \eta{\bftau_n}(t) \}_n} \text{ and }\\
& \mathrm{supp}(\nu_t) \subset \Ls{\Spz^*}{\{  \TT^{(2)}_{\bftau_n} \pws
  \zeta{\bftau_n}(t) \}_n} \qquad \foraa\  t \in (0,T) ,
\end{aligned}
\end{equation}
(see \eqref{props-YM-1} for the definition of the limsup of sets), and
the weak limits $\eta$ and $\zeta$ read
\begin{equation}
\label{props-YM-2-81}
\eta(t)   = \int_{\Spy^*} \tilde\eta \dd \mu_t(\tilde\eta), \qquad \zeta(t)   =
\int_{\Spz^*}\tilde\zeta \dd \nu_t(\tilde\zeta) \quad \foraa\  t \in (0,T)\,. 
\end{equation}
\end{subequations}
Now, arguing in the same way as in Section \ref{ss:slope} and exploiting the
closedness property \eqref{h:2.2-partial} we find $\foraa\ t \in (0,T)$ that 
\[
  \Ls{\Spy^*}{\{ \TT^{(1)}_{\bftau_n} \pws \eta{\bftau_n}(t) \}_n } \subset
  \blosub y (t,Y(t),Z(t)) \text{ and } \Ls{\Spz^*}{\{ \TT^{(2)}_{\bftau_n} \pws
    \zeta{\bftau_n}(t) \}_n } \subset \blosub z (t,Y(t),Z(t)) \,.
\]
Since the the subdifferentials are convex, we conclude with
\eqref{props-YM-2-81} that
\[
\eta(t)\in \blosub y (t,Y(t),Z(t)) \text{ \ and \ } 
\zeta(t) \in  \blosub z (t,Y(t),Z(t)) \quad \foraa\ t \in (0,T)\,.\medskip
\]
 
\noindent \STEP{3. Limit passage in \eqref{EDB-approx-81}:} We take the limit
as $n\to\infty$ in \eqref{EDB-approx-81} on the interval  $[0,T]$: 
thanks to \eqref{pointwise-cvg-81} we have 
$\liminf_{n\to\infty} \calE(T,\pws Y{\bftau_n}(T), \pws Z{\bftau_n}(T)) \geq
\calE(T,Y(T),Z(T)) $,   while due to
\eqref{cvg-repeated-rates-81-statement} and \eqref{cvgs-forces-81} we have
\begin{align}
  & \nonumber
  \begin{aligned}
    \liminf_{n\to\infty} \mathcal{D}_{\bftau_n}^{\mathrm{rate}}  ([0,T]) &
    \geq \liminf_{n\to\infty} \int_0^T \big\{\disblo y(\tfrac12 \TT^{(1)}_{\bftau_n}
    \pws Y{\bftau_n}'(r)) +  \disblo
    z(\tfrac12 \TT^{(2)}_{\bftau_n}\pws Z{\bftau_n}'(r)) \big\} \dd r
    \\
    & \geq \int_0^T  \big\{\disblo y(Y'(r)) + \disblo z(Z'(r))  \big\} \dd r\,,
\\
  \liminf_{n\to\infty} \mathcal{D}_{\bftau_n}^{\mathrm{slope}}  ([0,T])  
&  \geq   \liminf_{n\to\infty} \int_0^T   \big\{  \disblo y^*({-}
\TT^{(1)}_{\bftau_n} \pws \eta{\bftau_n}(r))  +   \disblo
z^*({-}\TT^{(2)}_{\bftau_n} \pws \zeta{\bftau_n} (r)) \big\}  \dd r  
 \\ & 
 \geq  \int_0^T \big\{ \disblo y^*({-}\eta(r)) + \disblo z^*({-}\zeta(r)) \big\} \dd r\,.
 \end{aligned}
 \end{align}
 Relying on convergences \eqref{pointwise-cvg-81} we likewise take the limit as
 $n\to\infty$ in the power terms of the right-hand side of
 \eqref{EDB-approx-81}. All in all, we conclude that the curve $U=(Y,Z)$ and
 the function $\xi=(\eta,\zeta)$ fulfill the energy-dissipation inequality
\begin{align*}
&
\calE(T,Y(T),Z(T))  + \int_0^T \!\!\big\{ \disblo y (Y'(r)){+}  \disblo z (Z'(r))
  + \disblo y^* ({-}\eta(r)){+}  \disblo z^* ({-}\zeta(r)) \big\} \dd r 
\\
& \leq 
\calE(0,Y(0),Z(0))  + \int_0^T  \partial_t \calE(r,Y(r),Z(r)) \dd r \,.\medskip
\end{align*}

\noindent \STEP{4. Energy-dissipation balance and conclusion of the proof:}
From the above inequality it immediately follows that the pair $(U,\xi)$
complies with condition \eqref{first-part-CR-block}. Hence, by Hypothesis
\ref{hyp:CR-block} we conclude that $t\mapsto \calE(t,U(t))$ is absolutely
continuous and the chain rule \eqref{eq:48strong} holds for
$(U,\xi)$. Therefore, Proposition \ref{p:EDP} allows us to conclude that
$(U,\xi)$ solves the subdifferential inclusion \eqref{subdiff-incl-eff-block}
(which rephrases as \eqref{solving-the-equn-81}), and that it fulfills the
energy-dissipation balance \eqref{EDB-effect-81}  for all subintervals
$[s,t]\subset [0,T]$.  

As in the case of Theorem \ref{cvgs-Thm43}, the enhanced
 convergences  \eqref{enh-cvg-81}  are a 
 by-product of the argument for passing to the limit in 
  \eqref{EDB-approx-81}; again, we refer 
 to the proof of 
\cite[Thm.\,4.4]{MRS2013} or \cite[Thm.\,3.11]{MRS13} for all details.
\par
This finishes the proof  of Theorem \ref{thm:sec8}.  
\QED

\appendix 

\section{More  on the Quantitative Young Estimate}
\label{s:QYE}
We aim to gain further insight into the connections between the QYE for the
dissipation potentials  $\cR_1$ and $\cR_2$,  (for which we will always
assume the validity of condition $\mathbf{<R>}$, cf.\ Hypothesis
\ref{hyp:dis}), and the validity of the same property for $\calR_\eff$. In the
particular case in which the norms $\norm{\cdot}1$ and $\norm{\cdot}2$ are
indeed equivalent (i.e., $\norms{\cdot}2$ controls $\norms{\cdot}1$), we have
the following result.

\begin{lemma}
 Suppose that $\Spx{} := \Spx 1= \Spx 2 $ and that 
\begin{equation}
\label{equivalent-norms}
\exists\, \bbC_{\mathrm{N}},\, \bbC_{\mathrm{N}}^*>0 \ \ \forall\, (v,\xi) \in \Spx{}\ti \Spx{}^* \, : \qquad 
\left\{
\begin{array}{ll}
\norm{v}1 \leq   \bbC_{\mathrm{N}},  \norm{v}2,
\\
\norms{\xi}1 \leq  \bbC_{\mathrm{N}}^*  \norms{\xi}2.
\end{array}
\right.
\end{equation}
Let $\calR_j$ satisfy the $\mathbf{<QYE>}$ \eqref{eq:QuYouEst} with constants
$c_j,C_j>0$. Then, also $\cR_\eff$ satisfies estimate \eqref{eq:QuYouEst}, with
respect to the norms $\norm{\cdot}1$ and $\norms{\cdot}1$, with the constants
$c_\eff = \min\left\{c_1, \frac{c_2}{ \bbC_{\mathrm{N}}^* \bbC_{\mathrm{N}}^*}
\right\} $ and $C_\eff = C_1+C_2$.
\end{lemma}
\begin{proof}
For any  $v\in \Spx{}$  and any $\eps>0$ there are $v_1,v_2\in \Spx{}$ with
$v_1+v_2=v$ such that $\cR_\eff(v)\geq \cR_1(v_1) +\cR_2(v_2)-\eps$. Combining
$\mathbf{<QYE>}$ for $\calR_j$ with 
\eqref{equivalent-norms} yields
\begin{align*}
\cR_\eff(v)+\cR_\eff^*(\xi)& \geq\cR_1(v_1)+\cR_2(v_2)-\eps +\cR_1^*(\xi)+\cR_2^*(\xi)\\
&\geq  c_1 \norm{v_1}1 \norms{\xi}1 +c_2 \norm{v_2}2 \norms{\xi}2 -C_1 -C_2 -\eps \\
& \geq  c_1 \norm{v_1}1 \norms{\xi}1 +\frac{c_2}{ \bbC_{\mathrm{N}}
  \bbC_{\mathrm{N}}^*}  \norm{v_2}1 \norms{\xi}1 -C_1 -C_2 -\eps \\ 
&\geq  \min\left\{c_1, \frac{c_2}{ \bbC_{\mathrm{N}} \bbC_{\mathrm{N}}^*}
\right\} \left(\norm{v_1}1+\norm{v_2}1  \right) \norms{\xi}1  -C_1 -C_2 -\eps
\\  
&\geq  \min\left\{c_1, \frac{c_2}{ \bbC_{\mathrm{N}} \bbC_{\mathrm{N}}^*}
\right\}  \norm{v}1 \norms{\xi}1-C_1 -C_2-\eps\,. 
\end{align*}
Since $\eps>0$ is arbitrary the claim follows. 
\end{proof}

In the general case in which we only have $\Spx 2\subset \Spx 1$ (densely and)
continuously (cf.\ \eqref{norm-control}),  the next result
provides some growth and coercivity conditions on $\calR_1$ and $\calR_2 $
under which the QYE for $\calR_1$ guarantees that for $\calR_\eff$.

\begin{lemma}
\label{l:QYE-more general}
 Assume there exist positive constants $C_1,\, c_2, \,C_2$ and $p,q\in
(1,\infty)$ with $q<p$ such that
\begin{equation}
\label{overline-growth}
\forall\, v\in \Spx1: \quad \calR_1(v) \leq
C_1\|v\|_1^q+C_1 \text{ and } \calR_2(v) \geq
c_2 \|v\|_1^p -C_2. 
\end{equation}
 If additionally  $\calR_1$ satisfies $\mathbf{<QYE>}$,  then  also
$\calR_\eff$ complies with $\mathbf{<QYE>}$.  
\end{lemma}
\begin{proof}
 Consider  $v_1,v_2\in \Spx 1$ with $v_1+v_2=v$ and
$\cR_1(v_1)+\cR_2(v_2)= \cR_\eff(v)$. Then, using \eqref{overline-growth}
gives 
\begin{align*}
c_2\norm{v_2}1^p -C_2  &\leq \calR_2(v_2)  \leq   \calR_1(v_1)+ \calR_2(v_2)
= \cR_\eff(v)\\
&\overset{(1)}\leq  \cR_1(v{-}0)+ \calR_2(0)  \leq 
C_1 \norm{v}1^q +C_1,
\end{align*}
where $\overset{(1)}\leq$ uses the definition of $\cR_\eff$ as an
inf-convolution. Thus in the optimal decomposition $v=v_1+v_2$ we have 
$\norm{v_2}1 \leq C \norm{v}1^{q/p} +C$, which implies
\begin{equation}
  \label{eq:v1.versus.v}
  \norm{v_1}1 \geq \norm{v}1 - \norm{v_2}1 \geq \norm{v}1 - C\norm{v}1^{q/p}
  -C\geq \frac12\:\norm{v}1 - C_*,
\end{equation}
where we used $q \lneqq p$. 

With this we derive the lower bound
\begin{align*}
\calR_\eff(v)+\calR_\eff^*(\xi) & = \cR_1(v_1)+\cR_2(v_2) + \cR_1^*(\xi)+
\cR_2^*(\xi) \geq  \cR_1(v_1) + \cR_1^*(\xi) 
\\
&\overset{\text{QYE}}\geq c_1^\text{QYE}\norm{v_1}1 \:\norm{\xi}* - C^\text{QYE} 
\overset{\text{\eqref{eq:v1.versus.v}}}\geq 
 c_1^\text{QYE} \big( \frac12 \norm{v}1{-} C_*\big) \:\norm{\xi}*  - C^\text{QYE} .
\end{align*}

This is almost the desired result, except for the linear term
$-c_1^\text{QYE}C_*\norm{\xi}*$ on the right-hand side. However, since
\eqref{overline-growth} provides an upper bound on $\cR_1$, we have a lower
bound on $\cR_1^*$ and hence of $\cR_\eff^*$, viz.\ 
\[
\cR_\eff^*(\xi)\geq \cR_1^*(\xi) \overset{\text{\eqref{overline-growth}}}\geq  
\ol c \norm{\xi}*^{q^*} -\ol C \geq \eps \norm{\xi}* -C_\eps \quad \text{ for
  all } \eps>0.
\]
Choosing $\eps>0$ sufficiently small, the linear term can be absorbed into the
left-hand side, and the QYE for $\cR_\eff$ on $\Spx1$ is established. 
\end{proof}

\section{ Proof of Lemma \ref{lem:repetition-operator}}
\label{s:appB}
It is sufficient to prove only items (3) and (4) of the statement; for
convenience, we will show them for a sequence $(g_{\bftau_n})_n$, with
$|\bftau_n| \to 0$ as $n\to\infty$.

\paragraph{\textbf{Ad (3):}} To fix ideas, we will show the statement for the
operators $\TT^{(1)}_{\bftau_n}$.  Let $g_{\bftau_n} \to g$ in
$ \rmC^0([0,T]; \Spgxw)$, namely
$\lim_{n\to\infty}\sup_{t\in [0,T]} d_{\mathrm{weak}}(g_{\bftau_n}(t), g(t))
=0$ (where $d_{\mathrm{weak}}$ is the distance inducing the weak topology on a
closed bounded subset of $\Spgx$). In order to show that
$\TT^{(1)}_{\bftau_n} g_{\bftau_n} \to g$ in $ \rmC^0([0,T]; \Spgxw)$, we
observe that
\[
\begin{aligned}
  & \sup_{t \in [0,T]} d_{\mathrm{weak}}(\TT^{(1)}_{\bftau_n} g_{\bftau_n}(t),
  g(t)) = \max \left\{ S_{\mathrm{left},\bftau_n}, S_{\mathrm{right},\bftau_n}
  \right\}
  \\
  & \qquad \text{with } \left\{
    \begin{array}{lll}
      S_{\mathrm{left},\bftau_n} = \sup_{t \in  \cup \lint{k}{\bftau_n}}
      d_{\mathrm{weak}}(g_{\bftau_n}(t), g(t)), 
      \\
      S_{\mathrm{right},\bftau_n} = \sup_{t \in  \cup \rint{k}{\bftau_n}}
      d_{\mathrm{weak}} \!\left(g_{\bftau_n}\big(t{-}\tfrac{\tau_n(t)}2\big), g(t)\right).
\end{array}
\right.
\end{aligned}
\]
Now, we clearly have $S_{\mathrm{left},\bftau_n}  \leq \sup_{t\in [0,T]}
d_{\mathrm{weak}}(g_{\bftau_n}(t), g(t))  \rightarrow 0$. In turn, 
\[
\begin{aligned}
S_{\mathrm{right},\bftau_n}  & =  \sup_{t \in  \cup \lint{k}{\bftau_n}}
d_{\mathrm{weak}} \!\left(g_{\bftau_n}(t), g\big(t{+}\tfrac{\tau_n(t)}2\big)\right) 
\\
&
\leq \sup_{t \in  \cup \lint{k}{\bftau_n}}
d_{\mathrm{weak}}\left(g_{\bftau_n}\left(t\right), g(t)\right) + \sup_{t \in
  \cup \lint{k}{\bftau_n}} d_{\mathrm{weak}}\! \left( g(t),
  g\big(t{+}\tfrac{\tau_n(t)}2\big)\right)  \longrightarrow 0 
\end{aligned}
\]
thanks to the fact that $g \in \rmC^0([0,T]; \Spgxw)$. This concludes the proof
of Claim (3).

\paragraph{\textbf{Ad (4):}} Let now
$ g_{\bftau_n} \weakto g \text{ in } \rmL^1([0,T]; \Spgx)$: we aim to show that
\begin{equation}
\label{aim-ITEM-4}
 \frac12 \left( \TT^{(1)}_{\bftau_n} {+}  \TT^{(2)}_{\bftau_n} \right) (g_{\bftau_n}) \weakto g \quad  \text{ in }  \rmL^1([0,T]; \Spgx)\,.
\end{equation}
In what follows, we will use the short-hand
$ \TT^{(0)}_{\bftau_n}:= \frac12\big(\TT^{(1)}_{\bftau_n} {+}\TT^{(2)}_{\bftau_n} \big)$.

First, we observe that, by a characterization of weak compactness in
$\rmL^1([0,T];\Spgx)$ known as the Dunford-Pettis Theorem (see \cite[\S
IV.2, p.\,101]{Diestel-Uhl} for a version in Bochner spaces), the sequence
$(g_{\bftau_n})_n$ is uniformly integrable, i.e.\ there exists a convex
superlinear function $\Phi: [0,\infty) \to [0,\infty) $ such that
\[
\sup_n \int_0^T \Phi( \| g_{\bftau_n}(r)\|) \dd r <\infty\,.
\]
We will now show that $(\TT^{(0)}_{\bftau_n}g_{\bftau_n})_{n\in \N}$ is also uniformly integrable. Indeed, 
by the convexity of $\Phi$ we have 
\[
\begin{aligned}
\sup_n &\int_0^T \!\!\Phi( \| \TT^{(0)}_{\bftau_n}g_{\bftau_n}(r)\|) \dd r     = \sup_n 
\left( \int_0^T \!\!\tfrac12 \Phi (\|  \TT^{(1)}_{\bftau_n}g_{\bftau_n}(r)\| ) \dd r
 + \int_0^T \!\!\tfrac12 \Phi (\|  \TT^{(2)}_{\bftau_n}g_{\bftau_n}(r)\| ) \dd r
\right) 
\\
 & \stackrel{(1)}{=}  
 \sup_n  \int_0^T \chi_{\bftau_n}^{(1)}(r) \Phi(\| g_{\bftau_n}(r)\|) \dd r +  \sup_n  \int_0^T \chi_{\bftau_n}^{(2)}(r) \Phi(\| g_{\bftau_n}(r)\|) \dd r \,, 
\end{aligned}
\]
where $\overset{(1)}=$ follows from direct calculations, cf.\
\eqref{useful-identities}.  Therefore,
$(\TT^{(0)}_{\bftau_n}g_{\bftau_n})_{n\in \N}$ is uniformly integrable as
well. Hence, again by the Dunford-Pettis criterion, up to a (non-relabeled)
subsequence $(\TT^{(0)}_{\bftau_n}g_{\bftau_n})_{n\in \N}$ weakly converges in
$\rmL^1([0,T];\Spgx)$ to some limit $\widetilde g$. In order to show that
$\widetilde g = g$, it will be sufficient to prove that
\[
 \int_{0}^{T} \langle {\phi},{\TT^{(0)}_{\bftau_n}g_{\bftau_n}{-}g}
 \rangle_{\Spgx}\,\d t \longrightarrow 0 \qquad \text{for all } \phi \in
 \rmC^0([0,T];\Spgx^*). 
 \]
 In fact, since $\TT^{(0)}_{\bftau_n} g\to g$ in $\rmL^1([0,T];\Spgx{})$, it
 will be sufficient to show that
\[
 \int_{0}^{T}\langle
 {\phi},{\TT^{(0)}_{\bftau_n}g_{\bftau_n}{-}\TT^{(0)}_{\bftau_n}g}
 \rangle_{\Spgx}\,\d t \longrightarrow 0 \qquad \text{for all } \phi \in
 \rmC^0([0,T];\Spgx^*). 
\]
With this aim, we compute the adjoint of the operator $\TT^{(0)}_{\bftau}$. 
For $f\in\rmL^1([0,T],\Spx{})$ and $\phi\in\rmC^0([0,T],\Spx{}^*)$ we have 
\begin{align*}
  \int_0^T \langle {\phi},{\TT^{(0)}_\bftau f} \rangle_{\Spgx}  \dd t 
  &=\sum_{k=1}^{N_{\bftau}} \int_{\lint {k}{\bftau}}  \langle
  {\phi},{\TT^{(0)}_\bftau f}  \rangle_{\Spgx}   \dd t  +
  \sum_{k=1}^{N_{\bftau}}  \int_{\rint{k}{\bftau}}  \langle
  {\phi},{\TT^{(0)}_\bftau f}   \rangle_{\Spgx}  \dd t \\ 
  &= \frac 1 2 \Big(  \sum_{k=1}^{N_{\bftau}} \int_{\lint {k}{\bftau}}  \langle
  {\phi},{f}  \rangle_{\Spgx} \dd t + \int_{\lint
    {k}{\bftau}}\langle{\phi},{f(\cdot   + \tau_k /2)}  \rangle_{\Spgx}  \dd t 
  \\
  & \qquad \qquad +
  \int_{\rint{k}{\bftau}} \langle {\phi},{ f(\cdot-{ \tau_k }/2)}
  \rangle_{\Spgx}   \dd t + \int_{\rint{k}{\bftau}}  \langle {\phi},{f}
  \rangle_{\Spgx} \dd t \Big) \\ 
  &= \frac 1 2  \Big( \sum_k \int_{\lint {k}{\bftau}}\langle{\phi},{f}
  \rangle_{\Spgx} \dd t + \int_{\rint {k}{\bftau}} \langle {\phi(\cdot -{
      \tau_k }/2)},{f}  \rangle_{\Spgx} \dd t  
  \\ 
  & \qquad \qquad +  \int_{\lint{k}{\bftau}}  
  \langle { \phi(\cdot+{ \tau_k }/2)},f   \rangle_{\Spgx} \dd t +
  \int_{\rint{k}{\bftau}}   \langle {\phi},f  \rangle_{\Spgx} \dd t \Big) \\ 
  & =\int_0^T  \langle {\TT^{(0)}_\bftau \phi},f  \rangle_{\Spgx} \dd t.
\end{align*}
Hence, for every  $\phi \in \rmC^0([0,T];\Spgx^*)$ we have
\[
  \int_{0}^{T}\langle
  {\phi,}{\TT^{(0)}_{\bftau_n}g_{\bftau_n}{-}\TT^{(0)}_{\bftau_n}g}
  \rangle_{\Spgx}\,\d t = \int_{0}^{T}\langle
  {\TT^{(0)}_{\bftau_n}\phi},{g_{\bftau_n}{-}g} \rangle_{\Spgx} \,\d t
  \longrightarrow 0 \,,
\]
because $\TT^{(0)}_{\bftau_n}\phi\to \phi$ strongly in
$\rmL^\infty([0,T];\Spgx{}^*)$, since $\phi \in \rmC^0([0,T];\Spgx^*)$ and
$g_{\bftau_n} -g\rightharpoonup 0$ weakly in $\rmL^1([0,T];\Spx{})$. This
concludes the proof of item (4) of the statement.  \QED

{\small

\markboth{References}{References}

\bibliographystyle{./alpha_AMs}
\bibliography{split_bib}

\newcommand{\etalchar}[1]{$^{#1}$}
\providecommand{\bysame}{\leavevmode\hbox to3em{\hrulefill}\thinspace}
\providecommand{\MR}{}
\begin{thebibliography}{11}\itemsep0.1em

\bibitem[AC{\etalchar{*}}17]{ACFS17}
M.~Artina, F.~Cagnetti, M.~Fornasier, and F.~Solombrino: \emph{Linearly
  constrained evolutions of critical points and an application to cohesive
  fractures}. Math. Models Methods Appl. Sci. \textbf{27}:2 (2017) 231--290.

\bibitem[AGS05]{AGS08}
L.~Ambrosio, N.~Gigli, and G.~Savar{\'e}, \emph{Gradient flows in metric spaces
  and in the space of probability measures}, second ed., Lectures in
  Mathematics ETH Z\"urich, Birkh\"auser Verlag, Basel, 2005.

\bibitem[Amb95]{Ambrosio95}
L.~Ambrosio: \emph{Minimizing movements}. Rend. Accad. Naz. Sci. XL Mem. Mat.
  Appl. (5) \textbf{19} (1995) 191--246.

\bibitem[Att84]{Att84VCFO}
H.~Attouch, \emph{Variational convergence for functions and operators},
  Applicable Mathematics Series, Pitman (Advanced Publishing Program), Boston,
  MA, 1984.

\bibitem[AuF09]{Aubin-Frankowska}
J.-P.~Aubin and H.~Frankowska, \emph{Set-valued analysis}, Modern
  Birkh\"{a}user Classics, Birkh\"{a}user Boston, Inc., Boston, MA, 2009,
  Reprint of the 1990 edition.

\bibitem[Bre73]{Brez73OMMS}
H.~Brezis, \emph{Op\'{e}rateurs maximaux monotones et semi-groupes de
  contractions dans les espaces de {H}ilbert}, North-Holland Mathematics
  Studies, No. 5, North-Holland Publishing Co., Amsterdam-London; American
  Elsevier Publishing Co., Inc., New York, 1973.

\bibitem[CaV77]{Castaing-Valadier77}
C.~Castaing and M.~Valadier, \emph{Convex analysis and measurable
  multifunctions}, Lectures Notes in Mathematics, Vol. 580, Springer-Verlag,
  Berlin-New York, 1977.

\bibitem[ClM11]{CleMaa11TPFG}
P.~Cl{\'e}ment and J.~Maas: \emph{A {T}rotter product formula for gradient
  flows in metric spaces}. J. Evol. Equ. \textbf{11}:2 (2011) 405--427.

\bibitem[Col92]{Colli92}
P.~Colli: \emph{On some doubly nonlinear evolution equations in {B}anach
  spaces}. Japan J. Indust. Appl. Math. \textbf{9}:2 (1992) 181--203.

\bibitem[CoV90]{ColliVisintin90}
P.~Colli and A.~Visintin: \emph{On a class of doubly nonlinear evolution
  equations}. Comm. Partial Differential Equations \textbf{15}:5 (1990)
  737--756.

\bibitem[CrP69]{Crandall-Pazy69}
M.~G.~Crandall and A.~Pazy: \emph{Semi-groups of nonlinear contractions and
  dissipative sets}. J. Functional Analysis \textbf{3} (1969) 376--418.

\bibitem[DFM19]{DoFrMi19GSWE}
P.~Dondl, T.~Frenzel, and A.~Mielke: \emph{A gradient system with a wiggly
  energy and relaxed {EDP}-convergence}. ESAIM Control Optim. Calc. Var.
  \textbf{25}:68 (2019) 1--45.

\bibitem[DiU77]{Diestel-Uhl}
J.~Diestel and J.~J.~Uhl, Jr., \emph{Vector measures}, Mathematical Surveys,
  No. 15, American Mathematical Society, Providence, R.I., 1977, With a
  foreword by B. J. Pettis.

\bibitem[IoT79]{IofTih79TEPe}
A.~D.~Ioffe and V.~M.~Tihomirov, \emph{Theory of extremal problems}, Studies in
  Mathematics and its Applications, vol.~6, North-Holland Publishing Co.,
  Amsterdam, 1979, Translated from the Russian by Karol Makowski.

\bibitem[KaM78]{KatMas78TPFNS}
T.~Kato and K.~Masuda: \emph{Trotter's product formula for nonlinear semigroups
  generated by the subdifferentials of convex functionals}. J. Math. Soc. Japan
  \textbf{30}:1 (1978) 155--162.

\bibitem[KnN17]{KneNeg17CAMS}
D.~Knees and M.~Negri: \emph{Convergence of alternate minimization schemes for
  phase field fracture and damage}. M3AS: Math. Models Meth. Appl. Sci.
  \textbf{27}:9 (2017) 1743--1794.

\bibitem[K{\={o}}m67]{Komura67}
Y.~K{\={o}}mura: \emph{Nonlinear semi-groups in {H}ilbert space}. J. Math. Soc.
  Japan \textbf{19} (1967) 493--507.

\bibitem[Mie22]{Miel22RRIS}
A.~Mielke: \emph{Relating a rate-independent system and a gradient system for
  the case of one-homogeneous potentials}. J. Dynam. Diff.\ Eqns. \textbf{34}
  (2022) 3143--3164.

\bibitem[Mie23]{Miel23IAGS}
\bysame: \emph{An introduction to the analysis of gradient systems}. Script of
  a lecture course (2023) 100\:pp., WIAS Preprint 3022, arXiv:2306.05026.

\bibitem[MiR23a]{MR21}
A.~Mielke and R.~Rossi: \emph{{Balanced-Viscosity} solutions to
  infinite-dimensional multi-rate systems}. Arch. Rational Mech. Anal.
  \textbf{247}:53 (2023) 1--100.

\bibitem[MiR23b]{MieRos20?DL}
A.~Mielke and R.~Rossi: \emph{On {De Giorgi}'s lemma for variational
  interpolants}. In preparation (2023) .

\bibitem[MMP21]{MiMoPe21EFED}
A.~Mielke, A.~Montefusco, and M.~A.~Peletier: \emph{Exploring families of
  energy-dissipation landscapes via tilting: three types of {EDP} convergence}.
  Contin. Mech. Thermodyn. \textbf{33}:3 (2021) 611--637.

\bibitem[Mor06]{Mordu-I}
B.~S.~Mordukhovich, \emph{Variational analysis and generalized differentiation.
  {I: Basic theory}}, Grundlehren der mathematischen Wissenschaften, vol. 330,
  Springer-Verlag, Berlin, 2006.

\bibitem[MPS21]{MiPeSt21EDPC}
A.~Mielke, M.~A.~Peletier, and A.~Stephan: \emph{E{DP}-convergence for
  nonlinear fast-slow reaction systems with detailed balance}. Nonlinearity
  \textbf{34}:8 (2021) 5762--5798.

\bibitem[MRS13]{MRS2013}
A.~Mielke, R.~Rossi, and G.~Savar{\'e}: \emph{Nonsmooth analysis of doubly
  nonlinear evolution equations}. Calc. Var. Partial Differential Equations
  \textbf{46}:1-2 (2013) 253--310.

\bibitem[MRS16]{MRS13}
\bysame: \emph{Balanced viscosity ({BV}) solutions to infinite-dimensional
  rate-independent systems}. J. Eur. Math. Soc. (JEMS) \textbf{18}:9 (2016)
  2107--2165.

\bibitem[RMS08]{RMS08}
R.~Rossi, A.~Mielke, and G.~Savar{\'e}: \emph{A metric approach to a class of
  doubly nonlinear evolution equations and applications}. Ann. Sc. Norm. Super.
  Pisa Cl. Sci. (5) \textbf{7}:1 (2008) 97--169.

\bibitem[RoS06]{RossiSavare06}
R.~Rossi and G.~Savar\'{e}: \emph{Gradient flows of non convex functionals in
  {H}ilbert spaces and applications}. ESAIM Control Optim. Calc. Var.
  \textbf{12}:3 (2006) 564--614.

\bibitem[Rou10]{Roub10TRIP}
T.~Roub{\'{\i}}{\v{c}}ek: \emph{Thermodynamics of rate-independent processes in
  viscous solids at small strains}. SIAM J. Math. Anal. \textbf{42}:1 (2010)
  256--297.

\bibitem[RTP15]{RoubThoPana15}
T.~Roub\'{\i}\v{c}ek, M.~Thomas, and C.~G.~Panagiotopoulos: \emph{Stress-driven
  local-solution approach to quasistatic brittle delamination}. Nonlinear Anal.
  Real World Appl. \textbf{22} (2015) 645--663.

\bibitem[SaS04]{Sandier-Serfaty}
E.~Sandier and S.~Serfaty: \emph{Gamma-convergence of gradient flows with
  applications to {G}inzburg-{L}andau}. Comm. Pure Appl. Math. \textbf{57}:12
  (2004) 1627--1672.

\end{thebibliography}
}

\end{document}